\documentclass[12pt]{amsart}

\usepackage{amssymb}
\usepackage{amsfonts}
\usepackage{amsmath}
\usepackage[mathscr]{eucal}

\makeindex

\newtheorem{thm}{{{Theorem}}}[section]
\newtheorem{prop}[thm]{{Proposition}}
\newtheorem{lem}[thm]{{Lemma}}

\newtheorem{cond}[thm]{{Condition}}
\newtheorem{exa}[thm]{{Example}}
\numberwithin{equation}{section}

\def\N{\mathbb{N}}
\def\Z{\mathbb{Z}}
\def\Q{\mathbb{Q}}
\def\R{\mathbb{R}}
\def\C{\mathbb{C}}
\def\A{\mathbb{A}}

\def\GL{{\mathop{\mathrm{GL}}}}
\def\PGL{{\mathop{\mathrm{PGL}}}}
\def\SL{{\mathop{\mathrm{SL}}}}

\def\SO{{\mathop{\mathrm{SO}}}}
\def\U{{\mathop{\textnormal{U}}}}
\def\Sp{{\mathop{\mathrm{Sp}}}}
\def\GSp{{\mathop{\mathrm{GSp}}}}

\def\bK{{\mathbf{K}}}

\def\Re{{\mathop{\mathrm{Re}}}}

\def\Tr{{\mathop{\mathrm{Tr}}}}

\def\Mat{{\mathop{\textnormal{M}}}}

\def\Hom{{\mathop{\mathrm{Hom}}}}
\def\diag{{\mathop{\mathrm{diag}}}}

\def\Ad{{\mathop{\mathrm{Ad}}}} 
\def\vol{{\mathop{\mathrm{vol}}}}
\def\Gal{{\mathop{\mathrm{Gal}}}}
\def\d{{\mathrm{d}}}

\def\bsl{\backslash}
\def\inf{\infty}
\def\Rp{{(\R^\times)^0}}

\def\fo{{\mathfrak o}}
\def\fO{\mathfrak O}
\def\cO{{\mathcal O}}
\def\cS{\mathcal S}
\def\xd{{\mathfrak{x}}}
\def\ds{{\check d}}
\def\para{{\mathfrak Q}}

\def\trep{{\mathbf{1}}}

\def\a{{\mathfrak{a}}}
\def\L{{\mathcal{L}}}
\def\P{{\mathcal{P}}}
\def\F{{\mathcal{F}}}
\def\cU{{\mathcal{U}}}

\def\fg{{\mathfrak{g}}}
\def\fn{{\mathfrak{n}}}
\def\fc{{\mathfrak{c}}}
\def\fC{{\mathfrak{C}}}

\def\tri{{\mathrm{tri}}}
\def\min{{\mathrm{min}}}
\def\sub{{\mathrm{sub}}}
\def\reg{{\mathrm{reg}}}
\def\fin{{\mathrm{fin}}}

\def\ss{{\mathrm{ss}}}
\def\st{{\mathrm{st}}}

\begin{document}

\pagestyle{plain}
\title{On the geometric side of the Arthur trace formula for the symplectic group of rank~2}

\author{Werner Hoffmann}
\address{Fakult\"at f\"ur Mathematik, Universit\"at Bielefeld, 33615 Bielefeld, Germany}
\email{hoffmann@math.uni-bielefeld.de}

\author{Satoshi Wakatsuki}
\address{Faculty of Mathematics and Physics, Institute of Science and Engineering, Kanazawa University, Kakumamachi, Kanazawa, Ishikawa, 920-1192, Japan}
\email{wakatsuk@staff.kanazawa-u.ac.jp}

\begin{abstract}
We study the non-semisimple terms in the geometric side of the Arthur trace formula for the split symplectic similitude group or the split symplectic group of rank $2$ over any algebraic number field.
In particular, we show that the coefficients of unipotent orbital integrals are expressed by the Dedekind zeta function, Hecke $L$-functions, and the Shintani zeta function for the space of binary quadratic forms.
\end{abstract}


\maketitle

\tableofcontents

\section{Introduction}

The Arthur-Selberg trace formula is one of the main tools in the study of automorphic representations of a reductive group $G$ defined over an algebraic number field~$F$ (or rather, of the group $G(\A)$ of the points of~$G$ over the adele ring $\A$ of~$F$). This formula is an identity between two expansions of a distribution $J^T(f)$ on~$G(\A)$, which depends on a test function $f$ and a truncation parameter~$T$. The spectral side of the trace formula, which will not be discussed in this paper, is an expansion in terms of automorphic representations of~$G$ and its Levi subgroups~$M$. The geometric side is, roughly speaking, an expansion in terms of conjugacy classes of elements $\gamma$ of $G(F)$ and its Levi subgroups~\cite{Arthur4}. It is a linear combination of weighted orbital integrals $J_M(\gamma,f)$ with coefficients $a^M(\gamma)$, which in general depend on the choice of a sufficiently large finite set $S$ of places of~$F$.

Those coefficients have been determined only in special cases, viz. for semisimple elements $\gamma$ (see~\cite[Theorem~8.2]{Arthur4}), for $M$ of $F$-rank one (see~\cite{Hoffmann}) and for $M=\GL(3)$ (see~\cite{Flicker}, \cite{Matz}). The general lack of knowledge has not prevented applications to the Langlands correspondence, where the above distributions for two different groups had to be compared. However, this problem constrains the study of the asymptotic distribution of automorphic representations of a single group.

In the current paper we determine the coefficients $a^M(S,\gamma)$ when $G$ is one of the groups $\GSp(2)$ or $\Sp(2)$ over~$F$. The coefficients will be expressed in terms of certain zeta functions. We will encounter the Dedekind zeta function and Hecke $L$-functions familiar from the analogous results for the groups $\GL(2)$ and~$\GL(3)$. The subregular unipotent class is more interesting, because the Shintani zeta function for the space of binary quadratic forms enters the stage. Along the way, we also compute the weight factors of unipotent weighted orbital integrals.
The contribution of semisimple elements to the trace formula is well known, and we leave it aside.

By comparing the coefficients for the two groups considered, we find that the differences can be interpreted by coefficients of stable unipotent orbital integrals of elliptic endoscopic groups of~$\Sp(2)$.
Our explanation agrees with Assem's results on local unipotent orbital integrals of~$\Sp(2)$ (cf.~\cite{Assem}).
This means that our results provide part of the stabilization of the unipotent terms for~$\Sp(2)$.
One more point is that the difference of subregular unipotent terms is a sum of special values of Shintani zeta functions with non-trivial quadratic characters.
From \cite{Saito1} and \cite{Saito2} we will derive a formula expressing such Shintani zeta functions as products of the Dedekind zeta function and special values of Hecke $L$-functions.
The formula looks like a stabilization of Shintani zeta functions and is a generalization of \cite[Theorem 1]{IS}.

The aim of this paper is not just the computation of certain coefficients. Groups of rank two are a test field for a program to rewrite the geometric side of the trace formula in terms of zeta integrals \cite{Hoffmann2}. This approach replaces the invariance argument of~\cite{Arthur3} and removes the restriction on~$S$. It also gives a new interpretation of the weight factors appearing in singular weighted orbital integrals.

Using the methods of \cite{FL} and \cite{Matz}, it should be possible to extend the results to functions~$f$ in an $L^1$-Sobolev space. This would allow matrix coefficients of integrable discrete series as test functions and make Godement's dimension formula a special case of the trace formula (cf. \cite{Godement}).
Godement's formula has been used to obtain explicit formulas for multiplicities of integrable holomorphic discrete series of $\Sp(2,\R)$ (see, e.g., \cite{Wakatsuki}). In the present paper, however, we are content with compactly supported test functions.
Applying our current result to the invariant trace formula and using an argument similar to \cite{Arthur7}, we can perhaps obtain multiplicity formulas for individual discrete series and extend such explicit formulas to non-integrable cases.

\subsection{Basic objects}
\label{basobj}

Let us recall the pertinent definitions. For any connected linear algebraic $F$-group~$G$, we denote by $X(G)_F$ the group of $F$-rational characters of~$G$ and by $A_G$ the $F$-split part of the center of~$G$. We have the vector space $\a_G=\Hom_\Z(X(G)_F,\R)$ and the homomorphism $H_G:G(\A)\to\a_G$ characterized by $\langle H_G(g),\chi\rangle=\log|\chi(g)|$. There is a linear function $\rho_G$ on~$\a_G$ such that the character $\det\circ\Ad$ is mapped to $2\rho_G$ under the natural embedding of $X(G)_F$ into the dual space~$\a_G^*$. The kernel of $H_G$ will be denoted by~$G(\A)^1$. Here, the locally compact $F$-algebra $\A$ may be replaced by any closed subalgebra, e.~g. by $F_S=\prod_{v\in S}F_v$, where $S$ is a finite subset of the set $\Sigma$ of places of~$F$. We have $\A=F_S\A^S$, where $\A^S$ is the restricted product of the completions $F_v$ over all places $v\notin S$. If $S$ equals the set $\Sigma_\inf$ of all infinite places of $F$, we write this decomposition as~$\A=F_\inf\A_\fin$.

Now let $G$ be reductive (in which case $\rho_G=0$). For any parabolic subgroup $P$ of~$G$ with unipotent radical~$N_P$ and Levi component~$M_P$, all defined over~$F$, we have a unitary representation $R_P$ of $G(\A)^1$ on $L^2(N_P(\A)M_P(F)\backslash G(\A)^1)$. It can be integrated to a representation of the Banach algebra $L^1(G(\A)^1)$, and for an element $f$ of the latter, $R_P(f)$ is an integral operator with kernel
\[
K_P(g,h)=\sum_{\gamma\in M_P(F)}\int_{N_P(\A)}f(g^{-1}\gamma nh)\, \d n.
\]
We fix a maximal compact subgroup $\bK=\prod_{v\in\Sigma}\bK_v$ of $G(\A)$ such that $\bK_\fin=\prod_{v<\infty}\bK_v$ is open in $G(\A_\fin)$ and that $G(\A)=P(\A)\bK$ for all parabolic $F$-subgroups~$P$. We extend the maps $H_P$ to~$G(\A)$ by setting $H_P(pk)=H_P(p)$ for $p\in P(\A)$ and $k\in\bK$. For a parabolic subgroup $P'$ containing~$P$, the natural projection $\a_P\to\a_{P'}$ maps $H_P(g)$ to~$H_{P'}(g)$, and we denote its kernel by~$\a_P^{P'}$. We denote by $\hat\tau_P$ the characteristic function of the set of all $X\in\a_P$ such that $\rho_{P'}(X)>0$ for all such~$P'$ different from~$G$.

We also fix a minimal parabolic subgroup $P_0$ and a truncation parameter $T\in\a_{P_0}$. Then we have its projections $T_P\in\a_P$ for the standard parabolic subgroups~$P$ (i.~e. those containing~$P_0$). Arthur's trace distribution is defined for $f\in C_c^\infty(G(\A)^1)$ as
\begin{multline}\label{defJ}
J^T(f)=\int_{G(F)\backslash G(\A)^1 } 
\sum_{P\supset P_0}(-1)^{\dim\a_P^G}
\sum_{\delta\in P(F)\backslash G(F)} \\
K_P(\delta g,\delta g) \, \widehat{\tau}_P(H_P(\delta g)-T_P)
\,\d^1 g ,
\end{multline}
where $\d^1 g$ is a Haar measure on~$G(\A)^1$. Its convergence was proved in~\cite{Arthur6} under some regularity condition on~$T$, which was shown to be unnecessary in~\cite{Hoffmann3}.

We can define instances $T_P\in\a_P=\Hom(X(P),\R)$ of the truncation parameter $T$ for nonstandard parabolic subgroups $P$, too, so that the identity
\[
H_{\delta P\delta^{-1}}(\delta g)-T_{\delta P\delta^{-1}}=\Ad(\delta)(H_P(g)-T_P)
\]
holds for all $\delta\in G(F)$ and $g\in G(\A)$. In view of $H_{\delta P\delta^{-1}}(\delta g)=\Ad(\delta)(H_P(g)-H_P(\delta^{-1}))$, it suffices to set $T_{\delta P\delta^{-1}}=\Ad(\delta)(T_P-H_P(\delta^{-1}))$ for all standard parabolic subgroups $P$ and all $\delta\in G(F)$. In the formula for~$J^T(f)$, we may omit the sum over~$\delta$ if we extend the other sum over all parabolic subgroups~$P$. Thus, the choice of~$P_0$ is unimportant. In \cite[Section~2]{Arthur2}, a version $J(f)=J^{T_0}(f)$ was defined, where $T_0$ is determined by the choice of $\bK$ and a minimal Levi subgroup~$M_0$. We will, however, retain the dependence on~$T$. When expanding $J^T(f)$ in terms of weighted orbital integrals, it is useful to make those depend on $T$ as well.

This entails a minor modification of the weight factors introduced in~\cite{Arthur2}. Let $M$ be a Levi subgroup of~$G$. This means that $M$ is a Levi component of a parabolic subgroup~$P$, in which case there is a canonical isomorphism~$\a_M\to\a_P$. The set $\P(M)$ of such parabolic subgroups is in bijection with the set of chambers in~$\a_M$. We consider the Fourier-Laplace transform
\[
\theta_P(\lambda)^{-1}=\int_{\a_M}\widehat\tau_P(X)\,e^{\lambda(X)}\,\d X
\]
of $\widehat\tau_P$ defined for $\lambda\in\a_{M,\C}^*$ with $\Re(\lambda)$ negative on the chamber $\a_P^+$ corresponding to~$P$. Its inverse $\theta_P(\lambda)$ is a homogeneous polynomial and extends to all of~$\a_{M,\C}^*$.
For any $g\in G(\A)$ and any truncation parameter~$T$, the family of functions
\[ v_P(\lambda,g,T) = e^{\lambda(T_P-H_P(g))}   \]
on~$i\a_M^*$, indexed by $P\in\P(M)$, is a $(G,M)$-family.
Hence
\[  v_M(\lambda,g,T) = \sum_{P\in\P(M)} \frac{v_P(\lambda ,g,T)}{\theta_P(\lambda)}  \]
extends to a smooth function on $i\a_M^*$.
We put
\[
v_M(g,T)=\lim_{\lambda\to 0}v_M(\lambda,g,T).
\]
Then $v_{\delta M\delta^{-1}}(\delta g,T)=v_M(g,T)$ for all $\delta\in G(F)$.

Next we recall from~\cite{Arthur5} the definition of weighted orbital integrals $J_M^T(\gamma,f)$ for $\gamma\in M(F)$ and $f\in C_c^\inf(G(F_S))$, where $S$ is a finite subset of $\Sigma$ containing $\Sigma_\inf$.
We denote by $G_{\gamma,+}$ the centralizer of $\gamma$ in $G$, by $G_\gamma$ the Zariski-connected component of $1$ in $G_{\gamma,+}$ and by $\fg_\gamma$ its Lie algebra.
We fix a $G(F_S)$-invariant measure on~$G_\gamma(F_S)\bsl G(F_S)$ and set
\[
D(\gamma)=\det(1-\mathrm{Ad}(\sigma))_{\fg/\fg_\sigma}\in F,
\]
where $\sigma$ is the semisimple part of~$\gamma$. In the case $G_\gamma\subset M$, the function $v_M$ is left $G_\gamma(F_S)$-invariant, and one defines
\begin{equation}\label{defJM}
J_M^T(\gamma,f)=\bigl| D(\gamma) \bigr|_S^{1/2} \, \int_{G_\gamma(F_S)\bsl G(F_S)} f(g^{-1}\gamma g) \, v_M(g,T) \,\d g.
\end{equation}
In the case $G_\gamma\not\subset M$, one uses the weighted orbital integrals $J_L^T(a\gamma,f)$ with Levi subgroups $L$ containing~$M$ and $a\in A_M(F)$ such that $G_{a\gamma}\subset L$. Note that the set of such~$a\gamma$ has $\gamma$ as a limit point. Arthur has introduced a $(G,M)$-family $\{r_P(\lambda,\gamma,a)\mid P\in\P(M)\}$ such that, with the aid of the corresponding function
\[
r^G_M(\gamma,a)=\lim_{\lambda\to 0}\sum_{P\in\P(M)}\frac{r_P(\lambda,\gamma,a)}{\theta_P(\lambda)}
\]
and its analogues with $L$ in place of~$G$, one can define
\begin{equation}\label{defJM'}
J_M^T(\gamma,f)=\lim_{a\to 1} \sum_{L\in\L(M)}r^L_M(\gamma,a) \, J_L^T(a\gamma,f).
\end{equation}
We will recall the definition of $r_P$ in Section~\ref{2s4} below, where we introduce a minor modification in order to simplify the explicit formulas. One recovers the traditional $T$-independent distribution $J_M(\gamma,f)$ by replacing all points $T_P$ by zero. However, one then loses the property
\[
J_{\delta M\delta^{-1}}^T(\delta\gamma\delta^{-1},f)=J_M^T(\gamma,f)
\qquad(\delta\in G(F))
\]
and has to stick to the set $\L$ of standard Levi subgroups, i.~e., those containing~$M_0$.

Finally, we can state the fine geometric expansion of the trace distribution.
Suppose that $\gamma$, $\gamma'\in G(F)$ and let $\sigma$ (resp.~$\sigma'$) be the semi-simple part of the Jordan decomposition of $\gamma$ (resp.~$\gamma'$).
We say that $\gamma'$ is $(G,S)$-equivalent to $\gamma$ if there exists an element $\delta\in G(F)$ such that $\sigma=\delta^{-1}\sigma'\delta$ and the unipotent elements $\sigma^{-1}\gamma$ and $\sigma^{\prime \, -1}\delta^{-1}\gamma'\delta$ are $G_\sigma(F_S)$-conjugate.
We denote by $(M(F))_{M,S}$ the set of $(M,S)$-equivalence classes in $M(F)$.
Then, for every compact subset $\Delta$ of~$G(\A)^1$, there is a finite subset $S_\Delta\supset \Sigma_\inf$ of $\Sigma$ such that, for every finite subset $S$ of $\Sigma$ containing~$S_\Delta$, every $M\in\L$ and $\gamma\in M(F)$ there are numbers $a^M(S,\gamma)$ such that
\[
J^T(f)=\sum_{M\in\L}|W_0^M|\, |W_0^G|^{-1} \, \sum_{\gamma\in (M(F))_{M,S} } a^M(S,\gamma) \, J_M^T(\gamma,f) 
\]
for any $f\in C_c^\inf(G(F_S)^1)$ with $\operatorname{supp}f\subset\Delta$ (cf. \cite[Theorem~9.2]{Arthur4} for the case $T=T_0$). On the left-hand side, $f$ is extended to $G(\A)^1$ by setting $f(g_Sg^S)=f(g_S)\phi_{\bK^S}(g^S)$ for $g_S\in G(F_S)$ and $g^S\in G(\A^S)$, where $\phi_{\bK^S}$ is the characteristic function of~$\bK^S=\prod_{v\notin S}\bK_v$.
Due to the conjugacy invariance, we may replace the sum over $\L$ by a sum over the set of $G(F)$-conjugacy classes of Levi subgroups~$M$, in which case the factor $|W_0^M|\, |W_0^G|^{-1}$ disappears.

\subsection{Summary of results on coefficients}

We are going to summarize known results as well as our new results on the coefficients appearing on the geometric side of the Arthur trace formula. For groups of $F$-rank one, the coefficients can be expressed by prehomogeneous zeta functions (cf.~\cite{Hoffmann} for the non-adelic situation).
In the present paper, however, we work entirely in the adelic situation.
The coefficients $a^M(S,\gamma)$ for arbitrary elements $\gamma$ can be expressed in terms of the coefficients $a^{G_\sigma}(S,u)$, where $\sigma$ and $u$ are the semisimple and unipotent parts of~$\gamma$, resp. (cf. \cite[(8.1)]{Arthur4}).
Therefore it is enough to consider unipotent elements of groups $G$ which can occur as centralizers of semisimple elements. For unipotent elements, $(G,S)$-equivalence coincides with $G(F_S)$-conjugacy. Thus, we consider the set $(\cU_G(F))_{G,S}$ of unipotent $G(F_S)$-conjugacy classes which meet~$G(F)$.

We need the following additional notation related to the number field $F$ and its adele ring~$\A$. We denote by $| \; |$ the idele norm on $\A^\times$ and set $\A^1=\{ x\in\A^\times \,\big| \, |x|=1 \}$.
For each finite place $v$, let $\pi_v$ denote a prime element of $F_v$ and let $q_v$ denote the cardinality of the residue field of $F_v$. If $\chi=\prod_v\chi_v$ is a character of $\A^1/F^\times\cong \A^\times/F^\times(\R^\times)^0$, where $\chi_v$ is a character of $F_v^\times$, we set
\begin{multline}\label{lft}
 L_v(s,\chi_v)=\begin{cases}(1-\chi_v(\pi_v)q_v^{-s})^{-1} & \text{if $v<\inf$ and $\chi_v$ is unramified,}  \\ 1& \text{if $v<\inf$ and $\chi_v$ is ramified}, \end{cases} 
\end{multline}
\[
L^S(s,\chi)=\prod_{v\not\in S}L_v(s,\chi_v), \quad \text{and} \quad L(s,\chi)=\prod_{v<\inf} L_v(s,\chi_v),
\]
where $S$ always denotes a finite set of places of $F$ which contains all Archimedean places. It is well known that $L^S(s,\chi)$ is absolutely convergent and holomorphic for $\Re(s)>1$ and can be meromorphically continued to the whole complex $s$-plane.
We set
\begin{equation}\label{dzft}
\zeta_F^S(s)=L^S(s,\trep_F) \quad \text{and} \quad \zeta_F(s)=L(s,\trep_F) , 
\end{equation}
where $\trep_F$ denotes the trivial character of $\A^1/F^\times$.
It is known that $\zeta_F^S(s)$ has a simple pole at $s=1$.
If $\chi\neq \trep_F$, $L^S(s,\chi)$ is holomorphic at $s=1$.

For general conventions concerning measures on Levi subgroups, unipotent radicals, and maximal compact subgroups see Section~\ref{2s3} below.
The volume of $G(F)\bsl G(\A)^1$ will be denoted by~$\vol_G$.
Since $a^G(S,1)=\vol_G$ (cf. \cite{Arthur3}), it is enough to consider the case $u\neq 1$.
We must choose measures for $\a_M^G$ and $J_G(u,f)$.
For the choices in the following cases, we refer to Sections \ref{3s4}, \ref{5s2}, \ref{6s1}, and \ref{7s} and Appendices \ref{appen1} and \ref{appen2}.

First, we review known results for $\GL(2)$, $\SL(2)$, $\GL(3)$, and $\SL(3)$. For convenience, we set
\begin{equation}\label{1e1}
u_\alpha  =   \begin{pmatrix}1&\alpha \\ 0&1 \end{pmatrix} \in \SL(2). 
\end{equation}
If $G=\GL(n)$ or $\SL(n)$, then $M_0$ denotes the minimal Levi subgroup which consists of diagonal matrices in $G$.
In the case $G=\GL(2)$, we have $(\cU_G(F))_{G,S}=\{ 1 ,  u_1  \}$, where we identify each $G(F_S)$-conjugacy class with its representative element.
We denote by $c_F$ the residue of $\zeta_F(s)$ at $s=1$ and by $\fc_F(S)$ the constant term in the Laurent expansion of $\zeta_F^S(s)$ at $s=1$.
We know from~\cite{FL,JL,GJ} that
\[
a^{\GL(2)}(S,u_1) = \frac{\vol_{M_0}}{2\,c_F} \, \fc_F(S).
\]
In the case $G=\SL(2)$, we have $(\cU_G(F))_{G,S}=\{1, \,  u_\alpha \, | \, \alpha\in F^\times/(F^\times\cap(F_S^\times)^2) \}$.
We set $\chi_S=\prod_{v\in S}\chi_v$ and $|x|_S=\prod_{v\in S}|x|_v$.
The stabilization of unipotent terms in \cite[(5.11)]{LL} provides
\[
a^{\SL(2)}(S,u_\alpha) = \frac{\vol_{M_0}}{2\, c_F} \Big\{\fc_F(S)+ \sum_{\chi\neq \trep_F} \chi_S(\alpha) \, L^S(1,\chi) \Big\}
\]
for $\alpha\in F^\times/(F^\times\cap (F_S^\times)^2)$, where $\chi$ runs over all nontrivial quadratic characters of $\A^1/F^\times$ unramified outside~$S$ (cf. Section \ref{3s}).
Note that the right hand side is a finite sum.

For $G=\GL(3)$ or $\SL(3)$, let $M'$ be a maximal Levi subgroup containing~$M_0$ and the elements
\[
u'= \begin{pmatrix}1&0&0 \\ 0&1&1 \\ 0&0&1 \end{pmatrix} \quad \text{and} \quad u''_\alpha = \begin{pmatrix}1&1&0 \\ 0&1&\alpha \\ 0&0&1 \end{pmatrix} .
\]
In the case $G=\GL(3)$, we have $(\cU_G(F))_{G,S}=\{ \, 1 \, , \; u' \, , \; u''_1 \}$.
Let $c_F^S$ denote the residue of $\zeta_F^S(s)$ at $s=1$ and let $\fc_F'(S)$ denote the coefficient of $(s-1)$ in the Laurent expansion of $\zeta_F^S(s)$ at $s=1$.
We deduce from \cite{Flicker} (cf. also~\cite{Matz}) that
\begin{align*}
& a^{\GL(3)}(S,u')= \vol_{M'} \frac{ \frac{\d}{\d s}\zeta_F^S(s)|_{s=2} }{\zeta_F^S(2)} ,\\
&  a^{\GL(3)}(S,u''_1)= \frac{\vol_{M_0}}{3 \, c_F^2} \Big\{ \mathfrak c_F(S)^2 +  \fc_F'(S)\, c_F^S \Big\}.
\end{align*}
If $G=\SL(3)$, then $(\cU_G(F))_{G,S}=\{ \, 1 \, , \; u' \, , \; u''_\alpha \, | \, \alpha\in F^\times/(F^\times\cap(F_S^\times)^3) \}$.
Using the arguments for $\SL(2)$ and $\GL(3)$ we have
\begin{align*}
  a^{\SL(3)}(S,u') & = \vol_{M'} \frac{ \frac{\d}{\d s}\zeta_F^S(s)|_{s=2} }{\zeta_F^S(2)} , \\
a^{\SL(3)}(S,u''_\alpha) & =\frac{\vol_{M_0}}{3 \, c_F^2} \Big\{ \fc_F(S)^2 + \fc_F'(S)\, c_F^S  \\
& \qquad\qquad + \sum_{\chi\neq\trep_F}\chi_S(\alpha)\, L^S(1,\chi) \, L^S(1,\chi^{-1})  \Big\} 
\end{align*}
where $\chi$ runs over all nontrivial cubic characters on $\A^1/F^\times$ unramified outside~$S$ (cf. Appendix~\ref{appen2}).

In order to explain our formulas for $\GSp(2)$ and $\Sp(2)$, we recall the Shintani zeta function for the space of binary quadratic forms.  For the sake of simple statements, we assume that $S$ contains all places dividing two.
We denote the equivalence class of an element $d\in F^\times$ in $F^\times/(F^\times)^2$ by~$\ds$.
For each~$\ds$, we have the quadratic character $\chi_d=\prod_v \chi_{d,v}$ on $\A^1/F^\times$ by class field theory.
We set
\[
\para(F)=\{\ds\in F^\times/(F^\times)^2 \mid d\in F^\times - (F^\times)^2 \}
\]
and, for $d_S\in F_S^\times$,
\[
\para(F,S,d_S)=\{ \ds \in \para(F) \mid d \in d_S(F_S^\times)^2 \}. 
\]
The Shintani zeta function is given by
\[
\xi^S(s;d_S)= \frac{\zeta^S_F(2s-1) \, \zeta^S_F(2s)}{\zeta_F^S(2)} \,\sum_{\ds\in \para(F,S,d_S)} \frac{ L^S(1,\chi_d)  }{L^S(2s,\chi_d) \, N(\mathfrak{f}^S_d)^{s-\frac{1}{2}}},
\]
where $N(\mathfrak{f}^S_d)$ equals the product of the numbers $q_v$ over all $v\notin S$ for which $\chi_{d,v}$ is ramified.
This zeta function is essentially the same as the series $\xi_{\mathbf{x}_S}(s)$ of \cite{Datskovsky}.
However, his formulation and argument are not suitable for the trace formula.
We will reproduce them using Saito's results \cite{Saito1,Saito2,Saito3} in order to simplify some arguments and clear up relations for rational and adelic orbits as well as local and global measures.
Additionally, we need his results to study explicit forms of Shintani zeta functions with non-trivial quadratic characters.
The original Shintani zeta functions in \cite{Siegel1,Shintani} can be related to the zeta functions $\xi^S(s;d_S)$ via zeta integrals (cf. \cite[Section 2]{Saito2}).

The groups $\GSp(2)$ and $\Sp(2)$ have four classes of unipotent elements, which are (1)~unit, (2)~minimal, (3)~subregular, and (4)~regular.
For $\alpha\neq 0$ and $\det(x)\neq 0$ we set
\begin{align}\label{1e2}
& n_\min(\alpha)=\begin{pmatrix}1&0&\alpha&0 \\ 0&1&0&0 \\ 0&0&1&0 \\ 0&0&0&1 \end{pmatrix}  , \quad  n_\reg(\alpha)=\begin{pmatrix}1&1&0&\alpha \\ 0&1&0&\alpha \\ 0&0&1&0 \\ 0&0&-1&1 \end{pmatrix}, \\
& n_\sub(x)=\begin{pmatrix}I_2&x \\ O_2 &I_2 \end{pmatrix} \nonumber
\end{align}
where $I_2$ (resp.~$O_2$) is the unit (resp.~zero) matrix of degree two.
Note that $n_\sub(x)\in \Sp(2)$ if and only if $x$ is symmetric.
Denote by $V^\ss(F)$ the set of non-degenerate $2\times 2$ symmetric matrices over $F$ and by $\sim_S$ the equivalence relation on $V^\ss(F)$ such that $x\sim_S y$ if and only if $\det(x^{-1}y)\in (F_S^\times)^2$.
Then, we have
\[
(\cU_{\GSp(2)}(F))_{\GSp(2),S}=\{ 1 , \; n_\min(1) , \; n_\sub(x), \; n_\reg(1) \, | \, x\in V^\ss(F)/{\sim_S} \,  \}.
\]
When $G=\GSp(2)$ or $\Sp(2)$, then $M_0$ denotes the minimal Levi subgroup consisting of diagonal matrices in~$G$, and non-conjugate maximal Levi subgroups $M_1$ and $M_2$ containing $M_0$ will be chosen in Section~\ref{5s1}.
We will see in Section \ref{6s1} that
\begin{align*}
a^{\GSp(2)}(S,n_\min(1)) &= \frac{ \vol_{M_2}  }{2\, c_F} \, \zeta_F^S(2) , \\
a^{\GSp(2)}(S,n_\sub(x)) &=\frac{\vol_{M_1}}{2 \, c_F}  \, \fC_F(S,-\det(x)) \\
&\qquad + \frac{\vol_{M_1}}{2\, c_F}  \begin{cases} \zeta_F^S(3)^{-1} \frac{\d}{\d s}\zeta_F^S(s)|_{s=3}   & \text{if $x\sim_S \xd_1$} , \\ 0 & \text{if $x\not\sim_S \xd_1$} , \end{cases} \\
a^{\GSp(2)}(S,n_\reg(1)) &= \frac{\vol_{M_0}}{2\, c_F^2}\, (\fc_F(S))^2  +\frac{3\, \vol_{M_0}}{4\, c_F^2 }\, \fc_F'(S)  \, c_F^S,
\end{align*}
where $\xd_1=\begin{pmatrix}1&0 \\ 0& -1\end{pmatrix}$ and $\fC_F(S,\alpha)$ is the constant term in the Laurent expansion of $\xi^S(s;\alpha)$ at $s=3/2$.

Next, we treat $\Sp(2)$.
We denote the Hasse invariant of an element $x$ of~$V^\ss(F)$ over $F_v$ by $\varepsilon_v(x)$ and define an equivalence relation $\sim'_S$ on $V^\ss(F)$ by setting $x\sim'_S y$ if and only if $\det(x^{-1}y)\in (F_S^\times)^2$ and $\varepsilon_v(x)=\varepsilon_v(y)$ for all $v\in S$.
We easily see that
\begin{multline*}
(\cU_{\Sp(2)}(F))_{\Sp(2),S}=\{ 1 , \; n_\min(\alpha), \; n_\sub(x) , \; n_\reg(\alpha) \\
| \, \alpha\in F^\times/(F^\times\cap (F_S^\times)^2), \;  x\in V^\ss(F)/{\sim_S'} \}. 
\end{multline*}
We will see in Section \ref{7s} that
\[
a^{\Sp(2)}(S,n_\min(\alpha))= \frac{\vol_{M_2}}{2\, c_F} \sum_\chi \chi_S(\alpha) \, L^S(2,\chi),
\]
where $\chi$ runs over all quadratic characters of $\A^1/F^\times$ unramified outside~$S$, and also that
\begin{align*}
 a^{\Sp(2)}(S,n_\sub(x))&= \frac{\vol_{M_1}}{2\, c_F}  \fC_F(S,-\det(x)) \\
&\quad + \frac{\vol_{M_1}}{2\, c_F}  \begin{cases} \zeta_F^S(3)^{-1} \, \frac{\d}{\d s}\zeta_F^S(s)|_{s=3}  & \text{if $x\sim_S' \xd_1$} , \\ 0 & \text{if $x\not\sim_S' \xd_1$} \end{cases} \\
&\quad  + \frac{\vol_{M_1}}{2\, c_F} \Big(\prod_{v\in S} \varepsilon_v(x)  \Big) \sum_{\check d\in \para^{\mathrm{ur}}(F,S,-\det(x)) } L^S(1,\chi_d) 
\end{align*}
where
\[
\para^{\mathrm{ur}}(F,S,d_S)=\{\ds\in\para(F,S,d_S) \mid \text{$\chi_{d,v}$ is unramified for every $v\not\in S$ }  \} ,
\]
and
\begin{multline*}
 a^{\Sp(2)}(S,n_\reg(\alpha)) = \frac{\vol_{M_0}}{2\, c_F^2} \fc_F(S) \, \Big\{ \fc_F(S)+\sum_{\chi\neq \trep_F}\chi_S(\alpha) \, L^S(1,\chi) \Big\}   \\
 + \frac{3\, \vol_{M_0}}{4\, c_F^2 }\, \fc_F'(S)  \, c_F^S +\frac{ \vol_{M_0}}{4\, c_F^2 } c_F^S \sum_{\chi\neq \trep_F} \chi_S(\alpha) \, \frac{\d}{\d s}L^S(s,\chi)\Big|_{s=1},
\end{multline*}
where $\chi$ runs over all nontrivial quadratic characters on $\A^1/F^\times$ unramified outside~$S$.

Consulting the results on local unipotent orbital integrals in \cite{Assem}, we can find relations between the above coefficients and elliptic endoscopic groups of~$\Sp(2)$.
Assume that the value of $\vol_{M_k}$ $(k=0,1,2)$ is the same for $\GSp(2)$ and $\Sp(2)$.
Theorem 4.2 (ii) of \cite{Assem} suggests that the difference $a^{\Sp(2)}(S,n_\min(\alpha))-a^{\GSp(2)}(S,n_\min(1))$ comes from the central contributions of the quasi-split groups $\SO(4,\chi_d)$.
The value $L^S(2,\chi_d)$ in the coefficient $a^{\Sp(2)}(S,n_\min(\alpha))$ should stand for the volume $\vol_{\SO(4,\chi_d)}$.
Similarly, the value $L^S(1,\chi_d)$ in $a^{\Sp(2)}(S,n_\sub(x))$ should stand for $\vol_{\SO(2,\chi_d)}$.
This is compatible with \cite[Theorem 4.3 (ii)]{Assem} implying that the difference $a^{\Sp(2)}(S,n_\sub(x))-a^{\GSp(2)}(S,n_\sub(x))$ comes from the central contributions of $\SL(2)\times \SO(2,\chi_d)$.
Furthermore, we see that the difference $a^{\Sp(2)}(S,n_\reg(\alpha))-a^{\GSp(2)}(S,n_\reg(1))$ equals a linear combination of the coefficients of the regular stable unipotent orbital integrals of $\SO(4,\chi_d)$ and $\SL(2)\times \SO(2,\chi_d)$.
Such coefficients for $\SO(4,\chi_d)$ can be studied by the method of Labesse and Langlands (cf.\ the proof of Theorem~\ref{3t}).
This agrees with Theorems 4.2~(i) and 4.3~(i) of~\cite{Assem}.

\section{Preliminaries}

\subsection{Further notation}\label{2s1}

The cardinality of a finite set $X$ is denoted by $|X|$ and the subset of invertible elements of any ring $R$ by~$R^\times$.
We set $\Rp=\{ x\in\R \, | \, x>0  \}$ and denote by~$i\in\C$ the imaginary unit. If $K$ is a finite extension of a field $k$ of characteristic zero, we denote by $N_{K/k}(x)$ (resp.\ $\mathrm{Tr}_{K/k}(x)$) the norm (resp.\ trace) of an element $x\in K$.

Let $F$ be an algebraic number field.
In situations involving more than one field, we distinguish the objects introduced above by an additional subscript $F$.
Let $\fO$ be the ring of integers of~$F$.
We denote by $\Sigma_\inf$ (resp.\ $\Sigma_\fin$) the set of all the infinite (resp.\ finite) places of $F$.
For any $v\in\Sigma=\Sigma_\inf\cup\Sigma_\fin$, we denote by $F_v$ the completion of $F$ at $v$.
We set $\Sigma_\C=\{v\in\Sigma_\inf \, | \, F_v\cong \C\}$, $\Sigma_\R=\{v\in \Sigma_\inf \, | \, F_v\cong \R\}$, and $\Sigma_2=\{ v\in\Sigma_F \,\big| \,  v|2    \}$.
For each $v\in\Sigma_\fin$, we denote by $\fO_v$ the ring of integers of~$F_v$.
We choose a prime element $\pi_v$ of $\fO_v$ and set $q_v=|\fO_v/\pi_v\fO_v|$.

Let $\A$ be the adele ring of $F$, $\A_\fin=\{ (x_v)\in\A \, | \, x_v=0  \text{ for any }v\in\Sigma_\inf \}$, and $F_\inf=\prod_{v\in\Sigma_\inf}F_v$.
Then we have $\A=F_\inf \times \A_\fin$ and similarly $\A^\times=F_\infty^\times\times\A_\fin^\times$ for the idele group.
For a vector space $V$ over $F$, we denote by $\cS(V(F_v))$, $\cS(V(\A))$, $\cS(V(F_\inf))$, and $\cS(V(\A_\fin))$ the space of Schwartz-Bruhat functions on $V(F_v)$, $V(\A)$, $V(\A_\inf)$, and $V(\A_\fin)$ respectively.

Let $\d x$ denote the Haar measure on $\A$ normalized by $\int_{\A/F}\d x=1$. 
We fix a non-trivial additive character $\psi_\Q$ on $\A_\Q/\Q$ and set $\psi_F=\psi_\Q \circ \Tr_{F/\Q}$.
Then, $\d x$ is the self-dual Haar measure with respect to $\psi_F$, i.e., if we set
\[
\hat\phi(y)=\int_\A \phi(x)\psi_F(xy) \, \d x
\]
for $\phi\in\cS(\A)$ and~$y\in \A$, then
\[
\phi(x) =\int_\A \hat\phi(y)\psi_F(-xy)\, \d y. 
\]
Moreover, the Poisson summation formula then reads
\[
\sum_{x\in F}\phi(x)=\sum_{y\in F} \hat\phi(y). 
\]
For each $v\in\Sigma$, let $\d x_v$ denote a Haar measure on $F_v$.
We assume that
\[
\d x=\prod_{v\in\Sigma} \d x_v \quad \text{and} \quad \int_{\fO_v}\d x_v=1 \quad (v\in\Sigma_\fin) .
\]
In this setting, for $\psi_F=\prod_{v\in\Sigma}\psi_{F_v}$, we note that $\d x_v$ is not the self-dual measure with respect to $\psi_{F_v}$ in general.
We denote by $|\; |_v$ the normal valuation of $F_v$ and by $|\; |=|\; |_\A$  the idele norm on $\A^\times$, i.~e., $\d (a_vx_v)=|a_v|_v \d x_v$ for $a_v\in F_v^\times$ and
 $\d (ax)=|a| \, \d x$ for $a\in\A^\times$. Then $|x|=\prod_{v\in\Sigma} |x_v|_v$, and we have $|\pi_v|_v=q_v^{-1}$ if $v\in\Sigma_\fin$, $|x_v|_v=|x_v|$ if $v\in\Sigma_\R$, and $|z|_v=|z|^2$ if $v\in\Sigma_\C$.

Let $\d^\times x_v$ denote a Haar measure on $F_v^\times$ for each $v\in\Sigma$.
We normalize $\d^\times x_v$ for $v\in\Sigma_\fin$ by the condition $\int_{\fO_v^\times}\d^\times x_v=1$.
We also normalize $\d^\times x_v$ for $v\in\Sigma_\inf$ by the condition $\d^\times x_v=\d x_v/|x|_v$.
The idele norm induces an isomorphism $\A^\times/\A^1\to\Rp$.
We choose the Haar measure $\d^\times x=\prod_{v\in\Sigma}\d^\times x_v$
on $\A^\times$ and normalize the Haar measure $\d^1 x$ on $\A^1$ in such a way
that the quotient measure on $\Rp$ is $\d t/t$, where $\d t$ is the
Lebesgue measure on~$\R$.



For $(x_1,x_2,\dots,x_n)\in F_v^n$, we define
\[
\|(x_1,x_2,\dots,x_n)\|_v = \begin{cases}  |x_1|^2+|x_2|^2+ \dots + |x_n|^2  & \text{if $v\in \Sigma_\C$,} \\  (|x_1|^2+|x_2|^2+ \dots + |x_n|^2)^{1/2} & \text{if $v\in \Sigma_\R$,} \\   \mathrm{max}(|x_1|_v , \, |x_2|_v, \, \dots , \, |x_n|_v )  & \text{if $v\in \Sigma_\fin$.}  \end{cases}
\]
Let $x=(x_v)\in \A^n$.
If $\|x_v\|_v=1$ for almost all $v$, then the height of $x$ is defined by $\|x\|=\prod_{v\in\Sigma}\|x_v\|_v$.

Let $S$ be a finite subset of $\Sigma$.
We set $F_S=\prod_{v\in S}F_v$.
We define the norm $|\; |_S$ on $F_S$ by $|x|_S=\prod_{v\in S}|x_v|_v$ where $x=(x_v)\in F_S$.
We also define $\|\;\|_S$ on $F_S^n$ by $\|x\|_S=\prod_{v\in S}\|x_v\|_v$ where $x=(x_v)\in F_S^n$.
A measure $\d x_S$ (resp. $\d^\times x_S$) on $F_S$ (resp. $F_S^\times$) is defined by
\[
\d x_S=\prod_{v\in S} \d x_v \quad (\text{resp.} \quad \d^\times x_S=\prod_{v\in S} \d^\times x_v).
\]

For a commutative unital ring~$R$, we denote the ring of matrices of degree $n$ over $R$ by $\Mat(n,R)$, the unit matrix by $I_n$ and the zero matrix by~$O_n$.
We write $^tx$ for the transpose of $x\in\Mat(n,R)$ and $\diag(a_1,a_2,\ldots,a_n)$ for the diagonal matrix with entries $a_1$, $a_2$, \ldots, $a_n$.
For an algebraic group $G$ defined over $F$ and a commutative unital $F$-algebra~$R$, we denote by $G(R)$ the group of $R$-rational points of $G$.
Let $\mathbb{G}_m$ denote the multiplicative group and let $R_{E/F}(H)$ denote the group obtained from an algebraic group $H$ over $E$ by restricting scalars from $E$ to $F$.
If $G\subset \GL(n)$, we set $G(\fO)=G(F)\cap \GL(n,\fO)$ and $G(\fO_v)=G(F_v)\cap \GL(n,\fO_v)$ for each $v\in\Sigma_\fin$.
We set
\begin{multline*}
\GSp(2)=\Big\{  g\in \GL(4) \,  \Big|  \, \exists \, \mu(g)\in \GL(1) \,\, {\rm s.t.} \,\, \\
 ^tg \begin{pmatrix}O_2 & I_2 \\ -I_2 & O_2 \end{pmatrix} g=\mu(g) \begin{pmatrix} O_2 & I_2 \\ -I_2 & O_2 \end{pmatrix}  \Big\}  
\end{multline*}
and $\Sp(2)=\{g\in \GSp(2) \, | \, \mu(g)=1  \}$.

Let $C_c^\inf(G(F_\inf))$ denote the space of compactly supported smooth functions on $G(F_\inf)$ and let $C_c^\inf(G(\A_\fin))$ denote the space of compactly supported locally constant functions on $G(\A_\fin)$.
We set $C_c^\inf(G(\A))=C_c^\inf(G(F_\inf))\otimes C_c^\inf(G(\A_\fin))$.


\subsection{Tate integral}\label{2s2}

Let $\chi$ be a character of $\A^1/F^\times$.
Since $\A^1/F^\times$ is identified with $\A^\times/F^\times(\R^\times)^0$, we may put $\chi =\prod_{v\in\Sigma}\chi_v$ where $\chi_v$ is a character of $F_v^\times$ for each $v\in\Sigma$.
For $\phi\in\cS(\A)$, we set
\[ \zeta(\phi,s,\chi)=\int_{\A^\times}|x|^s \, \chi(x)  \,\phi(x) \, \d^\times x=\int_{\A^\times/F^\times} |x|^s \chi(x) \sum_{y\in F^\times}  \, \phi(yx)  \, \d^\times x .  \]
The Tate integral $\zeta(\phi,s,\chi)$ is absolutely convergent and holomorphic for $\Re(s)>1$.
If we put $\zeta_+(\phi,s,\chi)=\int_{\A^\times/F^\times \, , \, |x|>1} |x|^s\, \chi(x) \, \phi(x) \, \d^\times x$, then $\zeta_+(\phi,s,\chi)$ converges absolutely for any $s\in\C$ and is entire on $\C$.
By the Poisson summation formula we have
\begin{equation}\label{zetapoisson}
\zeta(\phi,s,\chi)= \zeta_+(\phi,s,\chi) + \zeta_+(\hat\phi,1-s,\chi^{-1}) + \delta_\chi \,  \int_{F^\times\bsl\A^1}\d^1 x \Big( \frac{\hat\phi(0)}{s-1}-\frac{\phi(0)}{s} \Big)
\end{equation}
where $\delta_\chi=1$ if $\chi$ is the trivial character $\trep_F$ of $\A^1/F^\times$, and $\delta_\chi=0$ if $\chi\neq \trep_F$.
Hence, $\zeta(\phi,s,\chi)$ is meromorphically continued to the whole complex $s$-plane and satisfies the functional equation
\[
\zeta(\phi,s,\chi)=\zeta(\hat\phi,1-s,\chi^{-1}).
\]

Let $S$ be a finite subset of $\Sigma$ containing $\Sigma_\inf$ and $\chi_S=\prod_{v\in S}\chi_v$.
In \eqref{lft} and \eqref{dzft}, we introduced the functions $L^S(s,\chi)$, $\zeta_F^S(s)$, and $\zeta_F(s)$.
We set
\[
\zeta_S(\phi_S,s,\chi_S)=\int_{F_S^\times}\phi_S(x_S) \, |x_S|_S^s \, \chi_S(x_S) \, \d^\times x_S \quad (\phi_S\in\cS(\A_S) ) . 
\]
It is well known that $\zeta_S(\phi_S,s,\chi_S)$ is absolutely convergent and holomorphic for $\mathrm{Re}(s)>0$, and $\zeta_S(\phi_S,s,\chi_S)$ can be meromorphically continued to the whole complex $s$-plane.
Let $\phi_{0,v}$ denote the characteristic function of $\fO_v$ for each $v\in\Sigma_\fin$.
For $\phi_S\in\cS(F_S)$, we set $\phi = \phi_S\times \prod_{v\not\in S}\phi_{0,v}\in\cS(\A)$.
If we assume that $\chi$ is unramified outside~$S$ (i.~e., that $\chi_v$ is unramified for every $v\not\in S$), then we have
\[
\zeta(\phi,s,\chi)= \zeta_S(\phi_S,s,\chi_S)\times L^S(s,\chi).
\]
Some of the main properties of $L^S(s,\chi)$ are deduced from this equality and \eqref{zetapoisson}.
For example, it follows that the residue $c_F$ of $\zeta_F(s)$ at $s=1$ satisfies $c_F=\int_{F^\times\bsl \A^1}\d^1 x$.
We will use the Poisson summation formula to study several integrals and zeta functions.

Note that $c_S \, \d x_S=|x_S|_S\, \d^\times x_S$, where
\[
\quad c_S=\prod_{v\in S} c_v, \qquad
c_v=\begin{cases} (1-q_v^{-1})^{-1} & \text{if $v\in\Sigma_\fin$}, \\ 1 & \text{if $v\in\Sigma_\inf$}.  \end{cases}
\]
It is clear that the residue of $L^S(s,\trep_F)$ at $s=1$ is $c_F^S=c_F/c_S$.
We also need the constants
\begin{align*}
& \fc_F(S,\chi)=\lim_{s\to +1}\frac{\d}{\d s}(s-1) \, L^S(s,\chi), \\
& \fc_F'(S,\chi)=\frac{1}{2}\lim_{s\to +1}\frac{\d^2}{\d s^2}(s-1) \, L^S(s,\chi). 
\end{align*}
If $\chi\neq\trep_F$, then $\fc_F(S,\chi)=L^S(1,\chi)$ and $\fc_F'(S,\chi)=\frac{\d}{\d s}L^S(s,\chi)|_{s=1}$.
With the notation $\fc_F(S)=\fc_F(S,\trep_F)$ and $\fc_F'(S)=\fc_F'(S,\trep_F)$, the Laurent expansion of $\zeta_F^S(s)$ at $s=1$ reads
\[
\zeta_F^S(s)=\frac{c_F^S}{s-1} + \fc_F(S)+ \fc_F'(S)\, (s-1) + \cdots .
\]

\subsection{Algebraic groups}\label{2s3}

Let $G$ be a connected reductive algebraic group over $F$.
For the notation $X(G)_F$, $A_G$, $\a_G$, $\a_G^*$, $H_G$, and $G(\A)^1$, we refer to the first paragraph of Section~\ref{basobj}.
We have $\a^*_G=X(G)_F\otimes_\Z \R$, and $X(G)_F$ is identified with a subset of $\a_G^*$. Due to the duality between split tori and free abelian groups, every $F$-split torus is obtained from a $\Q$-split torus by base change. Thus considering $A_G$ as a $\Q$-split torus, we set $A_G^+=A_G(\R)^0$ and imbed it into $A_G(F_\infty)$ using the diagonal homomorphism $\R\to F_\infty$.

Fix a Haar measure $\d g$ (resp.\ $\d H$) on $G(\A)$ (resp.\ $\a_G$).
The homomorphism $H_G$ restricts to an isomorphism $A_G^+\to\a_G$, which transports $\d H$ to a Haar measure on $A_G^+$. 
Hence, a Haar measure $\d^1 g$ on $G(\A)^1$ is deduced from $\d g$ and $\d H$ by the isomorphism $A_G^+\to G(\A)/G(\A)^1$.
We set
\[
\vol_G=\int_{G(F)\bsl G(\A)^1}\d^1 g.
\]
Let $\d g_v$ denote a Haar measure on $G(F_v)$ for each $v\in\Sigma$.
We may assume $\d g=\prod_{v\in\Sigma}\d g_v$.
For a finite subset $S$ of $\Sigma$, a Haar measure $\d g_S$ on $G(F_S)$ is defined by $\d g_S=\prod_{v\in S}\d g_v$.

Let $M$ be a Levi subgroup of $G$ over $F$.
We denote by $\L(M)=\L^G(M)$ (resp.\ $\F(M)=\F^G(M)$) the set of Levi subgroups (resp.\ parabolic subgroups) of $G$ over $F$ that contain $M$.
For each $P\in\F(M)$, the Levi decomposition of $P$ is denoted by $P=M_PN_P$ where $M_P\in\L(M)$ and $N_P$ is the unipotent radical of $P$.
We set $\P(M)=\P^G(M)=\{P\in\F(M) \, | \, M_P=M\}$.

From now on, we fix a minimal Levi subgroup $M_0$.
We set $\L=\L^G=\L(M_0)$, $\F=\F^G=\F(M_0)$, and $\P=\P^G=\P(M_0)$.
For each $v\in\Sigma$, we take a maximal compact subgroup $\bK_v$ of $G(F_v)$ which is admissible relative to $M_0$ (cf. \cite[Section 1]{Arthur2}).
We set $\bK= \prod_{v\in\Sigma} \bK_v$.
The group $\bK$ is a maximal compact subgroup of $G(\A)$ and called admissible relative to $M_0$.
Then, we have $G(\A)=P(\A)\bK$ for any $P\in\P$.
Furthermore, $\bK\cap M(\A)$ is admissible relative to $M_0$ for any $M\in\L$.
We set $A_P=A_{M_P}$ and $\a_P=\a_{M_P}$ for each $P\in\F$.
A surjection $H_P \, : \, G(\A)\to \a_P$ is defined by
\[
H_P(nmk)=H_{M_P}(m),\quad n\in N_P(\A), \; m\in M_P(\A), \; k\in\bK  .
\]

Let $M\in\L$.
By the restriction $X(M)_F\to X(A_M)_F$, we have a bijection $\a_M^*\to\a_{A_M}^*$, by which we identify both vector spaces.
Let $M_1\in\L$, $M_2\in\L$, and $M_1\subset M_2$.
Then, we have $A_{M_2}\subset A_{M_1}\subset M_1 \subset  M_2$ over $F$.
Since the restriction $X(M_2)_F\to X(M_1)_F$ is injective, we obtain a linear injection $\a_{M_2}^*\to\a_{M_1}^*$ and a linear surjection $\a_{M_1}\to \a_{M_2}$.
Since the restriction $X(A_{M_1})_F\to X(A_{M_2})_F$ is surjective, we have a linear surjection $\a_{M_1}^*\to \a_{M_2}^*$ and a linear injection $\a_{M_2}\to \a_{M_1}$.
We set 
\[
\a_{M_1}^{M_2}=\{ a_1 \in \a_{M_1} \; | \; \langle a_1,a^*_2\rangle=0 , \; \forall a^*_2 \in \a_{M_2}^* \}
\]
and
\[
(\a_{M_1}^{M_2})^*= \{ a_1^* \in \a_{M_1}^* \; | \; \langle a_2,a^*_1\rangle=0 , \; \forall a_2 \in \a_{M_2} \}  .
\]
Then, we have
\[
\a_{M_1}=\a_{M_2}\oplus \a_{M_1}^{M_2}  \quad \text{and} \quad \a_{M_1}^*=\a_{M_2}^*\oplus (\a_{M_1}^{M_2})^*.
\]

For $P\in\F$, we set $\a_P^*=\a_{M_P}^*$, $\a_{P_1}^{P_2}=\a_{M_{P_1}}^{M_{P_2}}$, and $(\a_{P_1}^{P_2})^*=(\a_{M_{P_1}}^{M_{P_2}})^*$.
We denote by $\Phi_P\subset X(A_P)_F$ the set of roots of $A_P$ in the Lie algebra $\mathfrak n_P$ of $N_P$ and by $\Delta_P$ the subset of primitive roots.

We fix a minimal parabolic subgroup $P_0\in\P$.
Set $\a_0=\a_{M_0}$, $\a_0^*=\a_{M_0}^*$, $\a_0^P=\a_0^{M_P}=\a_{M_0}^{M_P}$, $(\a_0^P)^*=(\a_0^{M_P})^*=(\a_{M_0}^{M_P})^*$, and $\Phi_0=\Phi_{P_0}$.
Now, $\Phi_0\cup (-\Phi_0)$ is a root system in $(\a_0^G)^*$ and $\Phi_0$ is a system of positive roots.
Let $W_0=W_0^G$ denote the Weyl group of the root system $\Phi_0\cup (-\Phi_0)$ in $(\a_0^G)^*$.
We set $\Delta_0=\Delta_{P_0}$, i.e., $\Delta_0$ is the set of simple roots attached to $\Phi_0$.
Let $\widehat{\Delta}_0$ denote the set of simple weights corresponding to $\Delta_0$.

Let $P\in \F$.
Assume that $P\supset P_0$.
A subset $\Delta_0^P$ of $\Delta_0$ is defined by
\[
\a_P=\{  a\in\a_0 \; | \; \langle a ,\alpha \rangle=0,\; \, \forall \alpha\in\Delta_0^P  \}.
\]
We set
\[
\widehat{\Delta}_P=\{ \varpi_{\alpha} \in (\a_P^G)^* \; | \; \alpha\in\Delta_0-\Delta_0^P  \},
\]
where $\varpi_{\alpha}$ is the fundamental weight corresponding to $\alpha$.
Assume that $P_1$, $P_2\in\F$ and $P_0\subset P_1\subset P_2$.
We set
\[
\Delta_{P_1}^{P_2}=\{ \alpha|_{\a_{P_1}^{P_2}}\in (\a_{P_1}^{P_2})^* \; | \; \alpha\in\Delta_0^{P_2}-\Delta_0^{P_1}\} ,
\]
\[
\widehat{\Delta}_{P_1}^{P_2}=\{ \varpi|_{\a_{P_1}^{P_2}} \in (\a_{P_1}^{P_2})^* \; | \; \varpi\in\widehat{\Delta}_{P_1}-\widehat{\Delta}_{P_2} \}.
\]
Note that $\Delta_{P_1\cap M_{P_2}}=\Delta_{P_1}^{P_2}$ and $\widehat{\Delta}_{P_1\cap M_{P_2}}=\widehat{\Delta}_{P_1}^{P_2}$.

Let $P\in\F$.
Recall the element $\rho_P\in(\a_P^G)^*$ defined in Section \ref{basobj}.
The element $\rho_P$ satisfies
\[
\rho_P=\frac{1}{2}\sum_{\alpha\in\Phi_P} (\dim \fn_{\alpha}) \alpha,
\]
and $\delta_P(x)=e^{2\rho_P(H_P(x))}$ $(x\in P(\A))$ is the modular character of $P$.
Fix a Haar measure $\d g$ on $G(\A)$ and a Haar measure $\d H$ on $\a_P$.
We denote by $\d n$ the Haar measure on $N_P(\A)$ normalized by $\int_{N_P(F)\bsl N_P(\A)}\d n=1$.
Let $\d k$ denote the Haar measure on $\bK$ normalized by $\int_\bK \d k=1$ and let $\d a$ denote the Haar measure on $A_{M_P}^+$ induced from $\d H$.
Then, there exists an unique Haar measure $\d^1 m$ on $M_P(\A)^1$ such that
\begin{multline*}
\int_{G(\A)}f(g)\d g=\\
\int_{N_P(\A)}\int_{M_P(\A)^1}\int_{A_{M_P}^+}\int_{\bK} f(nmak)e^{-2\rho_P(H_P(a))} \, \d k\, \d^1 m \,  \d a \, \d n
\end{multline*}
for any $f\in C_c^\inf(G(\A))$.
A Haar measure $\d m$ on $M_P(\A)$ is determined by $\d m=\d^1m \, \d a$.
For each $v\in\Sigma$, let $\d g_v$ denote a Haar measure on $G(F_v)$ and let $\d n_v$, $\d m_v$, and $\d k_v$ denote Haar measures on $N_P(F_v)$, $M_P(F_v)$ and $\bK_v$ respectively.
We may assume that $\d g=\prod_{v\in \Sigma}\d g_v$, $\d k=\prod_{v\in \Sigma} \d k_v$, $\d n=\prod_{v\in \Sigma} \d n_v$, and $\d m=\prod_{v\in \Sigma} \d m_v$. 
Furthermore, we may also assume that
\[ \int_{\bK_v}\d k_v=1\quad (v\in\Sigma), \quad \int_{N_P(F_v)\cap \bK_v}\d n_v=1 \quad (v\in\Sigma_\fin)\]
and
\begin{multline*}
 \int_{G(F_v)}f_v(g_v)\d g_v=\\
\int_{N_P(F_v)}\int_{M_P(F_v)}\int_{\bK_v}f_v(n_v m_v k_v)e^{-2\rho_P(H_P(m_v))} \, \d k_v \, \d m_v  \, \d n_v 
\end{multline*}
for any $v\in \Sigma$ and any $f_v\in C_c^\inf(G(F_v))$.
We also have the Haar measures $\d n_S=\prod_{v\in S}\d n_v$, $\d m_S=\prod_{v\in S}\d m_v$, and $\d k_S=\prod_{v\in S}\d k_v$ on $N(F_S)$, $M(F_S)$, and $\bK_S=\prod_{v\in S}\bK_v$ respectively.

\subsection{Weighted orbital integrals}\label{2s4}

We revisit the definition~\eqref{defJM'} of weighted orbital integrals $J_M^T(\gamma,f)$, where $M$ is a Levi subgroup, $\gamma\in G(F)$ and $f\in C_c^\infty(G(F_S)^1)$ for some finite subset $S$ of $\Sigma$ containing~$\Sigma_\infty$. There is an alternative description depending on the choice of a parabolic subgroup $P_1\in\P(M)$. We decompose the integral~\eqref{defJM} defining $J_M(a\gamma,f)$ in the case $G_{a\gamma}\subset M$ by writing $g=nmk$ with $n\in N_{P_1}(F_S)$, $m\in M_\gamma(F_S)\bsl M(F_S)$ and $k\in\bK_S=\prod_{v\in S}\bK_v$. Then $\mu=m^{-1}\gamma m$ runs through the $M(F_S)$-orbit $O_\gamma(M(F_S))$ of~$\gamma$, and a further substitution $n^{-1}a\mu n=a\mu\nu$ yields the formula
\begin{multline*}
J_M^T(a\gamma,f)=|D^M(a\gamma)|_S^{1/2}\delta_{P_1}(a\gamma)^{1/2} \\
\int_{\bK_S}\int_{O_\gamma(M(F_S))}\int_{N_{P_1}(F_S)}f(k^{-1}a\mu\nu k)\,v_M(n,T)\,\d\nu\,\d\mu\,\d k,
\end{multline*}
where $n$ is to be regarded as a function of~$\nu$ and~$a\mu$.
Recall that the weight factor was obtained from the $(G,M)$-family of functions $v_P(\lambda,n,T)$. A modified $(G,M)$-family is given by
\[
w_P(\lambda,a,\mu\nu,T)=v_P(\lambda,n,T)
\prod_{\beta}r_{\beta}(\lambda,\gamma,a),
\]
where $\beta$ runs through the reduced roots of $(N_P/(N_P\cap N_{P_1}),A_M)$ and
\[
r_\beta(\lambda,\gamma,a)
=|a^\beta|_S^{\lambda(\beta^\vee_\gamma)/2}|1-a^{-\beta}|_S^{\lambda(\beta^\vee_\gamma)}
=|(a^\beta-1)(1-a^{-\beta})|^{\lambda(\beta^\vee_\gamma)/2}.
\]
Here $\beta^\vee_\gamma\in\a_M$ is a coroot scaled by a nonnegative factor in such a way that the limits $w_P(\lambda,1,\mu\nu)$ (and their analogues for Levi subgroups $L\in\L(M)$ in place of~$G$) exist and are nonzero for generic $\lambda$ and $\nu$. With the corresponding function
\begin{equation}\label{2e1}
w_M(1,\mu\nu,T)=\lim_{\lambda\to 0}\sum_{P\in\P(M)} \frac{w_P(\lambda,1,\mu\nu,T)}{\theta_P(\lambda)} 
\end{equation}
one now has
\begin{multline}\label{defJM''}
J_M^T(\gamma,f)=|D^M(\gamma)|_S^{1/2}\delta_{P_1}(\gamma)^{1/2} \\
\int_{\bK_S}\int_{O_\gamma(M(F_S))}\int_{N_{P_1}(F_S)}f(k^{-1}\mu\nu k)\,w_M(1,\mu\nu,T)\,\d\nu\,\d\mu\,\d k.
\end{multline}

Arthur actually defines $r_\beta$ differently, namely as $|a^\beta-a^{-\beta}|_S^{\lambda(\beta^\vee_\gamma)}$. Since this produces additional constants in explicit formulas, we have adopted the present modified definition. It also has the property~\cite[(3.5)]{Arthur5}. Thus, if we define a $(G,M)$-family as in \cite[(5.1)]{Arthur5} by
\[
r_P(\lambda,\gamma,a)=\prod_\beta r_\beta(\lambda/2,\gamma,a),
\]
where $\beta$ ranges over the reduced roots of $(P,A_M)$, then the proof of~\cite[Lemma~5.3]{Arthur5} applies (slightly generalized by the inclusion of the truncation parameter~$T$) and shows that the formulas \eqref{defJM'} and \eqref{defJM''} agree.

The version of formula~\eqref{defJM''} given on top of~\cite[p.~256]{Arthur5} looks more complicated because it is intermingled with descent to the unipotent case. This descent is based on the following fact. If $\gamma$ has semisimple part $\sigma$ and unipotent part~$u\in G_\sigma(F)$, then  $r_P(\lambda,\gamma,a)=r_{P_\sigma}(\lambda_\sigma,u,a)$ for $a\in A_M(F)$, where $\lambda_\sigma\in\a_{M_\sigma,\C}^*$ is the image of $\lambda\in\a_{M,\C}^*$ under the natural map induced by the embedding $A_{M_\sigma}\to M$ (cf. \cite[(5.1)]{Arthur5}). There is also a global version of this descent introduced in~\cite{Arthur4}. Before stating it, we have to recall another form of the geometric side of the trace formula.

\subsection{The coarse geometric expansion}
\label{coarse}

Two elements of $G(F)$ are called $\mathcal O$-equivalent if their semisimple parts are $G(F)$-conjugate. Let $\mathcal O=\mathcal O^G$ denote the set of $\mathcal O$-equivalence classes in $G(F)$. For any parabolic subgroup~$P$, any $\fo\in\mathcal O$ and $f\in C_c^\infty(G(\A)^1)$, we define the partial kernel function
\[ K_{P,\fo}(g,h)= \sum_{\gamma\in M_P(F)\cap\fo} \int_{N_P(\A)}
f(g^{-1}\gamma n h) \d n.  \]
Since $P(F)\cap\fo=(M(F)\cap\fo)N_P(F)$, it does not depend on the choice of~$M_P$. If we replace $K_P$ by $K_{P,\fo}$ in the definition~\eqref{defJ} of~$J^T(f)$, we obtain a distribution $J_\fo^T(f)$, and the coarse geometric expansion is
\[ J^T(f)=\sum_{\fo\in\mathcal O}J^T_{\fo}(f).  \]
It was proved in~\cite{Arthur6} that the sum over $\mathcal O$ together with the integral over $G(F)\bsl G(\A)^1$ implicit in $J_\fo^T(f)$ is absolutely convergent. The unipotent elements of $G(F)$ make up an $\mathcal O$-equivalence class, and the corresponding term is denoted by~$J_\mathrm{unip}^T(f)$.

\subsection{Descent to the unipotent case}
\label{desc}

Let $\sigma\in G(F)$ be semisimple and assume that $\sigma$ is $F$-elliptic in~$M$, i.~e., $A_M=A_{M_\sigma}$. Let $\bK_\sigma$ be a maximal compact subgroup of~$G_\sigma(\A)$ with the analogous properties as~$\bK$. As always, we consider distributions that depend on a truncation parameter. In \cite[Cor.~8.7]{Arthur5} and~\cite[p.~197]{Arthur4}, the components of $T$ resp.~$T_\sigma$ corresponding to nonstandard parabolics are treated incorrectly. In fact, the truncation parameters $T$ for $G$ and $T_\sigma$ for~$G_\sigma$ are independent of each other. For $R\in\F^\sigma(M_\sigma)$ (the analogue of $\F(M)$ with $G$ replaced by~$G_\sigma$), $g\in G_\sigma(\A)$ and $h\in G(\A)$, we define
\begin{multline*}
\Gamma_R(g,h,T_\sigma,T)\\
=\sum_{\substack{S\in\F^\sigma(M_\sigma)\\S\supset R}}
\tau_R^S(H_R(g)-T_{\sigma,R})
\sum_{\substack{Q\in\F(M)\\Q_\sigma=S}}
(-1)^{\dim\a_Q^G}\,\hat\tau(H_Q(gh)-T_Q).
\end{multline*}
Then
\begin{multline*}
\sum_{\substack{P\in\F(M)\\P_\sigma=R}}
(-1)^{\dim\a_P^G}\,\hat\tau(H_P(gh)-T_P) \\
=\sum_{\substack{S\in\F^\sigma(R)\\S\supset R}}
(-1)^{\dim\a_R^S}\,
\hat\tau_R^S(H_R(g)-T_{\sigma,R})\Gamma_S(g,h,T_\sigma,T).
\end{multline*}
This follows from the equivalence of $(4.1)^*$ and $(4.2)^*$ in~\cite{Arthur4} applied to the $(G,M)$-orthogonal set that is given for $P\in\P(M)$ by
\[
Y_P^{T_\sigma,T}(g,h)=(H_{P_\sigma}(g)-T_{P_\sigma})
-(H_P(gh)-T_P).
\]
Thus, if we denote the integral of $\Gamma_R(a,h,T_\sigma,T)$ over $a\in(R(\A)\cap G_\sigma(\A)^1)/R(\A)^1$ by $v_R'(h,T_\sigma,T)$ and set
\[
f_{R,h}^{T_\sigma,T}(m)=\delta_R(m)^{1/2}\int_{\bK_\sigma}\int_{N_R(F_S)}
f(h^{-1}\sigma k^{-1}mnkh)v_R'(kh,T_\sigma,T)\,\d n\,\d k,
\]
then Lemma~6.2 in~\cite{Arthur4} takes the following form. If $\fo$ is an $\fO$-equivalence class in~$G(F)$ containing the semisimple element~$\sigma$ and $\iota^G(\sigma)=G_{\sigma,+}(F)/G_\sigma(F)$, then for $f\in C_c^\infty(G(\A)^1)$ we have
\[
J_\fo^T(f)=|\iota^G(\sigma)|^{-1}\int_{G_\sigma(\A)\bsl G(\A)}
\sum_{R\in\F^\sigma(M)}|W_0^{M_R}||W_0^{G_\sigma}|^{-1}
J_{\mathrm{unip}}^{M_R,T_\sigma^R}\bigl(f_{R,h}^{T_\sigma,T}\bigr)\,dh
\]
where $T^R$ is a truncation parameter for $M_R$ derived from $T$ in
the obvious way.

Now we fix $\gamma\in G(F)$ with Jordan decomposition $\gamma=\sigma u$ and return to the local situation. Let again $S$ be a finite subset of~$\Sigma$ containing~$\Sigma_\infty$ and set $\bK_{\sigma,S}=\bK_\sigma\cap G_\sigma(F_S)$. For a parabolic subgroup $R$ of~$G_\sigma$, let $K_R(g)$ denote the second component of $g\in G_\sigma(F_S)$ in the decomposition $G_\sigma(F_S)=R(F_S)\bK_{\sigma,S}$. Then, for $P\in\P(M)$ and $h\in G(F_S)$, we have
\[
Y_P^{T_\sigma,T}(g,h)=
T_P-H_P(K_{P_\sigma}(g)h)-T_{P_\sigma}.
\]
These points form a $(G,M)$-orthogonal set that can replace
\[
\{-H_P(K_{P_\sigma}(g)h)\mid P\in\P(M)\}
\]
in the formulas in~\cite[p.~250/251]{Arthur5}.
The resulting version of~\cite[Cor.~8.4]{Arthur5} then reads
\[
v_L(gh,T)=\sum_{R\in\F^\sigma(M_\sigma)} v_{L_\sigma}^{R,T_\sigma}(g)
v_R'(K_R(g)h,T_\sigma,T),
\]
and \cite[Cor.~8.7]{Arthur5} becomes
\[
J_M^T(\gamma,f)=|D(\gamma)|^{1/2}\int_{G_\sigma(F_S)\bsl G(F_S)}
\sum_{R\in\F^\sigma(M_\sigma)}
J_{M_\sigma}^{M_R,T_\sigma}\bigl(u,f_{R,h}^{T_\sigma,T}\bigr)\,\d h
\]
for $f\in C_c^\infty(G(F_S)^1)$.

Now the proof of \cite[Lemma~7.1]{Arthur4} is correct and applies to general truncation parameters. The dependence on~$T_\sigma$ drops out, as the left-hand side does not depend on it.
If we set
\[ \varepsilon^G(\sigma)=\begin{cases} 1 & \text{if $A_{G_\sigma}=A_G$} ,  \\
0 &  \text{otherwise} ,  \end{cases}  \]
we have the descent formula for coefficients
\[
a^G(S,\gamma)=\varepsilon^G(\sigma)\, |\iota^G(\sigma)|^{-1} \, \sum_{ \{u\, | \, \sigma u\sim\gamma \} }a^{G_\sigma}(S,u),
\]
where the sum is taken over all $u\in (\mathcal{U}_{G_\sigma}(F))_{G_\sigma,S}$ such that $\sigma u$ is $(G,S)$-equivalent to $\gamma$ (cf.~\cite[Theorem~8.1]{Arthur4}).

\subsection{Mean value formula for $\SL(n)$}\label{2s6}
Assume that $n$ is a natural number larger than $1$.
Let $G$ denote the algebraic group $\SL(n)$ over $F$ and let $V$ denote the vector space of $1\times n$ matrices.
Let $\d g$ be a Haar measure on $G(\A)$.
We denote by $\d x_j$ $(1\leq j \leq n)$ the Haar measure on $\A$ defined in Section \ref{2s1}.
Then, $\d x=\d x_1 \, \d x_2 \, \cdots \, \d x_n$ is a Haar measure on $V(\A)$ as $x=\begin{pmatrix}x_1&x_2&\cdots&x_n\end{pmatrix}$.
We define a right action of $G$ on $V$ by the matrix multiplication $(x,g)\mapsto xg$ $(x\in V, \, g\in G)$.
We set $V^0$ equal to the scheme-theoretic complement of $0$ in $V$. For any $F$-algebra~$B$, a point $x\in V(B)$ belongs to $V^0(B)$ if and only if there exists an $F$-linear function $l$ on $V$ such that $l(x)=1$.
It is known that
\begin{equation}\label{2e4}
\int_{G(F)\bsl G(\A)}\sum_{x\in V^0(F)} f( x\cdot g  )\, \d g = \vol_G \,  \int_{V(\A)} f(x)\, \d x
\end{equation}
where $f\in C_c^\inf(V(\A))$ (cf. \cite{Siegel2,Weil,Ono}), and that the complement of $V^0(\A)$ in $V(\A)$ has measure zero.
This formula is called a mean value formula and will be applied to the proof of Theorem \ref{6t3}.

\section{A formula of Labesse and Langlands}\label{3s}

In this section, we generalize the formula \cite[(5.11)]{LL} and discuss its application to the trace formula.
It can be regarded as a generalization of the non-adelic formula \cite[Theorem 1.2]{HH}, which was applied to dimension formulas for spaces of Hilbert cusp forms.
A further generalization was given by Wright~\cite[Theorem I.1]{Wright} in the study of prehomogeneous zeta functions.
However, the original method of Labesse and Langlands gives not only an alternative proof for \cite[Theorem I.1]{Wright} but also of more general formulas.

\subsection{A formula of Labesse and Langlands}

Let $F$ be an algebraic number field and $E$ a finite algebraic extension of $F$.
For each $w\in\Sigma_E$, we denote by $|\; |_w$ the normal valuation of $E_w$.
For $v\in\Sigma_F$ and $w\in\Sigma_E$, the notation $w|v$ means that $F_v$ is the completion of $F$ via $w$.
In this case, $|x|_w=|N_{E_w/F_v}(x)|_v$ for $x\in E_w$.

We fix a natural number $n$ and consider the $n$-fold direct sum $A=E^n$ of algebras and the $n$-fold direct
product $M=R_{E/F}(\mathbb{G}_m)^n$ of algebraic groups.
Hence, $A^\times=M(F)$.
A Haar measure $\d m$ on $M(\A)$ is defined by $\d m=\prod_{j=1}^n \d m_j$, where $\d m_j$ is the Haar measure on $\A_E^\times$ given in Section \ref{2s1}.
We identify $t\in\Rp$ with $(y_w)\in\A_E^\times$, $y_w=t^{1/[E:\Q]}$ $(\forall w\in\Sigma_{E,\inf})$, $y_w=1$ $(\forall w\in\Sigma_{E,\fin})$.
For each $k$, we denote by $\d^\times t_k$ the Haar measure on $(\R^0)^\times\subset \A_E^\times$ of Section \ref{2s1}.
We define the Haar measure $\d^\times t$ on $A^+_M=((\R^\times)^0)^n$ by $\d^\times t=\prod_{k=1}^n \d^\times t_k$.
From this we have a normalization of $\d^1 m$.
Let $V$ denote the vector space with the $F$-structure $V(F)=A$.
For $s=(s_1,s_2,\dots,s_n)\in\C^n$ and $m=(m_1,m_2,\dots,m_n)\in M(\A)$, we set
\[  |m|_A^s = \prod_{j=1}^n |m_j|_{\A_E}^{s_j} \quad \text{and} \quad  |m|_A = \prod_{j=1}^n |m_j|_{\A_E}.  \]
A zeta integral $\zeta_A(\phi,s,\chi)$ for $A$ is defined by
\[ \zeta_A(\phi,s,\chi)=\int_{M(\A)}\phi(m)\, |m|_A^s \, \chi(m) \, \d m   , \]
where $s\in\C^n$, $\phi\in \cS(V(\A))$, and $\chi$ is a character of $M(\A)^1/M(F)$.
We easily see that $\zeta_A(\phi,s,\chi)$ converges absolutely for $\mathrm{Re}(s_k)>1$ $(k=1,2,\dots, n)$ and can be meromorphically continued to all $s\in\C^n$.

Let $H$ be a connected reductive algebraic group over $F$.
Fix a Haar measure $\d h$ on $H(\A)$ and let $\d^1 h$ be the Haar measure on $H(\A)^1$ according to Section \ref{2s3}.
Suppose that we are given an $F$-rational homomorphism $\iota:H\to M$, which will be denoted by
\[ h^\iota=(h_1,h_2,\dots,h_n)   \]
for $h\in H$. We also consider the $F$-rational homomorphism $\kappa \, : \, M\to M$ defined by
\[ m^\kappa=(m_1^{\kappa_1},m_2^{\kappa_2},\dots,m_n^{\kappa_n}),  \]
where $\kappa_j\in\N$ $(1\leq j\leq n)$ are given numbers.
The algebraic group
\[ G=M\times H \]
over $F$ acts on $V$ by
\[ (m,h)\cdot x=m^\kappa h^\iota x= (m_1^{\kappa_1}h_1x_1,m_2^{\kappa_2}h_2x_2,\dots,m_n^{\kappa_n}h_nx_n),   \]
where $m=(m_1,m_2,\dots,m_n)\in M$, $h\in H$ and $x=(x_1,x_2,\dots,x_n)\in V$.
The pair $(G,V)$ is a prehomogeneous vector space over $F$.
We define a zeta integral by
\[ \zeta_A^{\kappa,\iota}(\phi,s)=\int_{M(\A)/M(F)}\int_{H(\A)^1/H(F)}  |m^\kappa|^s_A \sum_{x\in A^\times} \phi(m^\kappa h^\iota x)\, \d^1 h \, \d m , \]
where $s\in\C^n$ and $\phi\in\cS(V(\A))$.
If $\mathrm{Re}(s_k)>1$ $(k=1,2,\dots,n)$, then $\zeta_A^{\kappa,\iota}(\phi,s)$ is absolutely convergent.
The following is a generalization of the formula \cite[(5.11)]{LL}.
\begin{thm}\label{3t}
Let $\phi\in\cS(V(\A))$.
The zeta integral $\zeta_A^{\kappa,\iota}(\phi,s)$ can be meromorphically continued to the whole of~$\C^n$.
It satisfies the formula
\[ \zeta_A^{\kappa,\iota}(\phi,s)=\Big( \prod_{j=1}^n\kappa_j^{-1}\Big) \vol_H\sum_\chi \zeta_A(\phi,s,\chi),   \]
where $\chi$ runs over all characters of $M(\A)^1/M(F)$ which are trivial on $\kappa(M(\A)^1)$ and $\iota(H(\A)^1)$.
\end{thm}
\begin{proof}
We will mimick the proof of \cite[(5.11)]{LL}.
Assume $\mathrm{Re}(s_k)>1$ $(k=1,2,\dots,n)$.
Then we have
\begin{multline*}
\zeta_A^{\kappa,\iota}(\phi,s)=\\
\sum_{x\in M(F)/\kappa(M(F))\iota(H(F))}\int_{(M(\A)\times H(\A)^1)/I(F)}\phi((m,h)\cdot x) \, |m^\kappa|^s_A \, \d^\times m\, \d^1 h,
\end{multline*}
where $I(F)=\{(m,h)\in G(F) \, | \, m^\kappa h^\iota=1\}$.
By the Poisson summation formula for $M(F)/\kappa(M(F))\iota(H(F))$ and $M(\A)/\kappa(M(F))\iota(H(F))$, we have
\begin{multline*}
\zeta_A^{\kappa,\iota}(\phi,s)=  \vol_M^{-1} \sum_{\chi\in (M(F)/\kappa(M(F))\iota(H(F)))^\perp} \int_{M(\A)^1/\kappa(M(F)) \iota(H(F))} \d^1 m' \\
\int_{(M(\A)\times H(\A)^1)/I(F)} \d^\times m\, \d^1 h \; \phi(m^\kappa h^\iota m') \, |m^\kappa|_A^s\, \chi(m')  .
\end{multline*}
By change of variable we get
\begin{multline*}
\zeta_A^{\kappa,\iota}(\phi,s)= \Big( \prod_{j=1}^n\kappa_j^{-1}\Big)\, \vol_M^{-1} \sum_{\chi\in (M(F)/\kappa(M(F))\iota(H(F)))^\perp} \\
\int_{M(\A)/\kappa(M(F))\iota(H(F))}  \d m' \int_{(M(\A)^1\times H(\A)^1)/I(F)}\d^1m\, \d^1 h \\
\phi(m^\kappa h^\iota m') \, |m'|_A^s\, \chi(m') .
\end{multline*}
Since the map $G(F)/I(F)\to \kappa(M(F))\iota(H(F))$ $((m,h)\mapsto m^\kappa h^\iota)$ is bijective, we find
\begin{multline*}
\zeta_A^{\kappa,\iota}(\phi,s)= \Big( \prod_{j=1}^n\kappa_j^{-1}\Big)\, \vol_M^{-1} \sum_{\chi\in (M(F)/\kappa(M(F))\iota(H(F)))^\perp} \zeta_A(\phi,s,\chi) \\
  \int_{H(\A)^1/H(F)}\chi(h^\iota)^{-1} \d^1h\,  \int_{M(\A)^1/M(F)}\chi(m^\kappa)^{-1} \d^1 m 
\end{multline*}
by change of variable.
This implies the formula in the theorem for $\mathrm{Re}(s_k)>1$, $k=1,2,\dots,n$.
For a fixed test function $\phi$, the right hand side is a finite sum and can therefore be meromorphically continued to the whole $s$-space.
\end{proof}

\subsection{Modified zeta integral}\label{3s2}

We identify $\a_M$ with $\R^n$ such that
\[
H_M((m_1,\dots,m_n))=(\log|m_1|_{\A_E},\dots,\log|m_1|_{\A_E})
\]
for $m_i\in\A_{E}^\times$.
We denote the set of subsets of $\{1, \, 2 , \, \dots \, , n\}$ by $\mathcal{Y}_n$.
Let $T=(T_1,T_2,\dots,T_n)\in \a_M$ and $Y\in \mathcal Y_n$.
We set $\overline{Y}=\{1,2,\dots,n\}-Y$.
Let $\tau_{T,Y}(t_1,t_2,\dots,t_n)$ denote the characteristic function of
\[
\{ (t_1,t_2,\dots,t_n)\in \R^n \, | \,    t_k<-T_k \; (\forall k\in Y)  \} .
\]
Let $\d v$ denote the Haar measure on $\A_E$ normalized by $\int_{\A_E/E}\d v=1$.
We set
\[
\d v_Y=\prod_{k\in Y} \d v_k \quad \text{for $v=(v_1,v_2,\dots,v_n)\in V(\A)$}, 
\]
\[
A_Y=\{ (x_1,x_2,\dots,x_n)\in A \, | \, x_k=0 \; (\forall k \in \overline{Y})\}   \}  , 
\]
and
\[
V_Y(\A)=\{ (x_1,x_2,\dots,x_n)\in V(\A) \, | \, x_k=0 \; (\forall k\in \overline{Y}) \}.
\]
For $v_Y=(v_1,v_2,\dots,v_n)\in V_Y(\A)$, we set $\d v_Y=\prod_{j\in Y}\d v_j$.
We define a modified zeta integral $\zeta_A^{\kappa,\iota}(\phi,s,T)$ by
\begin{multline*}
 \zeta_A^{\kappa,\iota}(\phi,s,T)= \int_{M(\A)/M(F)}\d m \, \int_{H(\A)^1/H(F)}  \d^1 h \;  |m^\kappa|^s_A  \\
 \sum_{Y\in\mathcal Y_n} (-1)^{|Y|}  \sum_{x_{\overline{Y}} \in A_{\overline{Y}}^\times}\int_{V_Y(\A)} \phi(m^\kappa h^\iota (x_{\overline{Y}}+v_Y))  \, \d v_Y \, \tau_{T,Y}(H_M(m)). 
\end{multline*}
An additive character $\psi_Y$ on $V_Y(\A)$ is defined similarly to Section \ref{2s1}.
For $\phi\in\cS(V(\A))$ and $Y\in\mathcal Y_n$, we set
\[ \hat\phi^Y(x_{\overline{Y}}+x_Y)=\int_{V_Y(\A)}  \phi(x_{\overline{Y}}+y_Y)\, \psi_Y(x_Yy_Y)\, \d y_Y \]
where $x_Y\in V_Y(\A)$ and $x_{\overline{Y}}\in V_{\overline{Y}}(\A)$.
Let $1_n$ denote the element $(1,1,\dots,1)$ of $\C^n$.
Then, the Poisson summation formula reads
\[ \sum_{ x_Y\in A_Y }\phi(a(x_{\overline{Y}}+x_Y)) =  \sum_{ x_Y\in A_Y } \hat\phi^Y(ax_{\overline{Y}}+a^{-1}x_Y) \, |a_Y|_A^{-1_n},  \]
where $a\in M(\A)$.
From this formula we obtain
\begin{multline}\label{3e1}
\sum_{ x_Y\in A_Y^\times }\phi(ax_Y)= \\
 \sum_{Y''\subset Y'\subset Y}(-1)^{|Y-Y'|} \sum_{x_{Y''}\in A_{Y''}^\times } \hat\phi^{Y'}(0_{Y-Y''}+a^{-1}x_{Y''})\, |a_{Y'}|_A^{-1_n} .
\end{multline}
We set
\[ D_{Y,T}^+=\{    (t_1,\dots,t_n)\in ((\R^\times)^0)^n \, | \,    t_k>e^{-T_k} \; (\forall k \in Y) \, , \; t_k=1 \; (\forall k\in \overline{Y})  \}, \]
\[ D_{Y,T}^-=\{    (t_1,\dots,t_n)\in ((\R^\times)^0)^n \, | \,    t_k<e^{-T_k} \; (\forall k \in Y) \, , \; t_k=1 \; (\forall k\in \overline{Y})  \}, \]
and $\d^\times t_Y=\prod_{k\in Y}\d^\times t_k$.
By using (\ref{3e1}), we have
\begin{multline*}
 \zeta_A^{\kappa,\iota}(\phi,s,T) = \int_{G(\A)^1/G(F)}\d^1 m \, \d^1 h  \sum_{Y\in\mathcal Y_n} (-1)^{|Y|} \sum_{Z\subset Y} \int_{D_{\overline{Y},T}^+\cup D_{Z,T}^-} \d^\times t_{\overline{Y}\cup Z} \\
\sum_{x_{\overline{Y}} \in A_{\overline{Y}}^\times \, , \; x_{Z} \in A_Z^\times} |t_{\overline{Y}\cup Z}^\kappa|^s_A  \sum_{Z \subset Y} \hat\phi^{Z}( m^\kappa h^\iota t^\kappa x_{\overline{Y}} + (m^{-1})^\kappa(h^{-1})^\iota (t^{-1})^\kappa x_{Z}+0_{Y-Z})  \\
\times (-1)^{|Z|} \, |t^\kappa_Z|_A^{-1_n} \prod_{k\in Y-Z}\frac{e^{-\kappa_ks_kT_k}}{\kappa_k s_k}  .
\end{multline*}
Thus, we find that $\zeta_A^{\kappa,\iota}(\phi,s,T)$ converges absolutely for $\mathrm{Re}(s_k)>0$ $(k=1,2,\dots,n)$.
We also have
\begin{multline*}
 \zeta_A^{\kappa,\iota}(\phi,s) = \int_{G(\A)^1/G(F)}\d^1 m \, \d^1 h \\
\sum_{Y\in\mathcal Y_n} (-1)^{|Y|} \sum_{Z\subset Y} \int_{D_{\overline{Y},T}^+\cup D_{Z,T}^-} \d^\times t_{\overline{Y}\cup Z}  \sum_{x_{\overline{Y}} \in A_{\overline{Y}}^\times \, , \; x_{Z} \in A_Z^\times} |t_{\overline{Y}\cup Z}^\kappa|^s_A \\
 \sum_{Z\subset  U \subset Y} \hat\phi^{U}( m^\kappa h^\iota t^\kappa x_{\overline{Y}} + (m^{-1})^\kappa (h^{-1})^\iota (t^{-1})^\kappa x_{Z}+0_{Y-Z})  \\
\times  (-1)^{|U|}\, |t^\kappa_Z|_A^{-1_n} \prod_{k\in Y-U}\frac{e^{-\kappa_ks_kT_k}}{\kappa_k s_k} \times \prod_{j\in U-Z} \frac{e^{-(\kappa_ks_k-\kappa_k)T_k}}{\kappa_j s_j-\kappa_j} .
\end{multline*}
By the above two equalities we easily find that
\begin{multline*}
\lim_{s\to 1_n} \prod_{k=1}^n \frac{\partial}{\partial s_k}\prod_{j=1}^n (s_j-1) \; \zeta_A^{\kappa,\iota}(\phi,s) \\
=\sum_{Y\in\mathcal Y_n} (-1)^{|\overline{Y}|}\,  \Big(\prod_{l\in\overline{Y}}T_l\Big) \, \zeta_{A_Y}^{\kappa_Y,\iota_Y}(\phi_Y,1_{|Y|},T_Y)
\end{multline*}
where $\kappa_Y=(\kappa_k)_{k\in Y}$, $\iota_Y(h)=(h_k)_{k\in Y}$, $\phi_Y(x_Y)=\int_{V_{\overline{Y}}(\A)}\phi(x_Y+x_{\overline{Y}})\, \d x_{\overline{Y}}$, and $T_Y=(T_k)_{k\in Y}$.
Hence, we obtain
\begin{equation}\label{3e2}
\zeta_A^{\kappa,\iota}(\phi,1_n,T)= \sum_{Y\in\mathcal Y_n} \Big(\prod_{l\in\overline{Y}}T_l\Big) \, \lim_{s_Y\to 1_{|Y|}}\prod_{k\in Y} \frac{\partial}{\partial s_k} \prod_{j\in Y}^n (s_j-1) \; \zeta_{A_Y}^{\kappa_Y,\iota_Y}(\phi_Y,s_Y) ,
\end{equation}
where $s_Y=(s_k)_{k\in Y}$.
In particular, if we substitute $\mathcal T_0=(0,0,\dots,0)\in\a_M$ for $T$, we get
\begin{equation}\label{3e3}
\zeta_A^{\kappa,\iota}(\phi,1_n,\mathcal T_0)=\lim_{s\to 1_n} \prod_{k=1}^n \frac{\partial}{\partial s_k}\prod_{j=1}^n (s_j-1) \; \zeta_A^{\kappa,\iota}(\phi,s).
\end{equation}

\subsection{Orbits and Decompositions}

Let $S$ be a finite subset of $\Sigma_F$, $\chi$ a character of $M(\A)^1/M(F)$, and $\phi\in \cS(V(\A))$.
We set
\[
S_E=\bigcup_{v\in S} \{ \, w\in \Sigma_E \,\big| \, w|v \, \}.
\]
There exist characters $\chi_1$, $\chi_2,\dots,\chi_n$ on $\A_E^1/E^\times$ such that $\chi=\prod_{k=1}^n\chi_k$.
We set
\[  \chi_k=\prod_{w\in\Sigma_E}\chi_{k,w}, \quad \chi_v=\prod_{k=1}^n\prod_{w|v} \chi_{k,w},  \quad  \chi_S=\prod_{k=1}^n\prod_{w\in S_E} \chi_{k,w}  , \]
and
\[ L_A^S(s,\chi)=\prod_{k=1}^n L_E^{S_E}(s_k,\chi_k) \]
where $L_E^{S_E}(s,\chi)$ means $L^{S_E}(s,\chi)$ defined over $E$ (cf. \eqref{lft}).
Let $\phi_{0,v}$ denote the characteristic function of $V(\fO_v)$.
We may assume the following condition on~$S$.
\begin{cond}\label{cond1}
The finite set $S$ satisfies $S\supset \Sigma_\inf$, there exists a function $\phi_S\in\cS(V(F_S))$ such that $\phi=\phi_S \prod_{v\not\in S} \phi_{0,v}$, and $V(\fO_{F,v})= \bigoplus_{w|v} (\fO_{E,w})^{\oplus n} $ for any $v\not\in S$.
\end{cond}
This is clearly satisfied for $S$ large enough. In this case,
\begin{multline*}
\int_{V(\fO_v)}|m_v|_A^s \, \chi_v(m_v) \, \d m_v \\
=\begin{cases} \prod_{k=1}^n \prod_{w|v}L_{E,w}(s_k,\chi_{k,w})   & \text{if $\chi_v$ is unramified,}  \\  0 & \text{otherwise}  \end{cases}
\end{multline*}
for any $v\not\in S$, where $L_{E,w}(s,\chi_w)$ stands for $L_w(s,\chi_w)$ defined over $E_w$ (cf. \eqref{lft}).
If $\chi$ is unramified outside~$S$, then we get
\[
\zeta_A(\phi,s,\chi) =\int_{M(F_S)}\phi_S(m_S)|m_S|_A^s\, \chi_S(m_S)\, \d m_S  \times L_A^S(s,\chi) .
\]
For $Y\in\mathcal Y_n$, we set
\[
\chi_S^Y=\prod_{k\in \overline{Y}}\prod_{w\in S_E} \chi_{k,w} ,\quad |m|_{Y,S}=\prod_{k\in Y} |m_k|_{S_E} , \quad \fc_E^Y(S,\chi)=\prod_{k\in\overline{Y}}\fc_E(S_E,\chi_k), 
\]
\[ \mathfrak K_{Y,S}=\left\{ \chi=\prod_{k=1}^n \chi_k  \Big| \begin{array}{l} \text{$\chi|_{\kappa(M(\A)^1)}=1$, $\chi|_{\iota(H(\A)^1)}=1$},\\ \text{$\chi_k=\trep_E \, (\forall k\in Y)$,} \\ \text{$\chi_{k,w}$ is unramified for every $w\not\in S_E$} \\ \text{and every $k=1,2,\dots,n$} \end{array} \right\} . \]
Note that $|\mathfrak K_{Y,S}|$ is a finite set.
The following theorem is derived from Theorem \ref{3t} and (\ref{3e3}).
\begin{thm}\label{3t3}
Fix a test function $\phi\in\cS(V(\A))$.
Assume that $S$ satisfies Condition \ref{cond1}.
If we substitute $\mathcal T_0=(0,0,\dots,0)\in\a_M$ into (\ref{3e2}) as a polynomial of $T$, then we have
\begin{align*}
 \zeta_A^{\kappa,\iota}(\phi,1_n,\mathcal T_0)=& \frac{\vol_H}{\prod_{j=1}^n\kappa_j} \sum_{Y\in\mathcal Y_n} (c_E^S)^{|Y|} \sum_{\chi\in\mathfrak K_{Y,S}} \fc_E^Y(S,\chi) \\
& \times \int_{M(F_S)}\phi_S(m_S) \, |m_S|_A \, \chi_S^Y(m_S) \Big( \prod_{k\in Y}\log|m_k|_{S_E}\Big) \, \d m_S .
\end{align*}
\end{thm}

For $Y\in\mathcal Y_n$, we set
\[ M_Y=\{m=(m_1,m_2,\dots,m_n)\in M \, |\, m_j=1 \, \; (\forall j\in\overline{Y})\} ,    \]
and
\[ \mathscr O_{F,S}^Y=M_{\overline{Y}}(F)/\{ M_{\overline{Y}}(F)\cap (\kappa(M(F_S))\, \iota(H(F_S)))\}.  \]
Since $\mathscr O_{F,S}^Y=M_{\overline{Y}}(F_S)/\{ M_{\overline{Y}}(F_S)\cap (\kappa(M(F_S))\, \iota(H(F_S)))\}$, we deduce the following theorem from Theorem \ref{3t3}.
\begin{thm}\label{3t4}
Fix a test function $\phi\in\cS(V(\A))$.
Assume that $S$ satisfies Condition \ref{cond1}.
If we substitute $\mathcal T_0=(0,0,\dots,0)\in\a_M$ into (\ref{3e2}) as a polynomial of $T$, then we have
\begin{multline*}
\zeta_A^{\kappa,\iota}(\phi,1_n,\mathcal T_0)= \frac{\vol_H}{\prod_{j=1}^n\kappa_j}\sum_{Y\in\mathcal Y_n}  \sum_{\gamma\in\mathscr O^Y_{F,S}} (c_E^S)^{|Y|} \sum_{\chi\in \mathfrak K_{Y,S}} \chi_S^Y(\gamma) \, \fc_E^Y(S,\chi) \\ 
 \times \int_{ M_Y(F_S)\{ \gamma \, \kappa(M_{\overline{Y}}(F_S))\, \iota(H(F_S))\} }\phi_S(m_S) \, |m_S|_A \Big( \prod_{k\in Y}\log|m_k|_{S_E}\Big)\, \d m_S  .
\end{multline*}
\end{thm}

The contribution of regular unipotent elements of $\SL(n,F)$ to the trace formula for $n\geq 3$ seems to be related to $\zeta_A^{\kappa,\iota}(\phi,1_n,\mathcal T_0)$ with $A=F^{n-1}$, $H=\GL(1)^{n-2}$,
\begin{align*}
\kappa(m_1,m_2,\dots,m_{n-1})&=(m_1^n,m_2^n,\ldots,m_{n-1}^n), \\ \iota(h_1,h_2,\dots,h_{n-2})&=(h_1,h_2,\dots,h_{n-2},h_1h_2^2\dots h_{n-2}^{n-2}).
\end{align*}
Hence, the value $L^S(1,\chi)\, L^S(1,\chi^2)\dots L^S(1,\chi^{n-1})$ should appear in such coefficients for any character $\chi$ on $\A^1/F^\times$ satisfying that $\chi^n=\trep_F$ and $\chi_v$ unramified $(\forall v\not\in S)$.
In Appendix \ref{appen2}, we will apply Theorem \ref{3t4} to the case $\SL(3)$.

\subsection{Application}\label{3s4}

We use Theorem \ref{3t4} to express some of the coefficients of unipotent orbital integrals in terms of $\fc_F(S,\chi)$ and $\chi_S$.
In Propositions \ref{5pm0}, \ref{5pm1}, and \ref{5pm2}, we will give a classification for $\cO$-equivalence classes containing non-semisimple elements in $\GSp(2,F)$.
A similar classification for $\Sp(2,F)$ can be obtained from Lemma~\ref{5l}.
The group $\mathcal G$ in each of the following examples appears as the group $G_\sigma$, where $G=\GSp(2)$ or $\Sp(2)$ and
where $\sigma$ is a semi-simple element of such an $\cO$-equivalence class.
We do not mention all possibilities for such centralizers~$G_\sigma$ because the other ones can be treated similarly.
In each case, $\bK$ denotes a suitable maximal compact subgroup of $\mathcal G(\A)$ and $\mathcal M_0$ denotes a minimal Levi subgroup of $\mathcal G$.
Let $T_0$ denote the point in $\a_0$ defined in \cite[Section~2]{Arthur2}.
Fix a finite set $S\supset \Sigma_\inf$ and assume that $f\in C_c^\inf(\mathcal G(F_S))$.

\begin{exa}\normalfont\label{3ex1}
Let $\mathcal G=\GL(2)$ and $\bK$ the group $\bK_2$ from Section \ref{4s9}.
It is clear that $(\cU_{\mathcal G}(F))_{\mathcal G,S}=\{1, \,  u_1 \}$ in the notation of~(\ref{1e1}).
Let $\mathcal M_0=\{\diag(a,b)\in \mathcal G\}$.
We choose the Haar measures $t^{-1}\d t$ on $\{ \diag(t^{-1},t)\in A_{\mathcal M_0}^+/A_{\mathcal G}^+ \}$ where $\d t$ is the Lebesgue measure on $\R$.
Normalizing a measure on $O_{u_1}(\mathcal G(F_S))$, the orbital integral $J_{\mathcal G}(u_1,f)$ is given by
\[ J_{\mathcal G}(u_1,f)=c_S\, \int_{\bK_S}\int_{F_S} f(k_S^{-1}u_{x_S}k_S) \,  \d x_S \, \d k_S \]
where $\d x_S$ (resp. $\d k_S$) is the Haar measure on $F_S$ (resp. $\bK_S$) defined in Section \ref{2s1} (resp. \ref{2s3}).
We set $\phi(x)=\int_\bK f(k^{-1}u_x k)\, \d k$.
Then, we have
\[ J_{\mathrm{unip}}^{T_0}(f)=\vol_{\mathcal G}f(1)+ \frac{\vol_{\mathcal M_0}}{c_F^{2}} \, \zeta_A^{\kappa,\iota}(\phi,1,\mathcal T_0) \]
where $A=F$, $H=\mathbb G_m$, $\iota=\mathrm{Id}$, and $\kappa_1=2$.
Here, $\mathrm{Id}$ denotes the identity map.
Thus, we have
\[ a^{\mathcal G}(S,u_1) = \frac{\vol_{\mathcal M_0}}{2 c_F}\, \fc_F(S). \]
\end{exa}
In the following examples, the test function $\phi$ is defined as in Example~\ref{3ex1}.
\begin{exa}\normalfont\label{3ex2}
Let $\mathcal G=\SL(2)$, $\mathcal M_0=\{\diag(a,a^{-1})\in \mathcal G \}$, and $\bK$ be the same group as $\bK$ in Section \ref{2s6}.
We take the Haar measure $\d t/t$ on $\{ \diag (t^{-1},t)\in A_{\mathcal M_0}^+ \}$.
We have $(\cU_{\mathcal G}(F))_{\mathcal G,S}=\{1, \,  u_\alpha \, | \, \alpha\in F^\times/(F^\times\cap(F_S^\times)^2) \}$.
The distribution $J_{\mathrm{unip}}^{T_0}(f)$ is equal to $\vol_{\mathcal G}f(1)+ c_F^{-1}\, \vol_{\mathcal M_0}\, \zeta_A^{\kappa,\iota}(\phi,1,\mathcal T_0)$ where $A=F$, $H=\{1\}$, and $\kappa_1=2$.
Now, we set
\[ J_{\mathcal G}(u_\alpha,f)=c_S \, \int_{\bK_S}\int_{\alpha(F_S^\times)^2} f(k_S^{-1}u_{x_S}k_S) \, \d x_S \, \d k_S \quad (\alpha\in F^\times). \]
Hence, we have
\[ a^{\mathcal G}(S,u_\alpha) = \frac{\vol_{\mathcal M_0}}{2\, c_F} \sum_\chi \chi_S(\alpha) \fc_F(S,\chi) \quad (\alpha\in F^\times)\]
where $\chi$ runs over all quadratic characters of $\A^1/F^\times$ unramified outside~$S$.
\end{exa}

In the following examples, for $u\in (\cU_{\mathcal G}(F))_{\mathcal G,S}$, measures $\d\mu$ on $O_u(\mathcal G(F_S))$ are determined by $\d k_S$ and $\d m_S$ similarly to Examples \ref{3ex1} and \ref{3ex2}.
\begin{exa}\normalfont\label{3ex4}
Let $E$ be a quadratic extension of $F$ and $\sigma$ be the non-trivial element in $\mathrm{Gal}(E/F)$.
We set
\[
\mathcal G=\mathrm{U}_{E/F}(1,1)=\Big\{ g\in R_{E/F}(\GL(2)) \, \Big| \,  g\begin{pmatrix}0&1 \\ 1&0 \end{pmatrix}\, \sigma(^t g)= \begin{pmatrix}0&1 \\ 1&0 \end{pmatrix} \Big\}
\]
and
\[
\mathcal M_0=\{ \diag(a^{-1},\sigma(a)) \, | \, a\in R_{E/F}(\mathbb G_m)  \}\subset \mathcal G.
\]
We choose a suitable maximal compact subgroup $\bK$ such that $T_0=0$ and the Haar measure $t^{-1}\d t$ on $\{ \diag(t^{-1},t)\in A_{\mathcal M_0}^+ \}$.
In this case, we have
\[
(\cU_{\mathcal G}(F))_{\mathcal G,S}=\{1, \,  u_\alpha   \, | \, \alpha \in F^\times/(F^\times\cap N_{E/F}(E_{S_E}^\times)) \}
\]
and $J_{\mathrm{unip}}^{T_0}(f)=\vol_{\mathcal G}f(1)+c_F^{-1} \, c_E^{-1} \, \vol_{\mathcal M_0} \,  \zeta_A^{\kappa,\iota}(\phi,1,\mathcal T_0)$ where $A=F$, $H=R_{E/F}(\mathbb G_m)$, $\iota=N_{E/F}$, and $\kappa_1=2$.
Let $P_0=\Big\{\begin{pmatrix}*&* \\ 0&* \end{pmatrix}\in \mathcal G \Big\}$.
Assume that $E=F(\omega)$ and $N_{P_0}(F_v)\cap \bK_v=\{ u_{x\omega} \, | \, x\in\fO_{F,v} \}$ $(\forall v\not\in S)$.
It follows from Theorem \ref{3t4} that
\[ a^{\mathcal G}(S,u_\alpha)= \frac{\vol_{\mathcal M_0}}{2\, c_F} \{ \fc_F(S)+ \chi_{E,S}(\alpha)  \fc_F(S,\chi_E) \, \delta_{E,S} \}, \]
where $\chi_E$ is the character of $\A^1/F^\times$ corresponding to $E$, $\delta_{E,S}=1$ if $\chi_{E,v}$ is unramified for every $v\not\in S$, and $\delta_{E,S}=0$ if there exists a place $v\not\in S$ such that $\chi_{E,v}$ is ramified.
\end{exa}
\begin{exa}\normalfont
Set
\[ \mathcal G=\{ g\in R_{E/F}(\GL(2)) \, | \, \det(g)\in\mathbb G_m \}
\]
and
\[
\mathcal M_0=\{ \diag(a^{-1},ab) \, | \, a\in R_{E/F}(\mathbb G_m) \, , \; b\in\mathbb G_m  \}.
\]
It is possible to choose a suitable maximal compact subgroup $\bK$ such that $T_0=0$.
We choose the Haar measure $t^{-1}\d t$ on $\{ \diag(t^{-1},t)\in A_{\mathcal M_0}^+/A_{\mathcal G}^+\}$.
Then, we have
\[
(\cU_{\mathcal G}(F))_{\mathcal G,S}=\{1, \,  u_\alpha   \, | \, \alpha \in E^\times/(E^\times\cap((E_{S_E}^\times)^2 F_S^\times)) \}
\]
and $J_{\mathrm{unip}}^{T_0}(f)=\vol_{\mathcal G}f(1)+c_F^{-1} \, c_E^{-1} \, \vol_{\mathcal M_0} \, \zeta_A^{\kappa,\iota}(\phi,1,\mathcal T_0)$ where $A=E$, $H=\mathbb G_m$, $\iota=\mathrm{Id}$, and $\kappa_1=2$.
Let $P_0=\Big\{\begin{pmatrix}*&* \\ 0&* \end{pmatrix}\in \mathcal G \Big\}$.
If we assume $N_{P_0}(F_v)\cap \bK_v=\{ u_x \, | \, x\in \oplus_{w|v}\fO_{E,w} \}$ $(\forall v\not\in S)$, then we find that
\[ a^{\mathcal G}(S,u_\alpha)= \frac{\vol_{\mathcal M_0}}{2 \, c_E} \sum_\chi \fc_E(S_E,\chi) \]
where $\chi$ runs over all quadratic characters on $\A_E^1/E^\times$ unramified outside~$S$ and such that $\chi|_{\A_F^1}=1$.
\end{exa}
\begin{exa}\label{3ex10}\normalfont
We set
\[
\mathcal G=\{ (g_1,g_2)\in \GL(2)\times\GL(2) \, | \, \det(g_1)=\det(g_2) \},
\]
\[
\mathcal M_1=\{ (\diag(a^{-1},ab),h)\, | \, a,b\in\mathbb G_m , \; h\in\GL(2), \; b=\det(h) \},\]
\[
\mathcal M_2=\{ (h,\diag(a^{-1},ab))\, | \, a,b\in\mathbb G_m , \; h\in\GL(2) , \; b=\det(h) \},
\]
and
\[
\mathcal M_0=\{ (\diag(a^{-1},ac),\diag(b^{-1},bc)) \, | \, a,b,c\in\mathbb G_m  \}.
\]
A maximal compact subgroup $\bK$ is defined by $\mathcal{G}(\A)\cap(\bK_2\times\bK_2)$, where $\bK_2$ is
given in Section~\ref{4s9}.
We fix the Haar measure $t^{-1}\d t$ on the groups $\{ (\diag(t^{-1},t),I_2)\in A_{\mathcal M_1}^+/A_{\mathcal G}^+\}$ and $\{(I_2,\diag(t^{-1},t))\in A_{\mathcal M_2}^+/A_{\mathcal G}^+\}$ and we also fix the Haar measure $t_1^{-1}t_2^{-1}\d t_1 \, \d t_2$ on the group $\{ (\diag(t_1^{-1},t_1),\diag(t_2^{-1},t_2))\in A_{\mathcal M_0}^+/A_{\mathcal G}^+ \}$.
We set $u_{\alpha,\beta}=(u_\alpha,u_\beta)$.
Test functions $\phi_1$, $\phi_2$, and $\phi_3$ are defined as $\phi_1(x)=\int_{\bK}f(k^{-1}u_{x,0}k)\d k$, $\phi_2(x)=\int_{\bK}f(k^{-1}u_{0,x}k)\d k$, and $\phi_3(x,y)=\int_{\bK}f(k^{-1}u_{x,y}k)\d k$.
We put $A'=F$, $H'=\mathbb G_m$, $\iota'=\mathrm{Id}$, and $\kappa'(m)=m^2$.
We also put $A=F\oplus F$, $H=\mathbb G_m$, $\iota(a)=(a,a)$, and $\kappa(m_1,m_2)=(m_1^2,m_2^2)$.
Then, we get
\[
(\cU_{\mathcal G}(F))_{\mathcal G,S}=\{1, \, u_{1,0}, \, u_{0,1} ,  \, u_{\alpha,1} \, | \, \alpha \in F^\times/(F^\times\cap(F_S^\times)^2)\}
\]
and
\begin{align*}
J_{\mathrm{unip}}^{T_0}(f) = &\, \vol_{\mathcal G}f(1)+ \frac{\vol_{\mathcal M_1}}{c_F^2}\zeta_{A'}^{\kappa',\iota',H'}(\phi_1,1,\mathcal T_0) \\
&+ \frac{\vol_{\mathcal M_2}}{c_F^2}\zeta_{A'}^{\kappa',\iota',H'}(\phi_2,1,\mathcal T_0) + \frac{\vol_{\mathcal M_0}}{c_F^3}\zeta_A^{\kappa,\iota}(\phi_3,(1,1),\mathcal T_0). 
\end{align*}
Thus, we see that
\[ a^{\mathcal G}(S,u_{1,0})=\frac{\vol_{\mathcal M_1}}{2 \, c_F}\, \fc_F(S), \quad  a^{\mathcal G}(S,u_{0,1})=\frac{\vol_{\mathcal M_2}}{2 \, c_F} \fc_F(S),\]
and
\[ a^{\mathcal G}(S,u_{\alpha,1})= \frac{\vol_{\mathcal M_0}}{4\, c_F^{2}}  \sum_\chi  \chi_S(\alpha)\,  \fc_F(S,\chi)^2 \]
where $\chi$ runs over all quadratic characters on $\A^1/F^\times$ unramified outside~$S$.
\end{exa}

\section{Shintani zeta function for the space of binary quadratic forms}\label{4s}

\subsection{Zeta integral}\label{4s1}

Let $F$ be an algebraic number field.
We define algebraic groups over $F$ by
\begin{align*}
& H'=\GL(2),  \qquad  G'=\GL(1)\times H' , \\
& H=H'/A_{H'}\,(=\PGL(2)), \qquad G=\GL(1)\times H. 
\end{align*}
We set
\[
V=\{ x\in \Mat(2) \, | \, x=\, ^tx   \}
\]
Note that $A_G=\GL(1)$ and $G=A_G\times H$.
We will consider the following right actions of $G'$ and $G$ on $V$:
\begin{equation}\label{aG'}
x\cdot g'= a\, ^th'x\,h' \, ,\quad x\in V \, , \; g'=(a,h')\in G'  \, , \; a\in \GL(1) \, , \; h'\in H' 
\end{equation}
and
\begin{equation}\label{aG}
x\cdot g= \frac{a}{\det(h)}\, ^thx\,h \, ,\quad x\in V \, , \; g=(a,h)\in G  \, , \; a\in \GL(1) \, , \; h\in H. 
\end{equation}
A homomorphism $G'\to G$ is defined by
\[
\varrho(a,h')=(a \, \det(h'),h'A_{H'}).
\]
Then, $\varrho$ is compatible with the actions \eqref{aG'} and \eqref{aG}, i.e.,
\[
x\cdot g'=x\cdot \varrho(g')\quad (x\in V\, , \;g'\in G').
\]
The pairs $(G',V)$ and $(G,V)$ are prehomogeneous vector spaces.
Their prehomogeneous zeta integrals are essentially the same since we can relate them by the homomorphism $\varrho$.
In Section \ref{5s1}, we introduce a parabolic subgroup $P_1$ of $\GSp(2)$ and its Levi subgroup $M_1$.
The group $M_1$ is identified with $G'$ by the isomorphism
\begin{equation}\label{M1}
M_1 \ni \begin{pmatrix}^th^{\prime -1}&O_2 \\ O_2& ah' \end{pmatrix}\mapsto (a,h')\in G'.
\end{equation}
The action \eqref{aG'} is also identified with the adjoint action of $M_1$ on $N_{P_1}$.
Hence, a zeta integral defined by \eqref{aG'} naturally appears in the subregular unipotent contribution of $\GSp(2)$.
On the other hand, the action \eqref{aG} is faithful and stabilizers of generic points satisfy certain conditions (cf. Lemmas \ref{4l2} and \ref{4l4}) which are required in Saito's formulation \cite{Saito2}.
This will enable us to apply the results of \cite{Saito1,Saito2,Saito3} directly to the zeta integral corresponding to the action~\eqref{aG}.
After that, we will relate it to the zeta integrals associated to the action~\eqref{aG'} which arise naturally in the subregular contributions of $\GSp(2)$ and~$\Sp(2)$.

We introduce a zeta integral defined by the action~\eqref{aG}.
Let $\d a$ denote a Haar measure on $A_G(\A)$ and let $\d h$ denote a Haar measure on $H(\A)$.
Then, $\d g=\d a \, \d h$ is a Haar measure on $G(\A)$.
We can identify $\a_G$ with $\R$ such that $H_G((a,h))=\log|a|$ for $(a,h)\in G(\A)$.
We choose the Haar measure on $\a_G$ induced from the Lebesgue measure on $\R$ by the identification.
Furthermore, a Haar measure $\d^1 a$ (resp. $\d^1 g$) on $A_G(\A)^1$ (resp. $G(\A)^1$) is induced from $\d a$ (resp. $\d g$) and the Lebesgue measure.
We have $\d^1 g=\d^1 a \, \d h$.
We set
\[ 
V^\ss=\{  x\in V \, | \, \det(x)\neq 0  \}
\]
(ss stands for ``semi-stable") and
\[
V^\st(F)=\{  x\in V^\ss(F) \, | \, -\det(x)\not\in(F^\times)^2  \}
\]
(st stands for ``strictly semi-stable").
Let $\chi$ be a quadratic character of $\A^1/F^\times$.
We set $\chi(g)=\chi(a \, \det(h^{-1}))$ and $\omega(g)=a^2$ for $g=(a,h)\in G(\A)$.
We define a zeta integral by
\[
Z(\Phi,s,\chi)=\int_{G(F)\bsl G(\A)} |\omega(g)|^s \, \chi(g)  \sum_{x\in V^\st(F)}\Phi(x\cdot g)\, \d g
\]
for $s\in\C$ and $\Phi\in \cS(V(\A))$.
\begin{lem}\label{4l1}
The prehomogeneous zeta integral $Z(\Phi,s,\chi)$ converges absolutely for $\Re(s)>3/2$ and defines a holomorphic function in this range.
\end{lem}
\begin{proof}
It follows from \cite[Theorem 1.1]{Saito3} and Lemma \ref{4l4} that $Z(\Phi,s,\chi)$ is absolutely convergent if $\Re(s)$ is sufficiently large.
Thus, we can easily prove this lemma using the proof of Theorem~\ref{6t2} below.
\end{proof}

\subsection{Orbits and Stabilizers}\label{4s2}

We need some lemmas in order to apply Saito's formula \cite{Saito2,Saito3} to $Z(\Phi,s,\chi)$.
Lemma~\ref{4l6} appears in \cite[p.597]{Saito2}.
Note that the Hasse principle, which is required for Saito's formula, holds for $G$ (cf. \cite{PR}).

Let $\ds$ denote the element of $F^\times/(F^\times)^2$ represented by $d\in F^\times$.
For $\alpha\in \mathbb{G}_m$, we set
\begin{equation}\label{xd}
\xd_\alpha=\begin{pmatrix}1&0 \\ 0&-\alpha \end{pmatrix}\in V^\ss \, .
\end{equation}
For each $x\in V(F)$, we put $G_{x,+}=\{g\in G \, | \, x\cdot g= x\}$.
Let $G_x$ denote the connected component of $1$ in $G_{x,+}$.
Let $H^1(F,\mathcal G)$ denote the first Galois cohomology set for an algebraic group $\mathcal G$ over $F$ (see, e.g., \cite[Section 1.3]{PR}).
We set $\Pi_d=G_{\xd_d}\bsl G_{\xd_d,+}$ for any $d\in F^\times$.
The following two lemmas can be proved by direct calculation.
\begin{lem}\label{4l2}
Let $d$ be an element of $F^\times$.
Then, we have
\[
G_{\xd_d,+}=\left\{ ((-1)^k \, , \, \begin{pmatrix}(-1)^k&0 \\ 0&1 \end{pmatrix} \begin{pmatrix}\alpha& d \beta \\ \beta& \alpha \end{pmatrix} ) \in  G \, \Bigl| \,\begin{array}{c} k=0\text{ or }1  , \\  \alpha^2-d\beta^2\neq 0 \end{array} \right\} . 
\]
If $d\in (F^\times)^2$, then $G_{\xd_d} \cong  \mathbb{G}_m$ over $F$. 
Assume $d \not\in (F^\times)^2$ and let $L=F(\sqrt{d})$.
Then $G_{\xd_d} \cong \mathbb{G}_m \bsl R_{L/F}(\mathbb{G}_m)$ over $F$.
Note that $G_{\xd_d} \cong R_{L/F}^{(1)}(\mathbb{G}_m)=\{ y\in R_{L/F}(\mathbb{G}_m) \, | \, N_{L/F}(y)=1 \}$ over $F$ (cf. \cite[p.76]{PR}).
For any field $k\supset F$ and any $d\in F^\times$, we have $|\Pi_d(k)|=2$.
\end{lem}
\begin{lem}\label{4l3}
There is a one-to-one correspondence between $F^\times/(F^\times)^2$ and the set of $G(F)$-orbits in $V^\ss(F)$, which is given by
\[
\ds\in F^\times/(F^\times)^2 \longrightarrow \xd_d \cdot G(F)\in  V^\ss(F)/G(F) \, .
\]
For $d$, $d_1\in F^\times$ and $\sigma\in \Gal(\bar F/F)$, let
\[
z_{d,d_1}(\sigma)=\frac{\bigl(\sqrt{d/d_1}\,\bigr)^\sigma}{\sqrt{d/d_1}}.
\]
Here, $\sqrt{d/d_1}$ is identified with $g=(\sqrt{d/d_1},\diag(1,\sqrt{d/d_1}))\in G$.
Then the $1$-cocycle $z_{d,d_1}$ is trivial in $H^1(F,\Pi_d)$ if and only if  $\check d_1=\check d$.
Hence, we have a bijection $F^\times/(F^\times)^2  \to H^1(F,\Pi_d)$, $\check d_1 \mapsto[z_{d,d_1}]$.
\end{lem}

For an algebraic group $\mathcal G$, we denote by $X(\mathcal G)$ the group of the characters of $\mathcal G$.
\begin{lem}\label{4l4}
Let $d\in F^\times$ and $L=F(\sqrt{d})$.
The group $X(G_{\xd_d})$ is the free $\Z$-module $\Z \langle \omega \rangle$, where $\omega$ is defined by
\[
\omega \, : \, \mathbb{G}_m \begin{pmatrix}\alpha & d \beta \\ \beta & \alpha \end{pmatrix}  \; \longrightarrow \; \frac{\alpha+\beta \sqrt{d} }{\alpha-\beta\sqrt{d}} \, .
\]
If $d\in(F^\times)^2$, then $\omega$ is defined over $F$ and $X(G_{\xd_d})_F \neq \{ 1 \}$.
If $d\not\in(F^\times)^2$, then $\omega$ is defined over $L$ and $X(G_{\xd_d})_F= \{ 1 \}$.
\end{lem}
\begin{proof}
This follows from Lemma \ref{4l2}.
\end{proof}
We set $X_*(\mathcal G)=\Hom(\mathbb{G}_m,\mathcal G)$ for an algebraic group $\mathcal G$.
\begin{lem}\label{4l5}
Let $d\in F^\times$.
The group $X_*(G_{\xd_d})$ is the free $\Z$-module $\Z \langle \omega^* \rangle$, where $\omega^*$ is defined by
\[
\omega^* \, : \,  \alpha   \; \longrightarrow \; \mathbb{G}_m \begin{pmatrix}(\alpha+1)/2 & d(\alpha-1)/(2\sqrt{d}) \\ (\alpha-1)/(2\sqrt{d}) & (\alpha+1)/2 \end{pmatrix} \, .
\]
\end{lem}
\begin{proof}
We can prove this by direct calculation.
\end{proof}
Let $A(\mathcal G)$ denote the torsion group of $\mathcal G$ as defined in \cite{Borovoi}.
\begin{lem}\label{4l6}
Let $d\in F^\times$.
If $d\in (F^\times)^2$, then $A(G_{\xd_d})=\{1\}$, and if $d\not\in (F^\times)^2$, then $A(G_{\xd_d})=\Z/2\Z$.
The canonical map from $A(G_{\xd_d})$ to $A(G)=\Z/2\Z$ is injective.
\end{lem}
\begin{proof}
This follows from Lemma \ref{4l5}.
\end{proof}

\subsection{Quadratic characters}\label{4s3}

We denote by $(\; , \;)_v$ the Hilbert symbol on $F_v^\times \times F_v^\times$ for each $v\in\Sigma$.
For $d$, $x\in F_v^\times$ we set $\chi_{d,v}(x)=(d,x)_v$, that is,
\[
\chi_{d,v}(x)=\begin{cases} 1 & \text{if $d\alpha^2+x\beta^2=\gamma^2$ has a non-trivial solution in $F_v$}, \\ -1 & \text{otherwise}. \end{cases}  
\]

Let $d\in F^\times- (F^\times)^2$ and $L=F(\sqrt{d})$.
Let $\chi_d$ denote the quadratic character of $\A^1/F^\times$ corresponding to $L$ via class field theory.
We know that $\chi_d$ is an isomorphism $\A_F^\times/F^\times N_{L/F}(\A^\times_L) \to \{\pm 1\}$ and satisfies $\chi_d=\prod_{v\in\Sigma} \chi_{d,v}$.
For any non-trivial quadratic character $\chi$ on $\A^1/F^\times$, there is a unique nontrivial class $\check d\in F^\times/(F^\times)^2$ such that $\chi=\chi_d$.

\subsection{Local zeta functions}\label{4s4}

Fix a place $v\in \Sigma$.
For $d_v\in F_v^\times$ we set $\ds_v=d(F_v^\times)^2\in F_v^\times/(F_v^\times)^2$.
The following lemma is trivial.
\begin{lem}\label{4l7}
There is a one-to-one correspondence between $F_v^\times/(F_v^\times)^2$ and the set of $G(F_v)$-orbits in $V^\ss(F_v)$, which is given by
\[
\ds_v \in F_v^\times/(F_v^\times)^2 \longrightarrow \xd_{d_v}\cdot G(F_v) \in  V^\ss(F_v)/G(F_v) \, .
\]
In particular, $|F_v^\times/(F_v^\times)^2|$ is finite.
\end{lem}

Recall that $H'=\GL(2)$ acts on $V$ as \eqref{aG'}.
The Hasse invariant of $x_v \in  V^\ss(F_v)$ is defined by
\[
\varepsilon_v(x_v)=(\alpha,\alpha)_v(\alpha,\beta)_v(\beta,\beta)_v 
\]
if $^th' x_v h'=\diag (\alpha,\beta)$ for some $h'\in H'(2,F_v)$.
We set
\begin{multline*}
\mathcal E(F_v)= \Big\{  (\ds_v,\epsilon)\in F_v^\times/(F_v^\times)^2 \times \{\pm 1\} \\
\big| \, \begin{array}{c}  d_v\in F_v^\times ,\, \; \text{$\epsilon=1$ or $-1$} ,  \\  \text{$\exists x_v\in V^\ss(F_v)$ s.t. $-\det(x_v)=d_v$ and $\epsilon=\varepsilon_v(x_v)$} \end{array} \, \Big\}.
\end{multline*}
The following lemma is well-known.
\begin{lem}\label{4l8}
There is a one-to-one correspondence between the set of $H'(F_v)$-orbits in $V^\ss(F_v)$ and $\mathcal E(F_v)$, which is given by
\[
x_v\cdot H'(F_v) \in  V^\ss(F_v)/H'(F_v) \longrightarrow  (\ds_v,\varepsilon_v(x_v)) \in \mathcal E(F_v)
\]
where $d_v=-\det(x_v)$.
\end{lem}

Let $\chi_v$ be a quadratic character of $F_v^\times$.
The trivial representation of $F_v^\times$ is denoted by $\trep_{F_v}$.
Fix an element $d_v\in F^\times_v$.
Assume that $\chi_v$ is $\trep_{F_v}$ or $\chi_{d_v,v}$.
A function $\chi_v$ on $V(F_v)$ is defined by
\[
\chi_v(x_v)=\begin{cases} 1 & \text{if $\chi_v=\trep_{F_v}$}, \\ \varepsilon_v(x_v) \, \varepsilon_v(\xd_{d_v}) & \text{if $\chi_v=\chi_{d_v,v}$}. \end{cases}
\]
If $\chi_v=\chi_{d_v,v}$, then we have
\[
\chi_v( \xd_{d_v}\cdot (a_v,h_v))=\chi_v(a_v \, \det(h_v^{-1}))  \qquad ((a_v,h_v)\in G(F_v)).
\]
Note that $\varepsilon_v(\xd_{d_v})=1$ if $v\not\in\Sigma_\inf\cup\Sigma_2$ and $\chi_{d_v,v}$ is unramified.

We define the measure $\d x_v$ on $V(F_v)$ by
\[
\d x_v=\d x_{1,v} \, \d x_{12,v} \, \d x_{2,v} \, \text{ for }\,  x_v=\begin{pmatrix}x_{1,v}&x_{12,v}\\x_{12,v}&x_{2,v}\end{pmatrix}\in V(F_v)
\]
where $\d x_{*,v}$ is the measure on $F_v$ defined in Section \ref{2s1}.
Set
\[
V^\ss(F_v,d_v)=\xd_{d_v}\cdot G(F_v).
\]
We define a local zeta function by
\[
Z_v(\Phi_v,s,\chi_v;d_v)=\int_{V^\ss(F_v,d_v)}|\det(x_v)|_v^{s-\frac{3}{2}} \, \chi_v(x_v) \, \Phi_v(x_v) \, \d x_v, 
\]
where $s\in\C$ and $\Phi_v\in\cS(V(F_v))$.

From now on, $\varepsilon_v$ is $1$ or $-1$.
Let $V^\ss(F_v,d_v,\varepsilon_v)$ denote the preimage of $(\ds_v,\varepsilon_v)\in\mathcal E(F_v)$ via the map of Lemma \ref{4l8}.
A partial local zeta function is defined by
\[ \tilde Z_v(\Phi_v,s;d_v,\varepsilon_v)=\int_{V^\ss(F_v,d_v,\varepsilon_v)}|\det(x_v)|_v^{s-\frac{3}{2}}  \, \Phi_v(x_v) \, \d x_v,    \]
where $s\in\C$ and $\Phi_v\in\cS(V(F_v))$.
We set $\mathcal E(F_v,t_v)=\{(\check t_v,\varepsilon_v)\in\mathcal E(F_v)  \}$ for each $t_v\in F_v^\times$.
Clearly, $|\mathcal E(F_v,t_v)|=1$ or $2$.
Since $\chi_v$ is a constant on each $H'(F_v)$-orbit, we have
\begin{equation}\label{4e1}
Z_v(\Phi_v,s,\chi_v;d_v)=\sum_{(\ds_v,\varepsilon_v) \in \mathcal E(F_v,d_v)}\frac{\varepsilon_v}{\varepsilon_v(\xd_{d_v})} \, \tilde Z_v(\Phi_v,s;d_v,\varepsilon_v) 
\end{equation}
if $\chi_v=\chi_{d_v,v}$.
\begin{lem}\label{4l9}
If $v\in\Sigma_\fin$, then
$Z_v(\Phi_v,s,\chi_v;d_v)$ and $\tilde Z_v(\Phi_v,s;d_v,\varepsilon_v)$ are absolutely convergent for $\Re(s)>3/2$ and can be meromorphically continued to the whole complex $s$-plane.
Their poles are contained in $\{ \, 0 \, , \, 1/2\, \}$.
\end{lem}
\begin{proof}
This statement follows from \cite[Theorems 1 and 2]{Igusa1}, Lemmas \ref{4l7} and \ref{4l8}, and (\ref{4e1}).
\end{proof}
\begin{lem}\label{4l10}
If $v\in\Sigma_\inf$, then
$Z_v(\Phi_v,s,\chi_v;d_v)$ and $\tilde Z_v(\Phi_v,s,;d_v,\varepsilon_v)$ are absolutely convergent for $\Re(s)>3/2$ and can be meromorphically continued to the whole complex $s$-plane.
Their poles are contained in $\Z_{\leq 0} \cup (\frac{1}{2}+\Z_{\leq 0})$.
\end{lem}
\begin{proof}
See \cite{SS} or \cite{Igusa2}.
\end{proof}

Let $S$ be a finite set of places,
\[
d_S=(d_v)_{v\in S} \in F_S^\times, \quad \varepsilon_S=(\varepsilon_v)_{v\in S}\in \{\pm 1\}^{|S|}, \quad \Phi_S\in\cS(V(F_S)).
\]
Assume that $\chi_v=\trep_{F_v}$ or $\chi_{d_v,v}$.
We set
\[
\chi_S=\prod_{v\in S}\chi_v, \quad \d x_S=\prod_{v\in S}\d x_v, \quad V^\ss(F_S,d_S)=\prod_{v\in S}V^\ss(F_v,d_v),
\]
\[
\text{and}\quad  V^\ss(F_S,d_S,\varepsilon_S)=\prod_{v\in S}V^\ss(F_v,d_v,\varepsilon_v).
\]
We also set
\[  Z_S(\Phi_S,s,\chi_S;d_S)=\int_{V^\ss(F_S,d_S)}|\det (x_S)|_S^{s-\frac{3}{2}} \, \chi_S(x_S) \,  \Phi_S(x_S) \, \d x_S  \]
and
\[  \tilde Z_S(\Phi_S,s;d_S,\varepsilon_S)=\int_{ V^\ss(F_S,d_S,\varepsilon_S) }|\det (x_S)|_S^{s-\frac{3}{2}} \,  \Phi_S(x_S) \, \d x_S.  \]
We set $\mathcal E(F_S,d_S)= \prod_{v\in S}  \mathcal E(F_v,d_v)  $.
By \eqref{4e1} we have
\begin{equation}\label{4e2}
Z_S(\Phi_v,s,\chi_S;d_S)=\sum_{(d_S,\varepsilon_S) \in \mathcal E(F_S,d_S)}\Big(\prod_{v\in S} \mathfrak e_v  \Big) \tilde Z_S(\Phi_S,s;d_S,\varepsilon_S) .
\end{equation}
where $\mathfrak e_v=1$ if $\chi_v=\trep_{F_v}$ and $\mathfrak e_v=\varepsilon_v \times \varepsilon_v(\xd_{d_v})$ if $\chi_v=\chi_{d_v,v}$.
\begin{lem}\label{4l11}
The integrals $Z_S(\Phi_S,s,\chi_S;d_S)$ and $\tilde Z_S(\Phi_S,s;d_S,\varepsilon_S)$ are absolutely convergent for $\Re(s)>3/2$ and can be meromorphically continued to the whole complex $s$-plane. Moreover, they are holomorphic for $\Re(s)>1/2$.
\end{lem}
\begin{proof}
We may identify the pair $(G(F_\inf),V(F_\inf))$ with a prehomogeneous vector space over $\R$.
Hence, this lemma follows from \cite{SS}, \cite{Igusa2}, and Lemma \ref{4l9}.
\end{proof}

\subsection{Explicit formulas}\label{4s5}

Let $\Phi\in\cS(V(\A))$ and $\chi=\prod_{v\in\Sigma}\chi_v$ be a quadratic character.
We set
\begin{equation}\label{para}
\para(F)=\{ \ds \in F^\times/(F^\times)^2 \, | \, d\in F^\times -(F^\times)^2 \} 
\end{equation}
and
\[ Z(\Phi,s,\chi;d)=\int_{G(F)\bsl G(\A)} |\omega(g)|^s \, \chi(g) \sum_{x\in \xd_d\cdot G(F)}\Phi(x\cdot g)\, \d g \, ,\]
where $d\in F^\times-(F^\times )^2$.
By Lemma \ref{4l3} we have
\[  Z(\Phi,s,\chi)=\sum_{\ds\in \para(F)} Z(\Phi,s,\chi;d).  \]
Let $\Phi_{0,v}$ denote the characteristic function of $V(\fO_v)$.
We will assume the following condition for $S$.
\begin{cond}\label{cond2}
The finite set $S$ satisfies $S\supset\Sigma_\inf\cup\Sigma_2$ and there exists a function $\Phi_S\in\cS(V(F_S))$ such that $\Phi=\Phi_S \times \prod_{v\not\in S}\Phi_{0,v}$.
\end{cond}
For $d\in F^\times-(F^\times)^2$, we set
\[ N(\mathfrak{f}^S_d) = \prod_{v\not\in S \, , \, \text{$\chi_{d,v}$ is ramified} } q_v  \]
for $S\supset\Sigma_\inf$.
If $S$ contains $\Sigma_\inf\cup\Sigma_2$, then $\mathfrak{f}^S_d$ means the conductor of $\prod_{v\not\in S}\chi_{d,v}$ since $\chi_d$ is quadratic.
Let $\trep_S$ denote the trivial character of $F_S^\times$.
We set $\chi_{d,S}=\prod_{v\in S}\chi_{d,v}$.
\begin{thm}\label{4t13}
Fix $d\in F^\times-(F^\times)^2$, a quadratic character $\chi$ on $\A^1/F^\times$, and a test function $\Phi\in\cS(V(\A))$.
Assume that $\Re(s)>3/2$.
We have
\[ Z(\Phi,s,\chi;d)=0 \]
unless $\chi= \trep_F$ or $\chi_d$.
Assume that $S$ satisfies Condition \ref{cond2}.
If $\chi=\trep_F$, then we have
\begin{multline*}
Z(\Phi,s,\trep_F;d) \\
= \frac{\vol_G \, 2^{|S|}\,  L^S(1,\chi_d)}{2 \, c_F^{S}\, \zeta_F^S(2)} \times \frac{ \zeta^S_F(2s-1)  \zeta^S_F(2s) }{L^S(2s,\chi_d) \, N(\mathfrak{f}^S_d)^{s-\frac{1}{2}}} \times Z_S(\Phi_S,s,\trep_S;d) .
\end{multline*}
If $\chi=\chi_d$ and $\chi_{d,v}$ is unramified for each $v\not\in S$, then we have
\[ Z(\Phi,s,\chi_d;d)=\frac{ \vol_G\, 2^{|S|}\, L^S(1,\chi_d)}{2 \, c^S_F\, \zeta^S_F(2)} \times \zeta_F^S(2s-1) \times  Z_S(\Phi_S,s,\chi_{d,S};d). \]
If $\chi=\chi_d$ and there exists a place $v\not\in S$ such that $\chi_{d,v}$ is ramified, then we have
\[
Z(\Phi,s,\chi;d)=0.
\]
\end{thm}
\begin{proof}
We easily see that
\begin{align*}
Z(\Phi,s,\chi;d)=& \frac{1}{2}\int_{G_{\xd_d}(\A)\bsl G(\A)}\chi\left(\frac{a}{\det(h_1)}\right)\, |a|^{2s} \, \Phi(\xd_d\cdot (a,h_1)) \\
& \times \int_{G_{\xd_d}(F)\bsl G_{\xd_d}(\A)} \chi(\det(h_2))\,\d h_2 \, \d h_1 \, \d^\times a 
\end{align*}
where $h_1\in G_{\xd_d}(\A)\bsl H(\A)$, $h_2\in G_{\xd_d}(F)\bsl G_{\xd_d}(\A)$, and $\d h=\d h_1 \, \d h_2$.
By Lemma \ref{4l2} and the orthogonality of characters, $Z(\Phi,s,\chi;d)=0$ unless $\chi=\trep_F$ or $\chi_d$.
From now on we assume that $\chi=\trep_F$ or $\chi_d$.
By Lemmas \ref{4l2}, \ref{4l3}, \ref{4l4}, \ref{4l6}, and \cite[Proof of Theorem 2.1]{Saito2}, we obtain
\begin{align*}
Z(\Phi,s,\chi;d)=& \frac{\vol_G}{2} \times \frac{L(1,\chi_d)}{c_F}\times Z_S(\Phi_S,s,\chi_S;d) \\
& \times \prod_{v\in S}( 2 \, c_v) \times\prod_{v\in S\cap\Sigma_\fin} L_v(1,\chi_{d,v})^{-1} \\
& \times \prod_{v\not\in S}  \{ 2 \, c_v \, L_v(1,\chi_{d,v})^{-1}\, Z_v(\Phi_{0,v},s,\chi_v;d)  \} \, ,
\end{align*}
where $\chi_S=\trep_S$ (resp. $\chi_{d,S}$) and $\chi_v=\trep_{F_v}$ (resp. $\chi_{d,v}$) if $\chi=\trep_F$ (resp. $\chi_d$).
Note that Tamagawa measures are required for his formulation (cf. \cite[p.599]{Saito2}) and the Tamagawa number of $\PGL(2)$ is $2$ (see e.g., \cite[Section 5.3]{PR}).
The factor $\vol_G/2$ came from them.

Assume $v\not\in S$.
From \cite[Theorem 2.2]{Saito1} we see that
\begin{align*}
& 2 \, c_v \, L_v(1,\chi_{d,v})^{-1}\, Z_v(\Phi_{0,v},s,\trep_{F_v};d) \\
& =\frac{(1-q_v^{-2s+1})^{-1}(1-q_v^{-2s})^{-1}}{(1-q_v^{-2})^{-1} \, L_v(2s,\chi_{d,v})} \times \begin{cases} 1 & \text{if $\chi_{d,v}$ is unramified}, \\ q_v^{-s+\frac{1}{2}} & \text{if $\chi_{d,v}$ is ramified}. \end{cases}  
\end{align*}
If $\chi_{d,v}$ is unramified, we also find
\[ 2 \, c_v \, L_v(1,\chi_{d,v})^{-1}\, Z_v(\Phi_{0,v},s,\chi_{d,v};d)  =\frac{(1-q_v^{-2s+1})^{-1}}{(1-q_v^{-2})^{-1} }  \]
from \cite[Theorem 2.2]{Saito1}.
If $\chi_{d,v}$ is ramified, it is trivial that
\[
Z_v(\Phi_{0,v},s,\chi_{d,v};d)=0.
\]
Thus, we have proved the theorem.
\end{proof}

The assumption $S\supset\Sigma_2$ in Theorem \ref{4t13} is made for the sake of simplicity. It could be removed if we used the explicit form of $Z_v(\Phi_{0,v},s,\chi_v,d)$ computed in \cite[Theorem 2.2]{Saito1} for any $v\in\Sigma_2$.

\subsection{Zeta function}\label{4s6}

Let $d_S\in F_S^\times$ and $S$ be a finite set of $\Sigma$ satisfying $S\supset\Sigma_\inf\cup\Sigma_2$.
We set
\[ \para(F,S,d_S)=\{ \ds\in\para(F)\, | \, d\in d_S(F_S^\times)^2 \}  \]
We define the zeta function $\xi^S(s;d_S)$ by
\[\xi^S(s;d_S)=\frac{\zeta^S_F(2s-1) \, \zeta^S_F(2s)}{\zeta^S_F(2)} \sum_{\ds\in \para(F,S,d_S)} \frac{ L^S(1,\chi_d) }{L^S(2s,\chi_d) \, N(\mathfrak{f}^S_d)^{s-\frac{1}{2}}} . \]
\begin{lem}\label{4l14}
The zeta function $\xi^S(s;d_S)$ is absolutely convergent and holomorphic for $\Re(s)>3/2$.
\end{lem}
\begin{proof}
For $v\not\in S$, we set $\Phi_v=\Phi_{0,v}$.
For $v\in S\cap \Sigma_\fin$, we assume that $\Phi_v$ is the characteristic function of $\xd_{d_v}+\pi_v^{n_v}V(\fO_v)$ for an integer $n_v$ large enough.
The support of $\Phi_v$ is contained in $\xd_{d_v}\cdot G(F_v)$ because $-\det(\xd_{d_v}+\pi_v^{n_v}V(\fO_v))\subset d_v(F_v^\times)^2$.
For $v\in\Sigma_\inf$, it is trivial that there exists a test function $\Phi_v\in\cS(V(F_v))$ such that the support of $\Phi_v$ is contained in $\xd_{d_v}\cdot G(F_v)$.
If we set $\Phi_S=\prod_{v\in S}\Phi_v$ for such $\Phi_v$, then we get
\begin{multline*}
Z(\Phi,s,\trep_F) \\
= \frac{\vol_G \, 2^{|S|}}{2\, c_F^{S}\; \zeta_F^S(2)} \times \xi^S(s;d_S) \times \prod_{v\in S}\int_{V^\ss(F_v,d_v)}|\det(x_v)|_v^{s-\frac{3}{2}} \, \Phi_v(x_v) \, \d x_v 
\end{multline*}
by Theorem \ref{4t13}.
Hence, we deduce this lemma from Lemmas \ref{4l1}, \ref{4l9}, and \ref{4l10}.
\end{proof}

We obtain the following from Theorem \ref{4t13} and Lemma \ref{4l14}.
\begin{thm}\label{4t15}
Fix a quadratic character $\chi$ on $\A^1/F^\times$ and a test function $\Phi\in\cS(V(\A))$.
Assume that $\Re(s)>3/2$ and $S$ satisfies Condition \ref{cond2}.
If $\chi=\trep_F$, we have
\[ Z(\Phi,s,\trep_F)= \frac{\vol_G \, 2^{|S|}}{2\, c_F^S} \sum_{d_S\in F_S^\times/(F_S^\times)^2} Z_S(\Phi_S,s,\trep_S;d_S) \, \xi^S(s;d_S)  . \]
Let $d\in F^\times-(F^\times)^2$.
If $\chi=\chi_d$ and $\chi_{d,v}$ is unramified for each $v\not\in S$, then we have
\[ Z(\Phi,s,\chi_d)= \frac{\vol_G \, 2^{|S|}}{2\, c_F^S} \times Z_S(\Phi_S,s,\chi_{d,S};d) \, L^S(1,\chi_d) \, \frac{\zeta_F^S(2s-1)}{\zeta_F^S(2)} .  \]
If $\chi=\chi_d$ and there exists a place $v\not\in S$ such that $\chi_{d,v}$ is ramified, then we have
\[
Z(\Phi,s,\chi_d)=0.
\]
\end{thm}
The formula for $\chi=\chi_d$ is a generalization of \cite[Theorem 1]{IS}.

\subsection{Meromorphic continuation}\label{4s7}

\begin{prop}\label{4p16}
Let $\chi$ be a non-trivial quadratic character of $\A^1/F^\times$.
The zeta integral $Z(\Phi,s,\chi)$ is holomorphic for $\Re(s)>1/2$ and can be meromorphically continued to the whole complex $s$-plane.
\end{prop}
\begin{proof}
This follows from Lemma \ref{4l11} and Theorem \ref{4t15}.
\end{proof}

We set $\d x=\prod_{v\in\Sigma}\d x_v$, where $\d x_v$ was defined in Section \ref{4s4}.
For $\Phi\in\cS(X(\A))$, the Fourier transform $\hat\Phi$ of $\Phi$ is defined by
\[
\hat\Phi(y)=\int_{V(\A)}  \Phi(x) \, \psi_F(\Tr(xy))  \d x  .
\]

\begin{prop}\label{4p17}
The zeta integral $Z(\Phi,s,\trep_F)$ is holomorphic for $\Re(s)>3/2$ and can be meromorphically continued to the whole complex $s$-plane.
If $\hat\Phi(0)\neq 0$, it has a simple pole at $3/2$ and with residue $\hat\Phi(0)\, \vol_G/2$, otherwise it is holomorphic at $3/2$.
\end{prop}
\begin{proof}
This follows from \cite[Theorem 4.2]{Yukie}.
\end{proof}
\begin{prop}\label{4p18}
Assume that $S\supset\Sigma_\inf\cup\Sigma_2$.
The zeta function $\xi^S(s;d_S)$ is holomorphic for $\Re(s)>3/2$ and can be meromorphically continued to the whole complex $s$-plane.
It has a simple pole at $3/2$ with residue $2^{-|S|}c_F^S$ (which does not depend on $d_S\in F_S^\times$).
\end{prop}
\begin{proof}
This is deduced from Lemma \ref{4l11}, Theorem \ref{4t15}, Proposition \ref{4p17}, and the proof of Lemma \ref{4l14}.
\end{proof}

\subsection{Modified zeta integral}\label{4s8}

Let $T_1\in\R$.
The function $\tau_{T_1}(t)$ on $\R$ is defined by $\tau_{T_1}(t)=1$ if $t<-{T_1}$, and $\tau_{T_1}(t)=0$ if $t\geq -{T_1}$.
We define a modified zeta integral $Z(\Phi,s,\trep_F,T_1)$ by
\begin{multline*}
Z(\Phi,s,\trep_F,T_1)=  \\
\int_{G(F)\bsl G(\A)} |\omega(g)|^s   \Bigl\{ \sum_{x\in V^\st(F)} \Phi(x\cdot g)  - |\omega(g)|^{-3/2} \hat\Phi(0)\, \tau_{T_1}(H_G(g))  \Bigr\} \d g .
\end{multline*}
\begin{lem}\label{4l19}
The zeta integral $Z(\Phi,s,\trep_F,T_1)$ is absolutely convergent for $\Re(s)>5/4$.
Furthermore, $Z(\Phi,s,\trep_F,T_1)$ is holomorphic for $\Re(s)>1$.
\end{lem}
\begin{proof}
The first assertion follows from \eqref{4e3} and the proof of Theorem \ref{6t2}.
The second assertion follows from the argument of \cite[p.369--p.373]{Yukie}, because it is enough to replace $\hat\Phi(0)$ by $\hat\Phi(0)\{1-\tau_{T_1}(H_G(g)) \}$.
If we replace the condition $\Re(s)>1$ by $\Re(s)>5/4$, it also follows from \eqref{4e3} and the proof of Theorem \ref{6t2}.
The condition $\Re(s)>5/4$ is enough for our purpose.
\end{proof}
\begin{lem}\label{4l20}
If $T_1>0$, then we have
\[  Z(\Phi,3/2,\trep_F,T_1)=\lim_{s\to 3/2}\frac{\d}{\d s}\Bigl(s-\frac{3}{2}\Bigr) Z(\Phi,s,\trep_F)+\vol_G\, \hat\Phi(0) \, T_1 \, . \]
\end{lem}
\begin{proof}
This lemma follows from Lemma \ref{4l19}, \cite[Theorem 4.2]{Yukie}, and \cite[p.373]{Yukie}.
We can also deduce this lemma from  Lemma \ref{4l19}, \eqref{4e3}, and the proof of Theorem \ref{6t2}, because our argument for Theorem \ref{6t2} can be applied to $Z(\Phi,s,\trep_F)$ and $Z(\Phi,s,\trep_F,T_1)$.
\end{proof}
For a finite subset $S\supset \Sigma_\inf$ and an element $d_S\in F_S^\times/(F_S^\times)^2$, we define the constant $\fC_F(S,d_S)$ by
\[ \fC_F(S,d_S)=\lim_{s\to 3/2}\frac{\d}{\d s}\Bigl(s-\frac{3}{2}\Bigr)\; \xi^S(s;d_S) .  \]
\begin{lem}\label{4l21}
Fix a test function $\Phi\in C_c^\inf(V(\A))$.
If $S$ satisfies Condition \ref{cond2}, then we have
\begin{align*}
&\lim_{s\to 3/2}\frac{\d}{\d s}\Bigl(s-\frac{3}{2}\Bigr) Z(\Phi,s,\trep_F)\\
&=\frac{\vol_G}{2}\,  \, \int_{V(F_S)}\Phi_S(x_S) \, \log|\det(x_S)|_S\, \d x_S \\
&\quad  +\frac{2^{|S|}\, \vol_G}{2\, c_F^S} \sum_{d_S\in F_S^\times/(F_S^\times)^2}\fC_F(S,d_S) \, Z_S(\Phi_S,3/2,\trep_S;d_S)\, .
\end{align*}
\end{lem}
\begin{proof}
We can easily show this by using Theorem \ref{4t15} and Proposition \ref{4p18}.
Note that the local integrals converge absolutely at $s=3/2$ since $\Phi\in C_c^\inf(V(\A))$.
\end{proof}

In application to the trace formula, $V$ will appear as the unipotent radical of the parabolic subgroups $P_1\subset \GSp(2)$ and $P_1\cap\Sp(2)$ (cf. Section \ref{5s}). We will encounter the modified zeta integrals
\begin{multline*}
Z^{\GSp(2)}(\Phi,s,T_1)= \int_{A_G(F)\bsl A_G(\A)^1}\int_{H'(F)\bsl H'(\A)} |\det(h')^2|^s  \\
\times \Big\{ \sum_{x\in V^\st(F)} \Phi(a\,^th' xh')   - |\det(h')^2|^{-3/2} \hat\Phi(0)\, \tau_{T_1}(H_{H'}(h')) \Big\}   \, \d h' \d^1 a 
\end{multline*}
and
\begin{multline*}
Z^{\Sp(2)}(\Phi,s,T_1)= \int_{H'(F)\bsl H'(\A)} |\det(h')^2|^s \\
 \times\Big\{ \sum_{x\in V^\st(F)} \Phi(\,^th' xh')   - |\det(h')^2|^{-3/2} \hat\Phi(0)\, \tau_{T_1}(H_{H'}(h')) \Big\} \, \d h',
\end{multline*}
where $\Phi\in \cS(V(\A))$.
We identify $\a_{A_{H'}}$ with $\R$ such that 
\[
H_{A_{H'}}(\diag(b,b))=\log|b| \quad (b\in\A^\times),
\]
and we choose the Haar measure $\d^1 z$ on $A_{H'}(\A)^1$ induced from $\d z$ and the Lebesgue measure on $\R$ by the identification.
Then, we have
\begin{equation}\label{4e3}
Z^{\GSp(2)}(\Phi,s,T_1) = \frac{\vol_{A_{H'}}}{2} \, Z(\Phi,s,\trep_F,T_1). 
\end{equation}
By the proof of Theorem \ref{6t2}, $Z^{\Sp(2)}(\Phi,s,T_1)$ is absolutely convergent and holomorphic for $\Re(s)>5/4$.
Applying the method of \cite[(5.11)]{LL} to $Z^{\Sp(2)}(\Phi,s,T_1)$ (cf. the proof of Theorem \ref{3t}), we obtain
\begin{equation}\label{4e4}
Z^{\Sp(2)}(\Phi,s,T_1)=  \frac{\vol_{A_{H'}}}{2\, \vol_{A_G}} \Bigl\{Z(\Phi,s,\trep_F,T_1) + \sum_\chi Z(\Phi,s,\chi) \Bigr\},
\end{equation}
where $\chi$ runs over all non-trivial quadratic characters on $\A^1/F^\times$.
We set
\[ \para^{\mathrm{ur}}(F,S)=\{ \ds \in \para(F) \, | \, \text{$\chi_{d,v}$ is unramified for each $v\not\in S$}   \} . \]
We fix a measure on $\a_{H'}$ by the identification \eqref{M1} and Condition \ref{5c1}.
Note that this measure equals to the product of $1/2$ and the chosen measure on $\a_{A_{H'}}(=\a_{H'})$.
These normalizations were implicitly used in the proof of Theorem \ref{4t13} and are in line with the rule of Tamagawa measures.
\begin{thm}\label{4t22}
Fix a test function $\Phi\in C_c^\inf(V(\A))$.
Assume that $S$ satisfies Condition \ref{cond2}.
If $T_1>0$, then we have
\begin{align*}
&Z^{\GSp(2)}(\Phi,3/2,T_1)\\
&=\frac{ \vol_{G'}}{2}\,  \, \int_{V(F_S)}\Phi_S(x_S) \, \log|\det(x_S)|_S\, \d x_S +   \vol_{G'} \, \hat\Phi(0) \, T_1 \\
& \quad +\frac{2^{|S|}\,  \vol_{G'}}{2\, c_F^S} \sum_{\alpha\in F^\times/(F^\times\cap(F_S^\times)^2)}\fC_F(S,\alpha) \, Z_S(\Phi_S,3/2,\trep_S;\alpha)
\end{align*}
and
\begin{align*}
& Z^{\Sp(2)}(\Phi,3/2,T_1) \\
&=\frac{\vol_{H'}}{2}\,  \, \int_{V(F_S)}\Phi_S(x_S) \, \log|\det(x_S)|_S\, \d x_S +  \vol_{H'} \, \hat\Phi(0) \, T_1 \\
& \quad +\frac{2^{|S|}\, \vol_{H'}}{2\, c_F^S} \sum_{\alpha\in F^\times/(F^\times\cap(F_S^\times)^2)}\fC_F(S,\alpha) \, Z_S(\Phi_S,3/2,\trep_F;\alpha)\\
& \quad +\frac{2^{|S|}\, \vol_{H'}}{2\, c_F^S} \sum_{\ds\in \para^{\mathrm{ur}}(F,S)}L^S(1,\chi_d) \, Z_S(\Phi_S,3/2,\chi_{d,S};d).
\end{align*}
\end{thm}
\begin{proof}
This lemma follows from Theorem \ref{4t15}, Lemmas \ref{4l20} and \ref{4l21}, (\ref{4e3}), and (\ref{4e4}).
\end{proof}
For $d_S\in F_S^\times$, we set
\[
\para^{\mathrm{ur}}(F,S,d_S)=\{ \ds \in \para^{\mathrm{ur}}(F,S) \, | \, d\in d_S(F_S^\times)^2  \} .
\]
We also set
\[
\mathcal E(F,S)=\Big\{ (\alpha,\varepsilon_S) \in \prod_{v\in S}\mathcal E(F_v) \, \Big| \,  \alpha\in F^\times/(F^\times\cap(F_S^\times)^2)  \Big\}.
\]
The following is deduced from \eqref{4e2} and Theorem \ref{4t22}.
Note that $\prod_{v\in S}\varepsilon_v(\xd_d)=1$ for $\ds\in \para^{\mathrm{ur}}(F,S)$ under Condition \ref{cond2}.
\begin{thm}\label{4t23}
Fix a test function $\Phi\in C_c^\inf(V(\A))$.
Assume that $S$ satisfies Condition \ref{cond2}.
If $T_1>0$, then we have
\begin{align*}
& Z^{\Sp(2)}(\Phi,3/2,T_1)\\
& =\frac{\vol_{H'}}{2}\,  \, \int_{V(F_S)}\Phi_S(x_S) \, \log|\det(x_S)|_S\, \d x_S +  \vol_{H'}\, \hat\Phi(0) \, T_1 \\
& \quad +\frac{2^{|S|}\, \vol_{H'}}{2\, c_F^S} \sum_{(\alpha,\varepsilon_S)\in\mathcal E(F,S)}   \tilde Z_S(\Phi_S,3/2;\alpha,\varepsilon_S)     \\
& \qquad \times \Big\{  \fC_F(S,\alpha) + \Big( \prod_{v\in S}\varepsilon_v \Big)  \sum_{\ds\in\para^{\mathrm{ur}}(F,S,\alpha)}  L^S(1,\chi_d)  \Big\}
\end{align*}
where $\varepsilon_S=(\varepsilon_v)_{v\in S}$.
\end{thm}

\subsection{Functional equation}\label{4s9}

In order to state the functional equation, we have to include the contribution from split quadratic forms in the zeta integral.

Let $\bK$ be the maximal compact subgroup of $G'$ which is given by the identification \eqref{M1} and the intersection of $M_1(\A)$ and $\bK$ of Section \ref{5s1}.
For $\Phi\in\cS(V(\A))$, we set
\[
\Phi_\bK(x)=\int_\bK\Phi(x\cdot k)\d k  .
\]
Note that $Z(\Phi,s,\trep_F)=Z(\Phi_\bK,s,\trep_F)$.
We also set
\[
R_0\Phi_\bK(\alpha,\beta)=\Phi_\bK(\begin{pmatrix}\beta&\alpha/2\\ \alpha/2&0\end{pmatrix})\in\cS(\A^2) .
\]
Let $\d u$ (resp. $\d^\times y$) denote the Haar measure on $\A$ (resp. $\A^\times$) defined in Section \ref{2s1}.
Yukie defined the functions
\[
\mathfrak{T}(R_0\Phi_\bK,s,s_1)=\int_\A \int_{\A^\times}   R_0\Phi_\bK (y,yu)\, |y^2|^s \,\|(1,u)\|^{-s_1}\, \d^\times y \, \d u 
\]
and
\[
\mathfrak{T}(R_0\Phi_\bK,s)=-\frac{\d}{\d s_1}\mathfrak{T}(R_0\Phi_\bK,s,s_1)\Bigl|_{s_1=0}.
\]
By \cite[Proposition 2.12]{Yukie} and its proof, the function $\mathfrak{T}(R_0\Phi_\bK,s)$ is absolutely convergent and holomorphic for $\Re(s)>1$ and satisfies
\[
\mathfrak{T}(R_0\Phi_\bK,s)= \int_{\A^\times}\int_\A R_0\Phi_\bK(y,yu)\, |y^2|^s\, \log\|(1,u)\| \, \d u \, \d^\times y .
\]
Let $M$ denote the minimal Levi subgroup of $H'$ which consists of diagonal matrices.
Then, $M$ is the preimage of $M_0$ of $\Sp(2)$ by \eqref{M1}.
A Haar measure on $\a_M$ is fixed by Condition \ref{5c1} and the identification.
An adjusted zeta integral $Z_\mathrm{ad}(\Phi,s)$ is defined by
\[
Z_\mathrm{ad}(\Phi,s)= Z(\Phi,s,\trep_F)+\frac{\vol_M \, \vol_{A_G}}{c_F \, \vol_{A_{H'}}}\, \mathfrak{T}(R_0\Phi_\bK,s) .
\]
\begin{thm}\label{4t25}
{\rm \cite[Theorem 4.2]{Yukie}.}
The adjusted zeta integral $Z_\mathrm{ad}(\Phi,s)$ is holomorphic for $\Re(s)>3/2$ and can be meromorphically continued to the whole complex $s$-plane. It satisfies the functional equation
\[ Z_\mathrm{ad}(\Phi,s)=Z_\mathrm{ad}\Big(\hat\Phi,\frac{3}{2}-s\Big). \]
If $\hat\Phi(0)\neq 0$ (resp. $\Phi(0)\neq 0$), then $Z_\mathrm{ad}(\Phi,s)$ has a simple pole at $s=3/2$ (resp. $s=0$) and the residue is $\hat\Phi(0) \vol_G /2$ (resp. $-\Phi(0) \vol_G /2$).
If $\hat\Phi(0)= 0$ (resp. $\Phi(0)=0$), then $Z_\mathrm{ad}(\Phi,s)$ has no pole at $s=3/2$ (resp. $s=0$).
\end{thm}
We want to restate this result, for application to the trace formula,
in terms of the adjusted zeta integrals
\begin{align*}
&Z^{\GSp(2)}_\mathrm{ad}(\Phi,s) \\
& =  \int_{A_G(F)\bsl A_G(\A)^1}\int_{H'(F)\bsl H'(\A)} |\det(h')^2|^s   \sum_{x\in V^\st(F)} \Phi(a\,^th' xh')   \, \d h' \d^1 a \\
& \quad + \frac{\vol_M \, \vol_{A_G}}{2\, c_F}  \mathfrak{T}(R_0\Phi_\bK,s) 
\end{align*}
and
\begin{align*}
Z^{\Sp(2)}_\mathrm{ad}(\Phi,s)=& \int_{H'(F)\bsl H'(\A)} |\det(h')^2|^s  \sum_{x\in V^\st(F)} \Phi(\,^th' xh')   \d h' \\
&+ \frac{\vol_M}{2\, c_F}  \mathfrak{T}(R_0\Phi_\bK,s).
\end{align*}
In order to allow for truncation, we need the functions
\[
\mathfrak{T}_0(R_0\Phi_\bK,s)= \int_{\A^\times}\int_\A R_0\Phi_\bK(a,au)\, |a|^{2s}\, \log|a| \, \d u \, \d^\times a
\]
and
\[
\mathfrak{T}_1(R_0\Phi_\bK,s)= \int_{\A^\times}\int_\A R_0\Phi_\bK(a,au)\, |a|^{2s} \, \d u \, \d^\times a .
\]
It is clear that they are absolutely convergent and holomorphic for $\Re(s)>1$.
We also easily see that these functions can be meromorphically continued to the whole complex $s$-plane and satisfy the functional equations
\[
\mathfrak{T}_0(R_0\Phi_\bK,s)=-\mathfrak{T}_0\Bigl(\widetilde{R_0\Phi_\bK},\frac{3}{2}-s\Bigr),\;\,
\mathfrak{T}_1(R_0\Phi_\bK,s)=\mathfrak{T}_1\Bigl(\widetilde{R_0\Phi_\bK},\frac{3}{2}-s\Bigr),
\]
where
\[
\widetilde{R_0\Phi_\bK}(z,y)=\int_\A R_0\Phi_\bK(x,y) \, \psi_F(xz) \d x.
\]
The function $\mathfrak{T}_0(R_0\Phi_\bK,s)$ (resp. $\mathfrak{T}_1(R_0\Phi_\bK,s)$) has only double (resp.\ simple) poles at $s=1/2$ and $1$.
\begin{prop}\label{4p26}
Let $T_j\in \R$ and $T_j>0$ $(j=1,2)$.
Set
\[ \mathfrak{T}_2(R_0\Phi_\bK,s,T_2)=\mathfrak{T}(R_0\Phi_\bK,s)+\mathfrak{T}_0(R_0\Phi_\bK,s)+2T_2 \, \mathfrak{T}_1(R_0\Phi_\bK,s)  \]
and
\[ \mathfrak{T}_3(R_0\Phi_\bK,s,T_2)=-\mathfrak{T}_0(R_0\Phi_\bK,s)+2T_2
\, \mathfrak{T}_1(R_0\Phi_\bK,s) . \]
If we assume $\hat\Phi(0)=0$, then we have
\begin{multline*}
 Z^{\GSp(2)}(\Phi,3/2,T_1)+\frac{\vol_M\, \vol_{A_G}}{2 \, c_F} \mathfrak{T}_2(R_0\Phi_\bK,3/2,T_2)  \\
= Z^{\GSp(2)}_\mathrm{ad}(\hat\Phi,0) + \frac{\vol_M\, \vol_{A_G}}{2 \,
c_F} \mathfrak{T}_3(\widetilde{R_0\Phi_\bK},0,T_2)
\end{multline*}
and
\begin{multline*}
Z^{\Sp(2)}(\Phi,3/2,T_1)+\frac{\vol_M}{2\, c_F} \mathfrak{T}_2(R_0\Phi_\bK,3/2,T_2) \\
= Z^{\Sp(2)}_\mathrm{ad}(\hat\Phi,0)+\frac{\vol_M}{2\,c_F} \mathfrak{T}_3(\widetilde{R_0\Phi_\bK},0,T_2) . 
\end{multline*}
\end{prop}
\begin{proof}
The first formula is derived from Lemma \ref{4l20}, (\ref{4e3}), (\ref{4e4}), and Theorem \ref{4t25}.
By \cite[Theorem 4.2]{Yukie} we know that $Z(\Phi,s,\chi)=Z(\hat\Phi,\frac{3}{2}-s,\chi)$ if $\chi$ is a non-trivial quadratic character.
Hence, the second formula follows.
\end{proof}

\section{Structure of $\GSp(2)$}\label{5s}

Throughout this section, we denote
\[
G=\GSp(2).
\]
Most arguments and results can be easily transferred to $\Sp(2)$.

\subsection{Setup}\label{5s1}

For each $v\in \Sigma_\inf$ we set
\[
\bK_v=\left\{ \begin{pmatrix}A&cB \\
-\overline{B}& c\overline{A}\end{pmatrix}  \in G(F_v) \, \Bigl| \, \begin{array}{c} A\,^tB=B\,^tA , \\ A\,^t\overline{A}+B\,^t\overline{B}=I_2 , \; \, c\in\C^1 \end{array} \right\}.
\]
We see that $\bK_v$ is isomorphic to the semi-direct product of $\U(1)$ and the compact real form of $\Sp(2)$ if $v\in\Sigma_\C$, and $\bK_v\cong \{\pm 1\}\ltimes \U(2)$ if $v\in\Sigma_\R$.
For each $v\in\Sigma_\fin$ we set
\[
\bK_v= \GSp(2,\fO_v).
\]
The group $\bK_v$ is a maximal compact subgroup of $G(F_v)$ for any $v\in \Sigma$.
We set
\[
\bK=\prod_{v\in\Sigma}\bK_v.
\]
The group $\bK$ is a maximal compact subgroup of $G(\A)$.
We fix the minimal Levi subgroup $M_0$ consisting of all diagonal matrices contained in~$G$.
Note that $\bK$ is admissible relative to $M_0$ in the sense of~\cite[Section~1]{Arthur2}.
We fix the minimal Levi subgroup $M_0$ and the minimal parabolic subgroup $P_0$ as
\[
M_0=\left\{ \begin{pmatrix} *&0&0&0 \\ 0&*&0&0 \\ 0&0&*&0 \\ 0&0&0&* \end{pmatrix} \in G \right\} \quad \text{and} \quad  P_0=\left\{ \begin{pmatrix} *&*&*&* \\ 0&*&*&* \\ 0&0&*&0 \\ 0&0&*&* \end{pmatrix} \in G \right\}.
\]

A homomorphism $\mu \, : \, G\to \mathbb{G}_m$ over $F$ is defined by
\[
g\begin{pmatrix}O_2&I_2 \\ -I_2&O_2 \end{pmatrix}\,^tg=\mu(g) \begin{pmatrix}O_2&I_2 \\ -I_2&O_2 \end{pmatrix}.
\]
A homomorphism $\chi_j$ $(j=1$ or $2)$ over $F$ is defined by
\[
\chi_j(\diag(a_1,a_2,a_1^{-1}a_0,a_2^{-1}a_0))=a_j . 
\]
Then, $\{ \mu$, $\chi_1$, $\chi_2 \}$ is a basis of $X(M_0)_F$.
We define $e_0$, $e_1$, $e_2\in\a_0$ by
\[
e_0(\mu^{j_0}\chi_1^{j_1}\chi_2^{j_2})=j_0,\quad  e_1(\mu^{j_0}\chi_1^{j_1}\chi_2^{j_2})=j_1 , \quad \text{and} \quad e_2(\mu^{j_0}\chi_1^{j_1}\chi_2^{j_2})=j_2 . 
\]
It is clear that
\[
\a_0=\R e_0\oplus\R e_1\oplus\R e_2 \quad \text{and} \quad \a_0^*=\R\mu\oplus\R\chi_1\oplus\R\chi_2 .
\]

We set
\[
P_1=\left\{ \begin{pmatrix} *&*&*&* \\ *&*&*&* \\ 0&0&*&* \\ 0&0&*&* \end{pmatrix} \in G \right\} \quad \text{and} \quad  P_2=\left\{ \begin{pmatrix} *&*&*&* \\ 0&*&*&* \\ 0&0&*&0 \\ 0&*&*&* \end{pmatrix} \in G \right\} .
\]
Then, we easily see $\{P\in\mathcal{P}\mid P\supset P_0\}=\{P_0,P_1,P_2,G\}$.
Put
\[
M_1=M_{P_1} \quad \mathrm{and} \quad M_2=M_{P_2}.
\]
Let $\alpha_1$ (resp. $\alpha_2$) denote the long (resp. short) root in $\Delta_0$ corresponding to $P_0$, i.e., 
\[
\Delta_0=\{\alpha_1,\alpha_2\} \quad (\alpha_1=2\chi_2, \; \alpha_2=\chi_1-\chi_2).
\]
The symmetry with $\alpha$ is denoted by $s_\alpha$.
We set
\[
s_0=s_{\alpha_2} , \quad s_1=s_{\alpha_1+2\alpha_2} , \quad  s_2=s_{\alpha_1}.
\]
The Weyl group $W_0$ is generated by $s_0$ and $s_2$.
We see that
\[
W_0=\{ 1 ,\; s_0 ,\; s_1,\; s_2,\; s_0s_1,\; s_0s_2,\; s_1s_2,\; s_0s_1s_2 \}.
\]
The elements satisfy the relations
\[
s_0^2=1, \quad s_1^2=1, \quad s_2^2=1, 
\]
\[
s_0s_1=s_2s_0, \quad s_0s_2=s_1s_0, \quad s_1s_2=s_2s_1.
\]
We choose the representatives
\[
w_0=\begin{pmatrix}0&1&0&0 \\ 1&0&0&0 \\ 0&0&0&1 \\ 0&0&1&0 \end{pmatrix}  , \; w_1=\begin{pmatrix}0&0&1&0 \\ 0&1&0&0 \\ -1&0&0&0 \\ 0&0&0&1 \end{pmatrix}  ,  \; w_2=\begin{pmatrix}1&0&0&0 \\ 0&0&0&1 \\ 0&0&1&0 \\ 0&-1&0&0 \end{pmatrix}.
\]
There is an injection $W_0\to G(F)\cap \bK$ mapping $s_i$ to~$w_i$ for $0\le i\le2$.
Let $w_s$ denote the image of $s\in W_0$ under this injection.
For $s\in W_0$ and $\mathcal{G}\subset G$, we set
\[
s\mathcal{G}=w_s\mathcal{G}w_s^{-1}.
\]
Then, we have
\begin{align*}
& \L=\{  M_0 ,\, M_1 ,\, s_1M_1 ,\, M_2 ,\, s_0M_2 ,\, G  \} , \\
& \L(M_1)=\{M_1,\, G\}, \quad \L(M_2)=\{M_2,\, G\} .
\end{align*}
Furthermore, we obtain
\begin{align*}
& \P= \{  sP_0 \, | \, s\in W_0  \},  \quad \P(M_1)= \{ P_1 , \, s_1s_2P_1 \} ,  \quad \P(M_2)= \{ P_2 , \, s_1P_2 \} \\
& \P(s_1M_1)= \{ s_1P_1 , \, s_2P_1 \} , \quad \P(s_0M_2)= \{ s_0P_2 , \, s_0s_1P_2 \} .
\end{align*}

\subsection{Weight factors for unipotent elements}\label{5s2}

Now we compute the weight factors $w_M(1,uv,T)$ for $\GSp(2,F_S)$ as defined in Section~\ref{2s4}.
The weight factors for $\Sp(2,F_S)$ are then obtained by restriction.

From now on, we fix Haar measures on $\a_0^G$, $\a_{M_1}^G$, and $\a_{M_2}^G$.
\begin{cond}\label{5c1}
We choose the Haar measure $\d H$ on $\a_0^G$ such that
\[
\d H=\d r_1  \, \d r_2  \quad \text{for $H=r_1e_1+ r_2e_2\in \a_0^G$}
\]
where $\d r_1$ and $\d r_2$ are the Lebesgue measure on $\R$.
A Haar measure on $\a_{M_1}^G$ (resp. $\a_{M_2}^G$) is fixed by the Lebesgue measure $\d r$ on $\R$ for $r(e_1+e_2)/2\in\a_{M_1}^G$ (resp. $re_1\in\a_{M_2}^G$).
\end{cond}
Under this condition and with our choice of coroots, we have
\[
\vol(\a^G_0/\Z(\Delta_0^\vee))=\vol(\a^G_{M_1}/\Z(\Delta_{P_1}^\vee))=\vol(\a^G_{M_2}/\Z(\Delta_{P_2}^\vee))=1
\]
where $\Delta_P^\vee$ is the basis of $\a_P^G$ dual to $\widehat\Delta_P$.
Now, for each $M=M_0$, $M_1$, and $M_2$, we have fixed the Haar measure on $A_M^+/A_G^+$ by the map $H_M$ and the Haar measure on $\a_M^G$.

We set
\[
\alpha_1^\vee=e_2 ,\quad \alpha_2^\vee=e_1-e_2 \in \a_0^G.
\]
Then, $\Delta_0^\vee=\{ \alpha_1^\vee \, , \; \alpha_2^\vee \}$ is the set of simple coroots corresponding to $\Delta_0$.
We also put
\[
\varpi_1=\alpha_1+\alpha_2=\chi_1+\chi_2 ,\quad \varpi_2=\frac{1}{2}\alpha_1+\alpha_2=\chi_1.
\]
Then, $\widehat\Delta_0=\{  \varpi_1\, , \; \varpi_2\}$ is the set of simple weight corresponding to $\Delta_0$.
We set
\[
T=T_1 \alpha_1^\vee +T_2\alpha_2^\vee\in \a_0^+,
\]
where $T_1$, $T_2\in\R$, and
\begin{equation}\label{u}
\nu(n_{12},n_{13},n_{14},n_{24})=\begin{pmatrix}1&n_{12}&n_{13}&n_{14} \\ 0&1&n_{14}-n_{12}n_{24}&n_{24} \\ 0&0&1&0 \\ 0&0&-n_{12}&1 \end{pmatrix}\in N_{P_0}. 
\end{equation}

\begin{lem}\label{5l3}
If $n=\nu(n_{12},n_{13},n_{14},n_{24})$ and $n_{23}=n_{14}-n_{12}n_{24}$, then
\begin{multline*}
v_{sP_0}(\lambda,n,T)= \\
\begin{cases}
e^{\lambda_1T_1+\lambda_2T_2}  & \text{if $s=1$,} \\
\|(1,n_{12})\|^{-\lambda_2}\, e^{\lambda_1T_1+\lambda_2(T_1-T_2)}  & \text{if $s=s_0$,} \\
\|(1,n_{24})\|^{-\lambda_1}\, e^{-\lambda_1(T_1-2T_2)+\lambda_2T_2}  & \text{if $s=s_2$,} \\
\|(1,n_{24})\|^{\lambda_1+\lambda_2}  \|(1,n_{23},n_{24})\|^{-2\lambda_1-\lambda_2}  & \\
\qquad \times e^{\lambda_1(T_1-2T_2)+\lambda_2(T_1-T_2)}  & \text{if $s=s_0s_1$,} \\
\|(1,n_{12},n_{12},n_{12}^2,n_{13}+n_{12}n_{14})\|^{-\lambda_1-\lambda_2} \|(1,n_{12})\|^{2\lambda_1} & \\
\qquad \times e^{-\lambda_1(T_1-2T_2)-\lambda_2(T_1-T_2)}  & \text{if $s=s_0s_2$,} \\
\|(1,n_{12},n_{12},n_{12}^2,n_{13}+n_{12}n_{14})\|^{\lambda_1}  & \\
\qquad \times \|(1,n_{12},n_{13},n_{14})\|^{-2\lambda_1-\lambda_2}e^{\lambda_1(T_1-2T_2)-\lambda_2T_2}  & \text{if $s=s_1$,} \\
\|(1,n_{23},n_{23},n_{24},n_{13}-n_{12}n_{23},n_{13}n_{24}-n_{14}n_{23})\|^{-\lambda_1-\lambda_2} & \\
\qquad \times \|(1,n_{23},n_{24})\|^{\lambda_2} \, e^{-\lambda_1T_1-\lambda_2(T_1-T_2)}  & \text{if $s=s_0s_1s_2$,} \\
\|(1,n_{23},n_{23},n_{24},n_{13}-n_{12}n_{23},n_{13}n_{24}-n_{14}n_{23})\|^{-\lambda_1} \\
\qquad \times \|(1,n_{12},n_{13},n_{14})\|^{-\lambda_2}\, e^{-\lambda_1T_1-\lambda_2T_2}  & \text{if $s=s_1s_2$}
\end{cases}
\end{multline*}
where $\lambda =\lambda_1\varpi_1+\lambda_2\varpi_2\in \a^*_{0,\C}$.
\end{lem}
\begin{proof}
We use the tautological right action of $G(F_S)$ on the space $F_S^4$ of row vectors endowed with the $\bK\cap G(F_S)$-invariant height function. Since $\mathbf e_3=(0,0,1,0)$ is an eigenvector of the elements of~$P_2$, we get
\[
v_{P_2}(\chi_1-\mu,g)=\|\mathbf e_3g\|=\|(g_{31},g_{32},g_{33},g_{34})\|.
\]
Moreover, we use the right action of $G$ on the exterior square of the space of row vectors. If $\mathbf e_4=(0,0,0,1)$, then $\mathbf e_3\wedge \mathbf e_4$ is an eigenvector of the elements of~$P_1$, and
\begin{multline*}
v_{P_1}(\chi_1+\chi_2-2\mu,g)=\|(\mathbf e_3\wedge\mathbf e_4)g\|=\\
\left\|\left(
\begin{vmatrix}
g_{31}&g_{32}\\
g_{41}&g_{42}
\end{vmatrix},
\begin{vmatrix}
g_{31}&g_{33}\\
g_{41}&g_{43}
\end{vmatrix},
\begin{vmatrix}
g_{31}&g_{34}\\
g_{41}&g_{44}
\end{vmatrix},
\begin{vmatrix}
g_{32}&g_{33}\\
g_{42}&g_{43}
\end{vmatrix},
\begin{vmatrix}
g_{32}&g_{34}\\
g_{42}&g_{44}
\end{vmatrix},
\begin{vmatrix}
g_{33}&g_{34}\\
g_{43}&g_{44}
\end{vmatrix}
\right)\right\|.
\end{multline*}
We need only the case $g\in G(F_S)^1$, in which $\mu$ can be omitted.
For general $\lambda=\lambda_i\varpi_i\in\a_{M_i,\C}^*$, where $i=1$ or~$2$, we have
\[
v_{P_i}(\lambda,g)=v_{P_i}(\varpi_i,g)^{\lambda_i} .
\]
Moreover, $H_{P_0}(g)$ is determined by its projections $H_{P_1}(g)$ and $H_{P_2}(g)$, hence
\[
v_{P_0}(\lambda,g)=v_{P_1}(\varpi_1,g)^{\lambda_1}v_{P_2}(\varpi_2,g)^{\lambda_2}.
\]
Finally, for $s\in W_{M_i}$, we have $v_{sP}(s\lambda,w_sg)=v_P(\lambda,g)$.
\end{proof}
The proof yields formulas for the case of maximal parabolics as intermediate results. We can also read them off from the lemma if we set $n_{12}=0$ resp.~$n_{24}=0$. Namely, if $\lambda=\lambda_1\varpi_1\in\a_{M_1,\C}^*$, then
\begin{align*}
& v_{P_1}(\lambda,n,T)=e^{\lambda_1 T_1},\\
& v_{s_1s_2P_1}(\lambda,n,T)=\|(1,n_{13},n_{14},n_{23},n_{24},n_{13}n_{24}-n_{23}n_{14})\|^{-\lambda_1}\, e^{-\lambda_1 T_1},
\end{align*}
while for $\lambda=\lambda_2\varpi_2\in\a_{M_2,\C}^*$ we get
\[
v_{P_2}(\lambda,n,T)=e^{\lambda_2T_2},\quad
v_{s_1P_1}(\lambda,n,T)=\|(1,n_{12},n_{13},n_{14})\|^{-\lambda_2}\, e^{-\lambda_2T_2}.
\]

\begin{prop}\label{5p4}
Assume Condition \ref{5c1} and $T=T_1\alpha_1^\vee +T_2\alpha_2^\vee\in \a_0^+$.
For $\nu=\nu(\nu_{12},\nu_{13},\nu_{14},\nu_{24})\in N_{P_0}(F_S)$ we have
\begin{align*}
  w_{M_0}(1,\nu,T)=& \, 2 (\log|\nu_{12}|_S)^2+(\log|\nu_{24}|_S)^2+4(\log|\nu_{12}|_S)(\log|\nu_{24}|_S) \\
& + 4T_1\log|\nu_{12}|_S + 4T_2\log|\nu_{24}|_S +8T_1T_2-2T_1^2-4T_2^2   . 
\end{align*}
\end{prop}
\begin{proof}
Set $\nu_{23}=\nu_{14}-\nu_{12}\nu_{24}$, $n=\nu(n_{12},n_{13},n_{14},n_{24})$, and $n_{23}=n_{14}-n_{12}n_{24}$.
If $a\nu=n^{-1}an$, then we can express $n$ in terms of $\nu$ as
\begin{align*}
n_{12}&=\left(1-\frac{a_2}{a_1}\right)^{-1}\nu_{12},\qquad
n_{24}=\left(1-\frac{a_0}{a_2^2}\right)^{-1}\nu_{24},\\
n_{23}&=\left(1-\frac{a_0}{a_2^2}\right)^{-1}\left(1-\frac{a_0}{a_1a_2}\right)^{-1}\left(\nu_{23}-\frac{a_0}{a_2^2}\nu_{14}\right),\\
n_{14}&=\left(1-\frac{a_2}{a_1}\right)^{-1}\left(1-\frac{a_0}{a_1a_2}\right)^{-1}\left(\nu_{14}-\frac{a_2}{a_1}\nu_{23}\right),\\
n_{13}&=\left(1-\frac{a_0}{a_1^2}\right)^{-1}\\
& \quad \times \left(\nu_{13}
+\left(1-\frac{a_2}{a_1}\right)^{-1}\left(1-\frac{a_0}{a_1a_2}\right)^{-1}\nu_{12}\left(\frac{a_2}{a_1}\nu_{23}-\frac{a_0}{a_1a_2}\nu_{14}\right)\right).
\end{align*}
This also implies that
\[
n_{13}+n_{12}n_{14}=\left(1-\frac{a_0}{a_1^2}\right)^{-1}\left(\nu_{13}
+\left(1-\frac{a_2}{a_1}\right)^{-2}
\nu_{12}\left(\nu_{14}-\frac{a_2^2}{a_1^2}\nu_{23}\right)\right),
\]
\begin{multline*}
n_{13}-n_{12}n_{23}
=\left(1-\frac{a_0}{a_1^2}\right)^{-1}\\
\times\left(\nu_{13}
-\left(1-\frac{a_0}{a_2^2}\right)^{-1}\left(1-\frac{a_0}{a_1a_2}\right)^{-1}\left(1+\frac{a_0}{a_1a_2}\right)
\nu_{12}\left(\nu_{23}-\frac{a_0}{a_2^2}\nu_{14}\right)\right),
\end{multline*}
\begin{multline*}
n_{13}n_{24}-n_{14}n_{23}
=\left(1-\frac{a_0}{a_1^2}\right)^{-1}\left(1-\frac{a_0}{a_2^2}\right)^{-1}\\
\quad\times\left(\nu_{13}\nu_{24}+\left(1-\frac{a_0}{a_1a_2}\right)^{-2}
\left(\frac{a_0}{a_1^2}\nu_{23}^2+\frac{a_0}{a_2^2}\nu_{14}^2-\left(1+\frac{a_0^2}{a_1^2a_2^2}\right)\nu_{23}\nu_{14}\right)\right).
\end{multline*}
Let $\lambda =\lambda_1\varpi_1+\lambda_2\varpi_2\in \a^*_{0,\C}$.
According to the definition of $w_{P}(\lambda,a,\nu,T)$ given in Section~\ref{2s4}, we have to multiply the functions $v_{sP_0}(\lambda,n,T)$ computed in~Lemma~\ref{5l3} with suitable factors $r_\beta(\lambda,1,a)$. Taking the limit as $a\to1$, we obtain
\begin{multline*}
w_{sP_0}(\lambda,1,\nu,T)= \\
\begin{cases}
e^{\lambda_1T_1+\lambda_2T_2} & \text{if $s=1$}, \\
|\nu_{12}|_S^{-\lambda_2}\, e^{\lambda_1T_1+\lambda_2(T_1-T_2)} & \text{if $s=s_0$}, \\
|\nu_{24}|_S^{-\lambda_1} \, e^{-\lambda_1(T_1-2T_2)+\lambda_2T_2} & \text{if $s=s_2$}, \\
|\nu_{12}^2\nu_{24}|_S^{-\lambda_1} \, |\nu_{12}|_S^{-\lambda_2} \, e^{\lambda_1(T_1-2T_2)+\lambda_2(T_1-T_2)} & \text{if $s=s_0s_1$} ,  \\
|\nu_{24}|_S^{-\lambda_1} \, |\nu_{12}\nu_{24}|_S^{-\lambda_2} \, e^{-\lambda_1(T_1-2T_2)-\lambda_2(T_1-T_2)} & \text{if $s=s_0s_2$}, \\
|\nu_{12}^2\nu_{24}|_S^{-\lambda_1} \, |\nu_{12}^2\nu_{24}|_S^{-\lambda_2} \, e^{\lambda_1(T_1-2T_2)-\lambda_2T_2} & \text{if $s=s_1$}, \\
|\nu_{12}^2\nu_{24}^2|_S^{-\lambda_1} \, |\nu_{12}\nu_{24}|_S^{-\lambda_2} \, e^{-\lambda_1T_1-\lambda_2(T_1-T_2)}  & \text{if $s=s_0s_1s_2$} \\
|\nu_{12}^2\nu_{24}^2|_S^{-\lambda_1} \, |\nu_{12}^2\nu_{24}|_S^{-\lambda_2} \, e^{-\lambda_1T_1-\lambda_2T_2} & \text{if $s=s_1s_2$}. \end{cases}
\end{multline*}
By the definition, we have
\[ \theta_{sP_0}(\lambda)=\begin{cases}
\lambda_1\lambda_2 & \text{if $s=1$}, \\
(\lambda_1+\lambda_2)(-\lambda_2) & \text{if $s=s_0$}, \\
(-\lambda_1)(2\lambda_1+\lambda_2) & \text{if $s=s_2$}, \\
(\lambda_1+\lambda_2)(-2\lambda_1-\lambda_1) & \text{if $s=s_0s_1$}, \\
(-\lambda_1-\lambda_2)(2\lambda_1+\lambda_2) & \text{if $s=s_0s_2$}, \\
\lambda_1(-2\lambda_1-\lambda_2) & \text{if $s=s_1$}, \\
(-\lambda_1-\lambda_2)\lambda_2 & \text{if $s=s_0s_1s_2$}, \\
(-\lambda_1)(-\lambda_2) & \text{if $s=s_1s_2$} . 
\end{cases}   \]
The explicit formula for $w_{M_0}(1,\nu,T)$ can now be computed using~\cite[Lemma~3.4]{Arthur0}.
\end{proof}
\begin{prop}\label{5p5}
We assume Condition \ref{5c1} and $T=T_1\alpha_1^\vee +T_2\alpha_2^\vee\in \a_0^+$.
For $u=\nu(u_{12},0,0,0)\in M_1(F_S)$ and $\nu=\nu(0,\nu_{13},\nu_{14},\nu_{24})\in N_{P_1}(F_S)$, we have
\begin{align*}
w_{M_1}(1,\nu,T) =& \, \log|\det(Y)|_S +2T_1 , \\
w_{M_1}(1,u\nu,T)=& \, \log|u_{12}^2 \nu_{24}^2|_S +2T_1 =2\log|u_{12}|_S+2\log|\nu_{24}|_S+2T_1. 
\end{align*}
For $u=\nu(0,0,0,u_{24})\in M_2(F_S)$ and $\nu=\nu(\nu_{12},\nu_{13},\nu_{14},0)\in N_{P_2}(F_S)$, we have
\begin{align*}
w_{M_2}(1,\nu,T)=& \, \log\|(\nu_{12},\nu_{13}/2,\nu_{14})\|_S+2T_2  , \\
w_{M_2}(1,u\nu,T)=& \, \log|\nu_{12}^2u_{24}/2|_S  +2T_2 \\
 =& \,  2\log|\nu_{12}|_S+ \log|u_{24}|_S-\log|2|_S  +2T_2. 
\end{align*}
\end{prop}
\begin{proof}
This follows in a similar way from (\ref{2e1}) and Lemma~\ref{5l3}.
\end{proof}

\subsection{$\cO$-equivalence classes}\label{5s3}

We determine the $\cO$-equivalence classes (cf.~Section~\ref{coarse}) containing non-semisimple elements of $G(F)$.

\begin{lem}\label{5l}
Let $\mathcal G$ be a connected reductive algebraic group over $F$.
Fix a minimal Levi subgroup $\mathcal M_0$.
Assume that $\fo\in\cO^{\mathcal G}$ contains a non-semisimple element of $\mathcal G(F)$.
Then, there exists a semisimple element $\gamma\in \fo$ such that $\gamma\in M(F)$ for a certain $M\in\mathcal L(\mathcal M_0)$ $(M\neq \mathcal G)$.
\end{lem}
\begin{proof}
Let $\delta$ be a non-semisimple element of $\fo$.
Let $\delta_s$ (resp. $\delta_u$) denote the semisimple (resp. unipotent) part of the Jordan decomposition of $\delta$.
Then, we have $\delta=\delta_s\delta_u$ and $\delta_u\neq 1$.
A Jacobson-Morozov parabolic subgroup $Q\neq \mathcal G$ of $\delta_u$ over $F$ is uniquely determined.
For $g\in \mathcal G_{\delta_u}(F)$, we have $\delta_u=g^{-1}\delta_u g$.
By using the uniqueness of $Q$ we find that $Q=g^{-1}Qg$.
This means that $Q(F)$ contains $\mathcal G_{\delta_u}(F)$.
Hence, $\delta_s\in Q(F)$.
Thus, $\delta_s$ is contained in a Levi subgroup of $Q$ over $F$ because $\delta_s$ is semisimple.
This lemma follows from the fact that any Levi subgroup of $Q$ over $F$ is $Q(F)$-conjugate to $M\supset \mathcal M_0$.
\end{proof}

Let $\fo\in\cO^G$.
If $\fo$ contains only semisimple elements, then $\fo$ is just a $G(F)$-conjugacy class.
Hence, we do not consider such $\cO$-equivalence classes.
Let $z\in Z(F)$ and $x\in F^\times$.
For $x\neq 1$, we set
\[
\sigma_{1,z}=z\,\diag(-1,1,-1,1)\quad \text{and}\quad \sigma_{2,z,x}=z\,\diag(x,x,1,1)  .
\]
For $x^2\neq 1$, we put
\[
\sigma_{3,z,x}=z\, \diag(x,1,x^{-1},1).
\]
For $\alpha\in F^\times-(F^\times)^2$, we set
\[
h_\alpha=\begin{pmatrix}0&1\\ \alpha&0 \end{pmatrix}  \quad \text{and} \quad \sigma_{4,z,\alpha}=z\begin{pmatrix}h_\alpha&O_2 \\ O_2& ^th_\alpha\end{pmatrix}.
\]
If $(x,y)\in F\oplus F$ satisfies $x^2-\alpha y^2\neq 0$ (resp. $x^2-\alpha y^2=1$), then we set
\[
\sigma_{5,\alpha,x,y} = \begin{pmatrix} xI_2+ y h_\alpha & O_2 \\ O_2 & xI_2-\alpha y h_{\alpha^{-1}}  \end{pmatrix}
\]
\[
(\text{resp.}\quad   \sigma_{6,z,\alpha,x,y} =z\begin{pmatrix}1&0&0&0 \\ 0&x&0&y \\ 0&0&1&0 \\ 0&\alpha y&0&x   \end{pmatrix} ).
\]
The $\cO$-equivalence class containing $z$ in $G(F)$ is denoted by $\fo_z$.
Let $\fo_{1,z}$, $\fo_{2,z,x}$, $\fo_{3,z,x}$, $\fo_{4,z,\alpha}$, $\fo_{5,\alpha,x,y}$, and $\fo_{6,z,\alpha,x,y}$ denote the $\cO$-equivalence class containing $\sigma_{1,z}$, $\sigma_{2,z,x}$, $\sigma_{3,z,x}$, $\sigma_{4,z,\alpha}$, $\sigma_{5,\alpha,x,y}$, and $\sigma_{6,z,\alpha,x,y}$ respectively.
We can easily prove the following propositions by Lemma \ref{5l} and direct calculation.
\begin{prop}\label{5pm0}
Assume that $\fo$ contains a non-semisimple element of $G(F)$ and a semisimple element of $M_0(F)$.
Then, there exist $z\in Z(F)$ and $x\in F^\times$ such that $\fo=\fo_z$, $\fo_{1,z}$, $\fo_{2,z,x}$ $(x\neq 1)$, or $\fo_{3,z,x}$ $(x^2\neq 1)$.
Moreover, we have
\[
G_{z}=G_{z,+}=G ,
\]
\[
G_{\sigma_{1,z}}=G_{\sigma_{1,z},+}\cong\{ (g_1,g_2)\in \GL(2)\times \GL(2) \, | \, \det(g_1)=\det(g_2)   \} ,
\]
\[
G_{\sigma_{2,z,x}}=G_{\sigma_{2,z,x},+}\cong \GL(1)\times \GL(2),
\]
\[
G_{\sigma_{3,z,x}}=G_{\sigma_{3,z,x},+}\cong\GL(1)\times \GL(2).
\]
\end{prop}
\begin{prop}\label{5pm1}
Assume that $\fo$ contains a non-semisimple element of $G(F)$ and a semisimple element of $M_1(F)$. We also assume that $\fo$ does not have any element of $M_0(F)$.
There exist $z\in Z(F)$, $\alpha\in F^\times-(F^\times)^2$, and $(x,y)\in F\oplus F$ $(x^2-\alpha y^2\neq 0)$ such that $\fo=\fo_{4,z,\alpha}$ or $\fo_{5,z,\alpha,x,y}$.
Let $E=F(h_\alpha)$.
Then
\begin{align*}
&G_{\sigma_{4,z,\alpha}}=G_{\sigma_{4,z,\alpha},+}\\
&=\left\{  \begin{pmatrix}A & B \\ C & D \end{pmatrix}\in G \, \Big| \, A,B\begin{pmatrix}0&1\\1&0\end{pmatrix},\begin{pmatrix}0&1\\1&0\end{pmatrix}C,D\in E  \right\} \\
& \cong \{  g\in R_{E/F}(\GL(2)) \, | \, \det(g)\in \mathbb G_m   \} 
\end{align*}
and
\begin{align*}
G_{\sigma_{5,z,\alpha,x,y}}=G_{\sigma_{5,z,\alpha,x,y}, +} &= \left\{ \begin{pmatrix}A&O_2 \\ O_2& \,^tA^{-1} \end{pmatrix} \begin{pmatrix}aI_2 & b\mathcal D \\ c\mathcal D' & dI_2 \end{pmatrix}\in G \, \Big| \, A \in E \right\} \\
& \cong \mathrm{GU}(1,1,E/F) 
\end{align*}
where $\mathcal D=\begin{pmatrix}-1&0 \\ 0&\alpha \end{pmatrix}$ and $\mathcal D'=\begin{pmatrix}-\alpha&0 \\ 0&1 \end{pmatrix}$.
\end{prop}
\begin{prop}\label{5pm2}
Assume that $\fo$ contains a non-semisimple element of $G(F)$ and a semisimple element of $M_2(F)$. We also assume that $\fo$ does not have any element of $M_0(F)$.
Then, there exist $z\in Z(F)$, $\alpha\in F^\times-(F^\times)^2$, and $(x,y)\in F\oplus F$ $(x^2-\alpha y^2=1)$ such that $\fo=\fo_{6,z,\alpha,x,y}$.
Let $E=F(h_\alpha)$.
Then we have
\begin{align*}
G_{\sigma_{6,z,\alpha,x,y}}&=G_{\sigma_{6,z,\alpha,x,y},+} \\
&\cong \left\{  (x,g)\in R_{E/F}(\mathbb G_m) \times \GL(2) \, | \, N_{E/F}(x)=\det(g) \right\} . 
\end{align*}
\end{prop}

\section{The geometric side of the trace formula for $\GSp(2)$}\label{6s}

Throughout this section, we set $G=\GSp(2)$ and we keep the notations and the assumptions of Section \ref{5s}.
In particular, we assume Condition \ref{5c1} which fixes Haar measures on $\a_0^G$, $\a_{M_1}^G$, and $\a_{M_2}^G$.

Let $\mathcal G$ be a connected reductive algebraic group over $F$.
For $g\in \mathcal G(F)$, we denote by $\{g\}_{\mathcal G}$ the $\mathcal G(F)$-conjugacy class of $g$.

\subsection{Unipotent contribution}\label{6s1}

We use the notations $u_\alpha$ (cf. (\ref{1e1})), $n_\min(\alpha)$, $n_\sub(x)$, $n_\reg(\alpha)$ (cf. (\ref{1e2})), $\xd_d=\diag(1,-d)$ (cf. (\ref{xd})), and $\para(F)$ (cf. (\ref{para})).
Let $Z$ denote the center of $G$.
For $z\in Z(F)$, $\fo_z$ denotes the $\cO$-equivalence class containing $z$ in $G(F)$.
The set $\fo_z$ is divided into the five classes
\[ \fo_z=O_{\tri,z}\cup O_{\min,z} \cup O_{\sub,z}\cup O_{\sub,z}' \cup O_{\reg,z}  \]
where we set 
\[ O_{\tri,z}=\{ z \} ,\quad  O_{\min,z}=\{z\, n_\min(1)\}_G, \quad O_{\sub,z}=\bigcup_{\ds\in \para(F)} \{z \, n_\sub(\xd_d) \}_G,  \]
\[  O_{\sub,z}'= \{z \, n_\sub(\xd_1) \}_G, \quad O_{\reg,z}=\{z\, n_\reg(1)\}_G . \]
Let $P\in\F$, $P\supset P_0$, and $f\in C_c^\inf(G(\A)^1)$.
For a subset $O$ in $M_P(F)$ we set
\begin{equation}\label{6ekk}
 K_{P,O}(g,h)=\int_{N_P(\A)} \sum_{\gamma\in O} f(g^{-1}\gamma n h) \d n . 
\end{equation}
For $T\in \a_0^+$, integrals $J_{O_{\reg,z}}^T(f)$, $J_{O_{\sub,z}}^T(f)$, $J_{O_{\sub,z}'}^T(f)$, $J_{O_{\min,z}}^T(f)$, and $J_{O_{\tri,z}}^T(f)$ are defined as follows:
\[  J_{O_{\tri,z}}^T(f)=\int_{G(F)\bsl G(\A)^1}   K_{G,O_{\tri,z}}(g,g)  \, \d^1 g , \]
\[  J_{O_{\min,z}}^T(f)=\int_{G(F)\bsl G(\A)^1}   K_{G,O_{\min,z}}(g,g)  \, \d^1 g , \]
\begin{align*}
 J_{O_{\sub,z}}^T(f)=&\int_{G(F)\bsl G(\A)^1} \Big\{  K_{G,O_{\sub,z}}(g,g) \\
& \qquad  -\sum_{\delta\in P_1(F)\bsl G(F)}K_{P_1,\{z \}}(\delta g,\delta g) \, \widehat{\tau}_{P_1}(H_{P_1}(\delta g)-T) \Big\} \, \d^1 g ,
\end{align*}
\begin{align*}
J_{O_{\sub,z}'}^T(f)=&\int_{G(F)\bsl G(\A)^1} \Big\{  K_{G,O_{\sub,z}'}(g,g) \\
& \qquad  -\sum_{\delta\in P_2(F)\bsl G(F)}K_{P_2,\{z \}}(\delta g,\delta g) \, \widehat{\tau}_{P_2}(H_{P_2}(\delta g)-T) \Big\} \, \d^1 g ,
\end{align*}
and
\begin{align*}
J_{O_{\reg,z}}^T(f)=&\int_{G(F)\bsl G(\A)^1} \Big\{  K_{G,O_{\reg,z}}(g,g) \\
&\qquad  -\sum_{\delta\in P_1(F)\bsl G(F)}K_{P_1,\{z u_1\}_{M_1}}(\delta g,\delta g) \, \widehat{\tau}_{P_1}(H_{P_1}(\delta g)-T)       \\
&\qquad  -\sum_{\delta\in P_2(F)\bsl G(F)}K_{P_2,\{z u_1\}_{M_2}}(\delta g,\delta g) \, \widehat{\tau}_{P_2}(H_{P_2}(\delta g)-T)       \\
&\qquad  +  \sum_{\delta\in P_0(F)\bsl G(F)}K_{P_0,\{z \}}(\delta g,\delta g) \, \widehat{\tau}_{P_0}(H_{P_0}(\delta g)-T)  \Big\} \d^1 g .
\end{align*}
We will show that these integrals converge absolutely for any $T\in\a_0^+$ (cf. Theorems \ref{6t1}, \ref{6t2}, \ref{6t3}, and \ref{6t4}).
Hence, the integral $J_{\fo_z}^T(f)$ satisfies the equality
\[
J_{\fo_z}^T(f)=J_{O_{\tri,z}}^T(f)+J_{O_{\min,z}}^T(f)+J_{O_{\sub,z}}^T(f)+J_{O_{\sub,z}'}^T(f)+J_{O_{\reg,z}}^T(f).
\]
Clearly, we see that $J_{O_{\tri,z}}^T(f)=\vol_G \, f(z)$.
We set
\[
f_\bK(g)=\int_\bK f(k^{-1}gk)\, \d k \quad (f\in C_c^\inf(G(\A)^1),\; \, g\in G(\A)^1)
\]
and
\[
f_{\bK_S}(g_S)=\int_{\bK_S} f(k_S^{-1}g_Sk_S)\, \d k_S \quad (f\in C_c^\inf(G(F_S)^1),\; \,  g_S\in G(F_S)^1).
\]
The orbital integral over the $G(F_S)$-conjugacy class of $n_\min(1)$ endowed with a suitable measure is
\[
J_G(z\, n_\min(1),f)= c_S \, \int_{F_S}f_{\bK_S}(z\, n_\min(x)) \, |x|_S \,\d x
\]
where $\d x$ is the Haar measure on $F_S$ defined in Section \ref{2s1} and the constant $c_S$ was defined in Section \ref{2s2}.
\begin{thm}\label{6t1}
Assume that $S$ contains $\Sigma_\inf$ and $z\in Z(\fO_v)$ $(\forall v\not\in S)$.
Let $f\in C_c^\inf(G(F_S)^1)$.
The integral $J_{O_{\min,z}}^T(f)$ converges absolutely and satisfies
\begin{align*}
J_{O_{\min,z}}^T(f)&=\frac{ \vol_{M_2}}{2\, c_F} \int_{\A^\times}f_\bK (z\, n_\min(x))\,|x|^2 \d^\times x\\
&=\frac{ \vol_{M_2} \, \zeta_F^S(2) }{2\, c_F}  \, J_G(z\, n_\min(1),f) ,
\end{align*}
where $\d^\times x$ is the Haar measure on $\A^\times$ which was defined in Section \ref{2s1}.
\end{thm}
\begin{proof}
This theorem is easily proved if we consider the centralizer of $n_\min(1)$.
\end{proof}

If $S$ contains $\Sigma_\inf$, then we have
\[ J_{M_1}^T(z,f)=\int_{N_{P_1}(F_S)} f_{\bK_S}(z\, n) \, w_{M_1}(1,n,T) \, \d n  \]
(cf. Proposition \ref{5p5}).
The spaces $V$ and $V^\ss$ were defined in Section \ref{4s1}.
For $d_S\in F^\times$, we set $V^\ss(F_S,d_S)=\{ x\in V^\ss(F_S) \, | \,  -\det(x)\in d_S(F_S^\times)^2  \}$ (cf. Section \ref{4s4}).
Let $\d x$ be the Haar measure on $V(F_S)$ defined in Section \ref{4s4}.
For any $y\in V^\ss(F_S,d_S)$ and a suitably chosen measure on the $G(F_S)$-conjugacy class of $n_\sub(y)$, we have
\[ J_G(z\, n_\sub(y),f)=  2^{|S|} c_S \int_{V^\ss(F_S,d_S)}f_{\bK_S}(z\, n_\sub(x))  \, \d x, \]
where the constant $c_S$ was defined in Section \ref{2s2}.
Let $\sim_S$ denote the equivalence relation on $V^\ss(F_S)$ such that $x\sim_S y$ if and only if $\det(x^{-1}y)\in (F_S^\times)^2$.
Then, $V^\ss(F_S,d_S)$ is an equivalence class in $V^\ss(F_S)$ with respect to $\sim_S$.
Since any equivalence class contains an element in $V^\ss(F)$, we have $V^\ss(F_S)/{\sim_S}=V^\ss(F)/{\sim_S}$.
Note that there exists a bijection between $V^\ss(F)/{\sim_S}$ and $F^\times/(F^\times\cap(F_S^\times)^2)$.
\begin{thm}\label{6t2}
Assume that $S$ contains $\Sigma_\inf\cup\Sigma_2$ and $z\in Z(\fO_v)$ $(\forall v\not\in S)$.
Let $f\in C_c^\inf(G(F_S)^1)$.
The integral $J_{O_{\sub,z}}^T(f)$ converges absolutely.
If we identify $M_1$ with $G'=\GL(1)\times \GL(2)$ by \eqref{M1} and we set $\Phi_z(x)=f_\bK(z \, n_\sub(x))$, then we have
\begin{align*}
 J_{O_{\sub,z}}^T(f)=& \,  Z^{\GSp(2)}(\Phi_z,3/2,T_1) \\
=& \, \frac{\vol_{M_1}}{2}J_{M_1}^T(z,f) \\
& \, +\frac{ \vol_{M_1}}{2\, c_F} \sum_{x\in V^\ss(F)/{\sim_S}} \fC_F(S,-\det(x)) \, J_G(z \, n_\sub(x),f)
\end{align*}
where $Z^{\GSp(2)}(\Phi,s,T_1)$ and $\fC_F(S,\alpha)$ were defined in Section \ref{4s8}.
\end{thm}
\begin{proof}
The notation $V^\st(F)$ was defined in Section \ref{4s1}.
For $s\in\C$, we set
\[
\lambda=(2s-3)\varpi_1\in \a^*_{M_1,\C}
\]
throughout this proof.
We consider the zeta integral
\begin{multline*}
J_{O_{\sub,z}}^T(f,\lambda)=\int_{P_1(F)\bsl G(\A)^1} \Big( \sum_{x\in V^\st(F)} f(g^{-1}z n_\sub(x) g) \\
- \int_{N_{P_1}(\A)} f(g^{-1} zn g)\, \d n \, \, \hat\tau_{P_1}(H_{P_1}(g)-T) \Big) \, e^{-\lambda(H_{P_1}(g))}  \, \d^1 g .
\end{multline*}
For $x\in V^\st(F)$ and $\delta\in G(F)$, we see that $\delta\in P_1(F)$ if and only if $\delta^{-1} n_\sub(x)\delta\in P_1(F)$.
Hence, it follows that
\[
K_{G,O_{\sub,z}}(g,g)=\sum_{\delta\in P_1(F)\bsl G(F)} \sum_{x\in V^\st(F)} f(g^{-1}\delta^{-1} z \, n_\sub(x) \, \delta g). 
\]
Therefore, we formally have
\[
J_{O_{\sub,z}}^T(f)=J_{O_{\sub,z}}^T(f,0) .
\]
Of course, we need only the absolute convergence of $J_{O_{\sub,z}}^T(f,\lambda)$ at $\lambda=0$ (i.e., $s=3/2$) for this theorem.
However, it is better to consider $J_{O_{\sub,z}}^T(f,\lambda)$ for any $s$ with a view to the general formulation (cf. \cite{Hoffmann2} and the proof of Theorem \ref{6t4}).
Furthermore, our proof gives absolute convergence and holomorphy of $J_{O_{\sub,z}}^T(f,\lambda)$.
By direct calculation we easily see
\[
e^{-\lambda(H_{P_1}((a,h')))}=e^{-(2s-3)\varpi_1(H_{P_1}((a,h')))}=|\det(h')|^{2s-3}
\]
and
\[
\widehat\tau_{P_1}(H_{P_1}((a,h'))-T)= \begin{cases} 1 & \text{if $-\log|\det(h')|-T_1>0$}, \\ 0 & \text{if $-\log|\det(h')|-T_1<0$}, \end{cases}
\]
where $(a,h')\in \GL(1,\A)^1\times \GL(2,\A)\cong M_1(\A)\cap G(\A)^1$ by \eqref{M1}.
Hence, we also get absolute convergence and holomorphy of $Z^{\GSp(2)}(\Phi_z,s,T_1)$ and $Z^{\Sp(2)}(\Phi_z,s,T_1)$ for $\Re(s)>5/4$.
They were already used in the argument of Section \ref{4s8}.

We will show that
\begin{multline}\label{6ezeta1}
 \int_{P_1(F)\bsl G(\A)^1} \Big|  \sum_{x\in V^\st(F)} f(g^{-1}\delta^{-1}zn_\sub(x)\delta g) \\
  -\int_{N_{P_1}(\A)}f(g^{-1}zn g) \, \d n \, \hat\tau_{P_1}(H_{P_1}(g)-T) \Big| \, |e^{-(2s-3)\varpi_1(H_{P_1}(g))}|  \, \d^1 g 
\end{multline}
converges for $\Re(s)>5/4$.
From this we have the absolute convergence of $J_{O_{\sub,z}}^T(f)$ and the equality
\[
J_{O_{\sub,z}}^T(f)=J_{O_{\sub,z}}^T(f,0)=Z^{\GSp(2)}(\Phi_z,3/2,T_1).
\]
Moreover, the formula follows from (\ref{4e3}) and Theorem \ref{4t22}.

We now give a proof for the convergence of (\ref{6ezeta1}) for $\Re(s)>5/4$ by an argument similar to the other cases, which is inspired by Shintani's argument in \cite[Section 2 in Chapter 1]{Shintani}.
We are studying the prehomogeneous vector space $(H,V)$, where the group $H=\GL(2)$ acts by $x\cdot h\mapsto \,^thxh$ $(h\in H$, $x\in V)$. For $x\in V$, let $H_x$ be the connected component of $1$ in the stabilizer of $x$ in $H$. Since $X(H_{\xd_1})_F$ is not trivial (cf. Lemma \ref{4l2}) and $\xd_1$ is a regular point, it is difficult to deal directly with zeta functions for~$(H,V)$. The idea is to deduce their properties from those of zeta functions for the prehomogeneous  vector space~$(B,V)$, where $B$ is the subgroup of lower triangular matrices. Indeed, $X(B_x)_F$ is trivial for any regular point $x$ in $V$.
In addition, we do not use the smoothed Eisenstein series.
This is completely different from \cite{Shintani} and \cite{Yukie}.
It seems difficult to prove the convergence of (\ref{6ezeta1}) using the smoothed Eisenstein series, because such argument requires a cancellation (cf. \cite[Section 3.8]{Datskovsky} and \cite[p.373]{Yukie}).

We will bound \eqref{6ezeta1} by $\eqref{6ezeta2}+\eqref{6ezeta4}+\eqref{6ezeta5}$.
This means that we reduce the problem of the convergence of \eqref{6ezeta1} to that of \eqref{6ezeta5}, which is identified with the absolute value of a part of a modified zeta integral for $(B,V)$. First, we have
\[
\eqref{6ezeta1} \leq \eqref{6ezeta2} + \eqref{6ezeta3} ,
\]
\begin{multline}\label{6ezeta2}
\int_{P_1(F)\bsl G(\A)^1} \Big|  \sum_{x\in V^\st(F)} f(g^{-1}zn_\sub(x) g)  \Big| \\
\{1-\hat\tau_{P_1}(H_{P_1}(g)-T)\} \, |e^{-(2s-3)\varpi_1(H_{P_1}(g))}|  \, \d^1 g ,
\end{multline}
\begin{multline}\label{6ezeta3}
\int_{P_1(F)\bsl G(\A)^1} \Big|  \sum_{x\in V^\st(F)} f(g^{-1}zn_\sub(x) g)   -\int_{N_{P_1}(\A)}f(g^{-1} zn g) \, \d n  \Big| \\
\hat\tau_{P_1}(H_{P_1}(g)-T) \, |e^{-(2s-3)\varpi_1(H_{P_1}(g))}|  \, \d^1 g.
\end{multline}
It is clear that (\ref{6ezeta2}) converges absolutely for any $s$.
Hence, we will show the absolute convergence of (\ref{6ezeta3}).
We use the same notation $F^P(x,T)$ as \cite[p.37]{Arthur1} and \cite[$\S$6]{Arthur6}.
Let $\tau_P^Q$ denote the characteristic function of $\{a\in\a_P^Q \, | \, \langle a,\alpha\rangle >0 \, (\forall \alpha\in\Delta_P^Q)\}$.
Then, we have
\[  \eqref{6ezeta3} \leq  \eqref{6ezeta4} + \eqref{6ezeta5}  , \]
\begin{multline}\label{6ezeta4}
 \int_{P_1(F)\bsl G(\A)^1} \Big|  \sum_{x\in V^\st(F)} f(g^{-1} zn_\sub(x) g)   -\int_{N_{P_1}(\A)}f(g^{-1} zn g) \, \d n  \Big|  \\
 |e^{-(2s-3)\varpi_1(H_{P_1}(g))}|  \, F^{P_1}(g,T)\, \hat\tau_{P_1}(H_{P_1}(g)-T) \, \d^1 g , 
\end{multline}
\begin{multline}\label{6ezeta5}
 \int_{P_0(F)\bsl G(\A)^1} \Big|  \sum_{x\in V^\st(F)} f(g^{-1} zn_\sub(x) g)   -\int_{N_{P_1}(\A)}f(g^{-1} zn g) \, \d n  \Big|  \\
|e^{-(2s-3)\varpi_1(H_{P_1}(g))}|  \, \tau_{P_0}^{P_1}(H_{P_0}(g)-T)\, \hat\tau_{P_1}(H_{P_1}(g)-T) \, \d^1 g . 
\end{multline}

Now, there are $T'\in \R$ and a compact subset $\omega$ of $P_0(\A)^1$ such that $F^{P_1}(g,T)$ is the characteristic function in $g$ of the projection of the set
\[
\mathfrak{S}^{P_1}(T',T)= \left\{ \, g=pak \; \Big| \; \begin{array}{c} p\in \omega , \; \, k\in\bK , \; \, e^{T'}<t_1^{-1}t_2\leq e^{2T_2-T_1}, \\  a=\diag(t_1^{-1},t_2^{-2},t_1,t_2)\in A_{M_0}^+ \end{array} \,  \right\} 
\]
onto $P_1(F)\bsl G(\A)^1$.
We use a decomposition of $g\in G(\A)^1$ as
\begin{equation}\label{6deco1}
g = \begin{pmatrix}I_2&* \\ O_2& I_2 \end{pmatrix} \begin{pmatrix} 1&-b&0&0 \\  0&1&0&0 \\ 0&0&1&0 \\ 0&0&b&1 \end{pmatrix}\begin{pmatrix}a^{-1}_1&0&0&0 \\ 0&a_2^{-1}&0&0 \\ 0&0&a_1c&0 \\ 0&0&0&a_2c \end{pmatrix}k \in G(\A)^1  
\end{equation}
where $k\in\bK$.
By $P_1(F)=P_0(F)\cup P_0(F)w_0N_{P_0}(F)$, for any $g\in G(F)\mathfrak{S}^{P_1}(T',T)$, there exists $\delta\in M_0(F)\cup P_0(F)w_0$ such that $\delta g\in N_0(F)\mathfrak{S}^{P_1}(T',T)$.
Hence, using
\[
T'< \alpha_2(H_{P_0}(g))\leq 2T_2-T_1 \quad (g\in N_0(F)\mathfrak{S}^{P_1}(T',T))
\]
and
\[
\alpha_2(H_{P_0}(w_0g))=-\log|a_1^{-1}a_2| -2\log||(1,a_1a_2^{-1}b)|| \quad (g\in G(\A)^1) ,
\]
we easily deduce
\begin{equation}\label{6eineq1}
-2T_2+T_1-2\log\|(1,a_1a_2^{-1}b)\| < \alpha_2(H_{P_0}(g)) \leq 2T_2-T_1 
\end{equation}
for any $g\in G(F)\mathfrak{S}^{P_1}(T',T)$.
Let $\tau_{T,\sub}^1(r_1,r_2)$ denote the characteristic function of
\[
\{ (r_1,r_2)\in\R^2 \, | \, r_1+r_2<-T_1 \; \text{ and } \; r_2-r_1\leq 2T_2-T_1 \} .
\]
Using \eqref{6eineq1} and the Poisson summation formula on $V(F)$, we have
\[
\eqref{6ezeta4} < \mathrm{(constant)}\times\{  \eqref{6eadd1}+\eqref{6eadd2}+\eqref{6eadd3}+\eqref{6eadd4}+\eqref{6eadd5}+\eqref{6eadd6}  \} ,
\]
where the summands are as follows:
\begin{equation}\label{6eadd1}
\int_{M_1(F)\bsl M_1(\A)^1} F^{P_1}(m,T) \, \d^1 m \, \int_0^{e^{-T_1}}t^{2s}\d^\times t \, \Big|\Phi_z(\begin{pmatrix} 0&0 \\ 0&0 \end{pmatrix})\Big| ,
\end{equation}
\begin{multline}\label{6eadd2}
\int_{F^\times\bsl\A^1}\d^1 c\int_{F^\times\bsl\A^\times} \d^\times a_1 \int_{F^\times\bsl\A^\times} \d^\times a_2 \, |a_1|^{2s+1}|a_2|^{2s-1} \\ \tau_{T,\sub}^1(\log|a_1|,\log|a_2|)  \sum_{x\in F^\times}\Big|\Phi_z(\begin{pmatrix} a_1^2cx&0 \\ 0&0 \end{pmatrix})\Big|  ,
\end{multline}
\begin{multline}\label{6eadd3}
\int_{F^\times\bsl\A^1}\d^1 c\int_{\A} \d b \int_{F^\times\bsl\A^\times , \, |a_2|<e^{-T_1}} \d^\times a_2 \, |a_2|^{2s} \, \{ 1+ \log\|(1,b)\| \} \\ \sum_{x\in F^\times}\Big|\Phi_z(c\begin{pmatrix} 2a_2b&a_2x \\ a_2x&0 \end{pmatrix})\Big|  ,
\end{multline}
\begin{multline}\label{6eadd4}
\int_{F^\times\bsl\A^1}\d^1 c \int_{\GL(2,\A), \, |\det(h')|<e^{-T_1}} \d h' \, |\det(h')|^{2s-3} \\ \sum_{x\in V^\st(F)} \big|\hat\Phi_z(c^{-1}h^{\prime -1}x\, ^th^{\prime -1})\big|  ,
\end{multline}
\begin{multline}\label{6eadd5}
\int_{F^\times\bsl\A^1}\d^1 c\int_{\A} \d b \int_{F^\times\bsl\A^\times , \, |a_1|<e^{-T_1}} \d^\times  a_1 \\ |a_1|^{2s-3} \, \{ 1+ \log\|(1,b)\| \} \sum_{x\in F^\times}\Big|\hat\Phi_z(c\begin{pmatrix} 0&a_1^{-1}x \\ a_1^{-1}x&-2a_1^{-1}b \end{pmatrix})\Big|  ,
\end{multline}
\begin{multline}\label{6eadd6}
\int_{F^\times\bsl\A^1}\d^1 c\int_{F^\times\bsl\A^\times} \d^\times a_1 \int_{F^\times\bsl\A^\times} \d^\times a_2 \\ |a_1|^{2s-2}|a_2|^{2s-4} \, \tau_{T,\sub}^1(\log|a_1|,\log|a_2|) \sum_{x\in F^\times}\Big|\hat\Phi_z(\begin{pmatrix} 0&0 \\ 0&a_2^{-2}c^{-1}x \end{pmatrix})\Big|  .
\end{multline}
The function $\tau_{T,\mathrm{sub}}^1(\log|a_1|,\log|a_2|)$ means that the domain of the integration for $(a_1,a_2)$ is
\[
\Big\{ (a_1,a_2)\in F^\times\bsl\A^\times \times F^\times\bsl\A^\times \, \Big| \begin{array}{c}  \text{($|a_1|>e^{-T_2}$ and $|a_2|<|a_1|^{-1}e^{-T_1}$) or} \\ \text{($|a_1|<e^{-T_2}$ and $|a_2|<|a_1|e^{2T_2-T_1}$)}  \end{array}  \Big\} .
\]
Hence, \eqref{6eadd2} converges for $\Re(s)>1/2$.
We can similarly see that \eqref{6eadd6} converges for $\Re(s)>1/2$.
The terms \eqref{6eadd3} and \eqref{6eadd5} converge for $\Re(s)>1$ by \cite[Proposition 2.12]{Yukie} (cf. Section \ref{4s9}).
The convergence of \eqref{6eadd1} and \eqref{6eadd4} is trivial.
Thus, \eqref{6ezeta4} converges for $\Re(s)>1$.

In order to simplify notation, we use the characteristic function $\tau_{T,\sub}^2(r_1,r_2)$ (resp. $\tau_{T,\sub}^3(r_1,r_2)$) of
\[
\{ (r_1,r_2)\in\R^2 \, | \, r_2>-T_1+T_2 \; \text{ and } \; r_1<-r_2-T_1 \}
\]
\[
\mathrm{(resp.}\quad  \{ (r_1,r_2)\in\R^2 \, | \, r_2<-T_1+T_2 \; \text{ and } \; r_1<r_2+T_1-2T_2 \}  \; \mathrm{)}.
\]
By the decomposition \eqref{6deco1} we have
\begin{multline*}
\tau_{P_0}^{P_1}(H_{P_0}(g)-T)\, \hat\tau_{P_1}(H_{P_1}(g)-T)=\\
\begin{cases} 1 & \text{if $\log|a_1a_2|<-T_1$ and $\log|a_1a_2^{-1}|<T_1-2T_2$}, \\ 0 & \text{otherwise}. \end{cases}
\end{multline*}
Hence,
\[  \eqref{6ezeta5}=\eqref{6eaddx1}+\eqref{6eaddx2},   \]
\begin{multline}\label{6eaddx1}
 \int_{P_0(F)\bsl G(\A)^1} \Big|  \sum_{x\in V^\st(F)} f(g^{-1} zn_\sub(x) g)   -\int_{N_{P_1}(\A)}f(g^{-1} zn g) \, \d n  \Big|  \\
|e^{-(2s-3)\varpi_1(H_{P_1}(g))}|  \, \tau_{T,\sub}^2(\log|a_1|,\log|a_2|) \, \d^1 g , 
\end{multline}
and
\begin{multline}\label{6eaddx2}
 \int_{P_0(F)\bsl G(\A)^1} \Big|  \sum_{x\in V^\st(F)} f(g^{-1} zn_\sub(x) g)   -\int_{N_{P_1}(\A)}f(g^{-1} zn g) \, \d n  \Big|  \\
|e^{-(2s-3)\varpi_1(H_{P_1}(g))}|  \, \tau_{T,\sub}^3(\log|a_1|,\log|a_2|) \, \d^1 g  
\end{multline}
where $a_1$ and $a_2$ are determined by \eqref{6deco1}.
The set
\begin{multline*}
\Big\{ x_{m_1,m_{12},m_2}=\begin{pmatrix}1& m_{12} \\ 0&1 \end{pmatrix}\begin{pmatrix}m_1&0 \\ 0& m_2 \end{pmatrix}\begin{pmatrix}1&0 \\ m_{12}&1 \end{pmatrix}\in V(F)  \\
\Big| \, m_2\in F^\times \, , \; m_1,m_{12}\in F  \Big\} 
\end{multline*}
is decomposed into
\begin{align}\label{6deco2}
& V^\st(F) \cup   \{ x_{m_1,m_{12},m_2} \, | \, m_2\in F^\times, \, \; m_{12}\in F \, , \;  m_1=0 \} \\
& \cup \{ x_{m_1,m_{12},m_2} \, | \, m_2\in F^\times, \, \; m_{12}\in F \, , \; m_1=-m_2(F^\times)^2 \}   \nonumber
\end{align}
(disjoint union), since $V^\st(F)$ can be expressed by
\begin{multline*}
V^\st(F)= \Big\{ x_{m_1,m_{12},m_2}\in V(F) \\
\Big| \, m_2\in F^\times \, , \; m_{12}\in F  \, , \; m_1\in F^\times-(-m_2)(F^\times)^2 \Big\}.
\end{multline*}
It is clear that \eqref{6eaddx1} is bounded by $\eqref{6ezeta13}+\eqref{6ezeta14}$, where we set
\begin{multline}\label{6ezeta13}
 \int_\A  \d b  \, \int_{F^\times\bsl\A^1}\d^1 c  \int_{F^\times\bsl\A^\times}\d^\times a_1 \int_{F^\times\bsl\A^\times}\d^\times a_2 \\
|a_1|^{2s} \, |a_2|^{2s-2} \, \tau_{T,\sub}^2(\log|a_1|,\log|a_2|) \\
 \sum_{m_2\in F^\times } \sum_{m_1\in F^\times-(-m_2)(F^\times)^2} \Big|  \Phi_z(c\begin{pmatrix} a_1^2m_1+ m_2^{-1}a_2^{-2}b^2 & b \\ b & a_2^2m_2 \end{pmatrix})  \Big| , 
\end{multline}
and
\begin{equation}\label{6ezeta14}
\frac{c_F}{2s-2}\int_{F^\times\bsl\A^1}\d^1 c   \int_{F^\times\bsl\A^\times, \, |a_2|>e^{-T_1+T_2}}\d^\times a_2 \, |a_2|^{-2} \, \Big| \hat\Phi_z(\begin{pmatrix} 0 & 0 \\ 0 &0 \end{pmatrix})  \Big| 
\end{equation}
if $\Re(s)>1$.
There exist functions $\phi_1$, $\phi_2\in C_c^\inf(\A)$ such that
\[ \sum_{m_1\in F^\times-(-m_2)(F^\times)^2} |\Phi_z(*) | < |a_1|^{-2}\times \phi_1(cb)\times\phi_2(ca_2^2m_2).   \]
Therefore, it follows that \eqref{6ezeta13} and \eqref{6ezeta14} are convergent for $\Re(s)>1$.
So, we get the absolute convergence of \eqref{6eaddx1} for $\Re(s)>1$.
By the decompositions \eqref{6deco1} and \eqref{6deco2} and repeated application of the Poisson summation formula we have
\[
\text{(\ref{6eaddx2})}< \text{(constant)}\times\{ \eqref{6ezeta6} + \eqref{6ezeta7} + \eqref{6ezeta8} + \eqref{6ezeta9}  + \eqref{6ezeta11} + \eqref{6ezeta12}  \} ,
\]
where the summands are as follows:
\begin{multline}\label{6ezeta6}
 \int_\A  \d b  \, \int_{F^\times\bsl\A^1}\d^1 c  \int_{F^\times\bsl\A^\times}\d^\times a_1 \int_{F^\times\bsl\A^\times}\d^\times a_2 \\
|a_1|^{2s} \, |a_2|^{2s-2} \, \tau_{T,\sub}^3(\log|a_1|,\log|a_2|)  \sum_{m_2\in F^\times }\Big|  \Phi_z(c\begin{pmatrix}m_2^{-1}a_2^{-2}b^2 & b \\ b & a_2^2m_2 \end{pmatrix})  \Big| , 
\end{multline}
\begin{multline}\label{6ezeta7}
\int_\A  \d b  \, \int_{F^\times\bsl\A^1}\d^1 c  \int_{F^\times\bsl\A^\times}\d^\times a_1 \int_{F^\times\bsl\A^\times}\d^\times a_2 \\
|a_1|^{2s}  |a_2|^{2s-2} \, \tau_{T,\sub}^3(\log|a_1|,\log|a_2|) \\
\sum_{m_2\in F^\times } \sum_{\alpha_1\in F^\times}\Big|  \Phi_z(c\begin{pmatrix}-m_2(a_1\alpha_1)^2 + m_2^{-1}a_2^{-2}b^2 & b \\ b & a_2^2m_2 \end{pmatrix})  \Big| ,
\end{multline}
\begin{multline}\label{6ezeta8}
\int_\A  \d b  \, \int_{F^\times\bsl\A^1}\d^1 c  \int_{F^\times\bsl\A^\times}\d^\times a_1 \int_{F^\times\bsl\A^\times}\d^\times a_2 \\
|a_1|^{2s-2} |a_2|^{2s-2} \, \tau_{T,\sub}^3(\log|a_1|,\log|a_2|) \\
\sum_{m_1\in F^\times } \sum_{m_2\in F^\times} \Big| \hat\Phi_z^1(\begin{pmatrix} c^{-1} a_1^{-2} m_1 & cb \\ cb & ca_2^2m_2 \end{pmatrix})  \Big| ,
\end{multline}
\begin{multline}\label{6ezeta9}
\int_{F\bsl \A}  \d b  \, \int_{F^\times\bsl\A^1}\d^1 c  \int_{F^\times\bsl\A^\times}\d^\times a_1 \int_{F^\times\bsl\A^\times}\d^\times a_2 \\
|a_1|^{2s-2} |a_2|^{2s-2} \tau_{T,\sub}^3(\log|a_1|,\log|a_2|) \\
\sum_{m_{12}\in F^\times} \sum_{m_2\in F^\times} \Big| \hat\Phi_z^{1,12}(\begin{pmatrix} 0 & c^{-1}a_1^{-1}a_2^{-1}m_{12} \\ c^{-1}a_1^{-1}a_2^{-1}m_{12} & ca_2^2m_2 \end{pmatrix})  \Big| , 
\end{multline}
\begin{equation}\label{6ezeta11}
\frac{c_F}{2s-2}\int_{F^\times\bsl\A^1}\d^1 c   \int_{F^\times\bsl\A^\times, \, |a_2|<e^{-T_1+T_2}}\d^\times a_2 \, |a_2|^{4s-4} \, \Big| \hat\Phi_z^{1,12}(\begin{pmatrix} 0 & 0 \\ 0 &0 \end{pmatrix})  \Big| , 
\end{equation}
\begin{multline}\label{6ezeta12}
\frac{c_F}{2s-2}\int_{F^\times\bsl\A^1}\d^1 c   \int_{F^\times\bsl\A^\times, \, |a_2|<e^{-T_1+T_2}}\d^\times a_2 \, |a_2|^{4s-6} \\
\sum_{m_2\in F^\times }  \Big| \hat\Phi_z(\begin{pmatrix} 0 & 0 \\ 0 & c^{-1}a_2^{-2}m_2 \end{pmatrix})  \Big| .
\end{multline}
In the above equations, we assume $\Re(s)>1$ and we set
\[
\hat\Phi_z^1(\begin{pmatrix}x_1&x_{12}\\ x_{12}&x_2 \end{pmatrix})=\int_\A \Phi_z(\begin{pmatrix}y_1& x_{12} \\ x_{12} & x_2 \end{pmatrix}) \, \psi_F(x_1y_1) \, \d y_1
\]
and
\[  \hat\Phi_z^{1,12}(\begin{pmatrix}x_1&x_{12}\\ x_{12}&x_2 \end{pmatrix})=\int_\A\int_\A \Phi_z(\begin{pmatrix}y_1& y_{12} \\ y_{12} & x_2 \end{pmatrix}) \, \psi_F(x_1y_1+2x_{12}y_{12}) \, \d y_1 \, \d y_{12} . \]
It is trivial that \eqref{6ezeta6}, \eqref{6ezeta11}, and \eqref{6ezeta12} converge for $\Re(s)>1$.
We also see that \eqref{6ezeta8} and \eqref{6ezeta9} converge for any $s$.
For \eqref{6ezeta7}, there exist functions $\phi_3$, $\phi_4\in C_c^\inf(\A)$ such that
\[ \sum_{\alpha_1\in F^\times}|\Phi_z(*)| < |a_1|^{-1}\times \phi_3(cb) \times \phi_4(ca_2^2m_2). \]
Therefore, \eqref{6ezeta7} converges for $\Re(s)>5/4$.
Hence, we have proved that \eqref{6eaddx2} and \eqref{6ezeta5} converge for $\Re(s)>5/4$.
Thus, we obtain the convergence of \eqref{6ezeta1} for $\Re(s)>5/4$.
\end{proof}

For $S\supset\Sigma_\inf$ we have
\[
J_{M_2}^T(z,f)=\int_{N_{P_2}(F_S)} f_{\bK_S}(z\, n) \, w_{M_2}(1,n,T) \, \d n 
\]
by the definition (cf. Proposition \ref{5p5}).
\begin{thm}\label{6t3}
Assume that $S$ contains $\Sigma_\inf\cup\Sigma_2$ and $z\in Z(\fO_v)$ $(\forall v\not\in S)$.
Let $f\in C_c^\inf(G(F_S)^1)$.
The integral $J_{O_{\sub,z}'}^T(f)$ converges absolutely.
We also have
\[ J_{O_{\sub,z}'}^T(f)= \frac{\vol_{M_2}}{2} J_{M_2}^T(z,f) + \frac{\vol_{M_1}}{2\, c_F} \frac{\frac{\d}{\d s}\zeta_F^S(s)|_{s=3} }{\zeta_F^S(3)} J_G(z \, n_\sub(\xd_1),f) . \]
\end{thm}
\begin{proof}
Throughout this proof, we set $\mathfrak u_z=z\,n_\sub(\begin{pmatrix}0&1 \\ 1&0 \end{pmatrix})$.
Then we have $O_{\sub,z}'=\{\mathfrak u_z\}_G$, $G_{\mathfrak u_z}=\{ \diag(a,b,b,a) \, | \, a,b\in\mathbb G_m  \}N_{P_1}$, and $G_{\mathfrak u_z,+}=G_{\mathfrak u_z} \cup w_0G_{\mathfrak u_z}$.

Let $\d^\times x$ (resp. $\d n_{13}$) denote the Haar measure on $\A^\times$ (resp. $\A$) which was defined in Section \ref{2s1}.
First, we will show
\begin{multline}\label{6e6}
J_{O_{\sub,z}'}^T(f)= \frac{\vol_{M_0}}{2\, c_F} \int_{\A^\times} \int_{\A} f_\bK(z\, \nu_2(0,n_{13},x) ) \\
|x|^2 \, \{ \log\|(x,n_{13}/2)\| + 2T_2 \} \, \d n_{13}\, \d^\times x .
\end{multline}
It is clear that the integral converges absolutely (cf. \cite[Section 2]{Yukie} and Section \ref{4s9}).
We consider the integral
\begin{multline}\label{6ell}
\int_{G(F)\bsl G(\A)^1}\Big\{   \sum_{ \delta\in G_{\mathfrak u_z,+} \bsl G(F) }  f(g^{-1}\delta^{-1} \mathfrak u_z \delta g) \\
 - \sum_{ \delta_1 \in G_{\mathfrak u_z} \bsl G(F) }  f(g^{-1}\delta_1^{-1} \mathfrak u_z \delta_1 g) \widehat\tau_{P_2}(H_{P_2}(\delta_1 g)-T) \Big\} \d^1 g.
\end{multline}
Since
\begin{multline*}
\sum_{ \delta\in G_{\mathfrak u_z,+} \bsl G(F) }  f(g^{-1}\delta^{-1}\, \mathfrak u_z \, \delta g) \\
 - \sum_{ \delta_1\in G_{\mathfrak u_z} \bsl G(F) }  f(g^{-1}\delta_1^{-1}\, \mathfrak u_z \, \delta_1 g) \widehat\tau_{P_2}(H_{P_2}(\delta_1 g)-T) =\\
\frac{1}{2}\sum_{ \delta\in G_{\mathfrak u_z} \bsl G(F) } f(g^{-1}\delta^{-1}\, \mathfrak u_z \, \delta g) \{ 1- \widehat\tau_{P_2}(H_{P_2}(g)-T) - \widehat\tau_{P_2}(H_{P_2}(w_0g)-T)\} ,
\end{multline*}
we easily prove the absolute convergence of (\ref{6ell}) by using Lemma \ref{5l3} and the absolute convergence of (\ref{6e6}).
Formally we have
\[ J_{O_{\sub,z}'}^T(f)= \eqref{6ell} + \eqref{6ellt}  \]
where
\begin{multline}\label{6ellt}
 \int_{G(F)\bsl G(\A)^1}  \sum_{ \delta_1\in P_2(F) \bsl G(F) } \Big\{   \sum_{\delta\in G_{\mathfrak u_z} \bsl P_2(F) }  f(g^{-1}\delta_1^{-1}\delta^{-1}\, \mathfrak u_z \, \delta \delta_1 g) \\
  - \int_{N_{P_2}(\A)} f(g^{-1}\delta_1^{-1}zn\delta_1g)\, \d n \Big\}\, \widehat\tau_{P_2}(H_{P_2}(\delta_1 g)-T) \, \d^1 g. 
\end{multline}
If we show that the integral (\ref{6ellt}) converges absolutely, then the absolute convergence of $J_{O_{\sub,z}'}^T(f)$ follows.
We easily see
\begin{multline*}
\int_{P_2(F)\bsl G(\A)^1} \Big|   \sum_{ \delta\in G_{\mathfrak u_z} \bsl P_2(F) }  f(g^{-1}\delta^{-1}\, \mathfrak u_z \, \delta g)   -\int_{N_{P_2}(\A)} f(g^{-1}zn g)\, \d n \Big| \\
\widehat\tau_{P_2}(H_{P_2}( g)-T) \, \d^1 g < \text{(\ref{6eab1})}+\text{(\ref{6eab2})}
\end{multline*}
where
\begin{multline}\label{6eab1}
\int_{P_2(F)\bsl G(\A)^1} \Big|   \sum_{ \delta\in G_{\mathfrak u_z} \bsl P_2(F) }  f(g^{-1}\delta^{-1}\, \mathfrak u_z \, \delta g)   -\int_{N_{P_2}(\A)} f(g^{-1}zn g)\, \d n \Big|  \\
F^{P_2}(g,T) \, \widehat\tau_{P_2}(H_{P_2}( g)-T) \, \d^1 g 
\end{multline}
and
\begin{multline}\label{6eab2}
\int_{P_0(F)\bsl G(\A)^1} \Big|   \sum_{ \delta\in G_{\mathfrak u_z} \bsl P_2(F) }  f(g^{-1}\delta^{-1}\, \mathfrak u_z \, \delta g)   -\int_{N_{P_2}(\A)} f(g^{-1}zn g)\, \d n \Big|  \\
\tau_{P_0}^{P_2}(H_{P_0}(g)-T) \, \widehat\tau_{P_2}(H_{P_2}( g)-T) \, \d^1 g .  
\end{multline}
Note that
\[
\{\delta^{-1}\, \mathfrak u_z \, \delta \, | \, \delta\in G_{\mathfrak u_z} \bsl P_2(F)  \}=(z\, N_{P_2}(F))\cap O_{\sub,z}'.
\]
Since Theorem \ref{6t1} is already proved, it suffices to show convergence of the integral with $O'_{\mathrm{sub},z}$ replaced by
\[
\bar O'_{\mathrm{sub},z}= O'_{\mathrm{sub},z} \cup O_{\mathrm{min},z}\cup O_{\mathrm{tri},z}.
\]
However, the intersection of the latter union with $P_2(F)$ is $N_{P_2}$-invariant, and we can apply Poisson summation as in \cite[p.945--947]{Arthur6} to prove the convergence of (\ref{6eab1}).
Therefore, we have only to prove the convergence of \eqref{6eab2}.

For this purpose, we use the notation
\[
\nu_2(n_{12},n_{13},n_{14})=\nu(n_{12},n_{13},n_{14},0) \in N_{P_2}  .
\]
where $\nu(n_{12},n_{13},n_{14},n_{24})$ was defined in \eqref{u}.
Let $\d n_*$ denote the Haar measure on $\A$ defined in Section \ref{2s1}.
We set
\[
f_{z,13}(x_{12},x_{14})=\int_\A f(z\, \nu_2(x_{12},n_{13},x_{14}))\, \d n_{13},
\]
\[
\hat f_{z,13}(x_{12},x_{14})= \int_\A \int_\A f_{z,13}(n_{12},n_{14}) \, \psi_F(x_{12}n_{12}) \, \psi_F(x_{14}n_{14}) \, \d n_{12} \, \d n_{14}.
\]
Put
\begin{multline}\label{6deco3}
g=\begin{pmatrix}I_2&* \\ O_2& I_2 \end{pmatrix} \begin{pmatrix} 1&-n_{13}/2&0&0 \\  0&1&0&0 \\ 0&0&1&0 \\ 0&0&n_{13}/2&1 \end{pmatrix}\begin{pmatrix}a_1^{-1}&0&0&0 \\ 0&a_2^{-1}&0&0 \\ 0&0&a_1 c&0 \\ 0&0&0&a_2c \end{pmatrix}k \\
\in G(\A)^1 
\end{multline}
where $k\in\bK$.
Then, we have
\begin{multline*}
\tau_{P_0}^{P_2}(H_{P_0}(g)-T)\, \hat\tau_{P_2}(H_{P_2}(g)-T)=\\
\begin{cases} 1 & \text{if $\log|a_1|<-T_2$ and $\log|a_2|<-2T_1+2T_2$}, \\ 0 & \text{otherwise}. \end{cases}
\end{multline*}
We denote by $\tau_{T,\sub}'(t_{12},t_{14})$ the characteristic function of 
\[
\{(t_{12},t_{14})\in\R^2 \, | \, t_{12}+t_{14}<-2T_2 \; \text{ and } \; -t_{12}+t_{14}<-4T_1+4T_2  \},
\]
to simplify the description of $\tau_{P_0}^{P_2}(H_{P_0}(g)-T)\, \hat\tau_{P_2}(H_{P_2}(g)-T)$ for $|a_1|=e^{(t_{12}+t_{14})/2}$ and $|a_2|=e^{(-t_{12}+t_{14})/2}$.
By using the decomposition \eqref{6deco3}, change of variable ($(a_{12},a_{14},a_2)=(a_1a_2^{-1},a_1a_2c,a_2)$ on $\A^1\times\A^1\times\A^1$, $(|a_{12}|,|a_{14}|)=(|a_1a_2^{-1}|,|a_1a_2|)$ on $\Rp\times\Rp$), and the Poisson summation formula, we have
\[ \text{(\ref{6eab2})}< \text{(constant)}\times \{  \eqref{6eabc1} + \eqref{6eabc2} + \eqref{6eabc3} + \eqref{6eabc4} + \eqref{6eabc5} \}, \]
where the summands are as follows:
\begin{equation}\label{6eabc1}
\int_{\A^\times / F^\times}\d^\times a_{12} \int_{\A^\times / F^\times}\d^\times a_{14} \, |a_{14}|^2\, \tau_{T,\sub}'(\log|a_{12}|,\log|a_{14}|)\, |f_{z,13}(0,0)| ,
\end{equation}
\begin{multline}\label{6eabc2}
\int_{\A^\times / F^\times}\d^\times a_{12} \int_{\A^\times / F^\times}\d^\times a_{14} \, |a_{12}|^{-1}\, |a_{14}|\\
\tau_{T,\sub}'(\log|a_{12}|,\log|a_{14}|) \sum_{\alpha_{12}\in F^\times} \sum_{\alpha_{14}\in F^\times} |\hat f_{z,13}(a_{12}^{-1}\alpha_{12},a_{14}^{-1}\alpha_{14})| , 
\end{multline}
\begin{equation}\label{6eabc3}
\int_{\A^\times / F^\times \, , \, |a_{12}|^{-1}<e^{-2T_1+3T_2}}\d^\times a_{12}  \, |a_{12}|^{-2}   \sum_{\alpha_{12}\in F^\times}  |\hat f_{z,13}(a_{12}^{-1}\alpha_{12},0)| , 
\end{equation}
\begin{equation}\label{6eabc4}
\int_{\A^\times / F^\times \, , \, |a_{12}|^{-1}>e^{-2T_1+3T_2}}\d^\times a_{12}  \,   \sum_{\alpha_{12}\in F^\times}  |\hat f_{z,13}(a_{12}^{-1}\alpha_{12},0)|,
\end{equation}
\begin{equation}\label{6eabc5}
\int_{\A^\times / F^\times \, , \, |a_{14}|^{-1}>e^{-2T_1+T_2}}\d^\times a_{14} \, (1+|a_{14}|^2)  \,   \sum_{\alpha_{14}\in F^\times}  |\hat f_{z,13}(0,a_{14}^{-1}\alpha_{14})| .
\end{equation}
The function $\tau_{T,\mathrm{sub}}'(\log|a_{12}|,\log|a_{14}|)$ means that the domain of the integration for $(a_{12},a_{14})$ is
\begin{multline*}
\Big\{ (a_{12},a_{14})\in F^\times\bsl\A^\times \times F^\times\bsl\A^\times \, \\
\Big| \begin{array}{c}  \text{($|a_{12}|>e^{-2T_1+3T_2}$ and $|a_{14}|<|a_{12}|^{-1}e^{-2T_2}$) or} \\ \text{($|a_{12}|<e^{-2T_1+3T_2}$ and $|a_{14}|<|a_{12}|e^{-4T_1+4T_2}$)}  \end{array}  \Big\} . 
\end{multline*}
Hence, the convergence of \eqref{6eabc1} and \eqref{6eabc2} follows.
The convergence of \eqref{6eabc3}, \eqref{6eabc4}, and \eqref{6eabc5} is trivial.
Therefore, the integral (\ref{6eab2}) converges absolutely.
Furthermore, we have proved that the integrals \eqref{6ellt} and $J_{O_{\sub,z}'}^T(f)$ are absolutely convergent.

Let $\d^1 p$ denote the Haar measure on $P_2(\A)^1$ induced from $\d^1 g$. We deduce from the mean-value formula~(\ref{2e4}) that
\begin{align*}
&  \int_{P_2(F)\bsl P_2(\A)^1} \sum_{ \delta\in G_{\mathfrak u_z} \bsl P_2(F) }  f_\bK(a^{-1}p^{-1}\delta^{-1}\, \mathfrak u_z \, \delta p a) \, \d^1p \\
& = \vol_{M_2}\, \int_{N_{P_2}(\A)}f_\bK(za^{-1}n a)\, \d n \\
& = \int_{P_2(F)\bsl P_2(\A)^1} \int_{N_{P_2}(\A)} f_\bK(a^{-1}p^{-1} \, zn\, pa) \, \d n \, \d^1 p \nonumber 
\end{align*}
where $a\in (A_{M_2}^G)^+$.
Therefore, (\ref{6ellt}) vanishes and we have proved (\ref{6e6}).

If we consider the Iwasawa decomposition of $\SL(2,\A)$ in the right hand side of (\ref{2e4}), then (\ref{6e6}) is equal to
\begin{multline*}
 J_{O_{\sub,z}'}^T(f)= \frac{\vol_{M_2}}{2} \int_{\A}\int_{\A}\int_{\A} f_\bK(z\,\nu_2(n_{12},n_{13},n_{14}) )  \\
 \{ \log\|(n_{12},n_{13}/2,n_{14})\| + 2T_2 \} \, \d n_{12}\, \d n_{13}\, \d n_{14}  . 
\end{multline*}
Let $\d n_{*,v}$ denote the Haar measure on $F_v$ defined in Section \ref{2s1}.
For each $v\not\in S$, we can show that
\[
\int_{N_{P_2}(\fO_v)}    \log\|(n_{12,v},n_{13,v},n_{14,v})\|_v \, \d n_{12,v} \, \d n_{13,v} \, \d n_{14,v} = \frac{q_v^{-3}\log(q_v^{-1})}{1-q_v^{-3}}. 
\]
by direct calculation.
Using (\ref{2e4}) and the same argument as the proof of Theorem \ref{4t13}, we have
\begin{align*}
& \frac{\vol_{M_2}}{2} \int_{N_{P_2}(F_S)}f_{\bK_S}(z\, n) \, \d  n= \frac{\vol_{M_2}}{2} \int_{N_{P_2}(\A)}f_{\bK}(z\, n) \, \d n \\
& = \frac{1}{2} \int_{N_1(\A) M_{2,\mathfrak u_z}(\A) \bsl N_2(\A) M_2(\A)}f_{\bK}( m^{-1}n^{-1} \mathfrak u_z n m) \, e^{ -\rho_{P_2}(H_{P_2}(m))} \, \d n  \, \d m \\
& = \frac{1}{2} \int_{N_1(\A) M_{1,\mathfrak u_z}(\A) \bsl N_1(\A) M_1(\A)}f_{\bK}( m^{-1}n^{-1} \mathfrak u_z n m) \, e^{ -\rho_{P_1}(H_{P_1}(m))} \, \d n  \, \d m \\
& = \frac{\vol_{M_1}}{2 \, c_F} \, J_G(z \, n_\sub(\mathfrak x_1),f).
\end{align*}
Hence, the proof is completed.
\end{proof}

For $S\supset\Sigma_\inf$, we have
\[
J_{M_0}^T(z,f)= \int_{N_{P_0}(F_S)} f_{\bK_S}(z\, n) \, w_{M_0}(1,n,T) \, \d n
\]
by definition (cf. Proposition \ref{5p4}).
Again, we use the notation
\[
\nu(n_{12},n_{13},n_{14},n_{24})
\]
defined in \eqref{u}.
Choosing Haar measures on unipotent conjugacy classes over $F_S$, the weighted orbital integrals $J_{M_1}^T(z\nu(1,0,0,0),f)$, $J_{M_2}^T(z\nu(0,0,0,1),f)$, and $J_{G}(z\nu(1,0,1,1),f)$ are given by
\[
J_{M_1}^T(z\nu(1,0,0,0),f)= c_S \int_{N_{P_0}(F_S)} f_{\bK_S}(z\, n) \, w_{M_1}(1,n,T) \, \d n ,
\]
\[
J_{M_2}^T(z\nu(0,0,0,1),f)= c_S \int_{N_{P_0}(F_S)} f_{\bK_S}(z\, n) \, w_{M_2}(1,n,T) \, \d n ,
\]
and
\[
J_G(z\nu(1,0,1,1),f)=c_S^2 \int_{N_{P_0}(F_S)} f_{\bK_S}(z\, n)   \, \d n 
\]
(cf. Proposition \ref{5p5}).
\begin{thm}\label{6t4}
Assume that $S$ contains $\Sigma_\inf\cup\Sigma_2$ and $z\in Z(\fO_v)$ $(\forall v\not\in S)$.
Let $f\in C_c^\inf(G(F_S)^1)$.
The integral $J_{O_{\reg,z}}^T(f)$ converges absolutely.
Furthermore, we obtain
\begin{align*}
 J_{O_{\reg,z}}^T(f)=& \, \frac{\vol_{M_0}}{8} \, J_{M_0}^T(z,f)+\frac{ \vol_{M_0}}{4 \, c_F} \, \mathfrak c_F(S) \, J_{M_1}^T(z\nu(1,0,0,0),f) \\
& +\frac{ \vol_{M_0}}{4 \, c_F} \, \mathfrak c_F(S) \, J_{M_2}^T(z\nu(0,0,0,1),f) \\
&+\frac{\vol_{M_0}}{4 \, c_F^2} \Big\{  2 (\mathfrak c_F(S))^2 + 3 \mathfrak c_F'(S) \, c_F^S \Big\} \, J_G(z\nu(1,0,1,1),f)   .
\end{align*}
\end{thm}
\begin{proof}
By \cite{Arthur6}, the integral defining $J_{\fo_z}^T(f)$ is absolutely convergent.
Subtracting the integrals defining $J_{O_{\tri,z}}^T(f)$, $J_{O_{\min,z}}^T(f)$, $J_{O_{\sub,z}}^T(f)$ and $J_{O'_{\sub,z}}^T(f)$, which are absolutely convergent by the preceding theorems, we obtain the absolute convergence of $J_{O_{\reg,z}}^T(f)$.
We will also get an alternative proof for the absolute convergence in the process of our proof of the above formula.
A change of variables shows that $J_{O_{\reg,z}}^T(f)$ is the value at $\lambda=0$ of
\begin{align*}
J_{O_{\reg,z}}^T(f,\lambda)&=\int_{P_0(F)\bsl G(\A)^1}\Bigg(
\sum_{\substack{\nu\in N_{P_0}(F)\\
\nu\notin N_{P_1}(F)\cup N_{P_2}(F)}} f(g^{-1}z\nu g)\\
&\quad-\sum_{\substack{\nu\in N_{P_0}(F)/N_{P_1}(F)\\
\nu\notin N_{P_1}(F)}}
\int_{N_{P_1}(\A)} f(g^{-1}z\nu ng)\,\d n\,
\hat\tau_{P_1}(H_{P_1}(g)-T)\\
&\quad-\sum_{\substack{\nu\in N_{P_0}(F)/N_{P_2}(F)\\
\nu\notin N_{P_2}(F)}}
\int_{N_{P_2}(\A)} f(g^{-1}z\nu ng)\,\d n\,
\hat\tau_{P_2}(H_{P_2}(g)-T)\\
&\quad+\int_{N_{P_0}(\A)} f(g^{-1}zng)\,\d n\,
\hat\tau_{P_0}(H_{P_0}(g)-T)\Bigg)e^{-\lambda(H_{P_0}(g))}\,\d^1g.
\end{align*}
We will show that $J_{O_{\reg,z}}^T(f,\lambda)$ converges absolutely for a range containing $\lambda=0$.
Now, the absolute integral of $J_{O_{\reg,z}}^T(f,\lambda)$ is bounded by $\eqref{6era1} + \eqref{6era2} + \eqref{6era3}$ where
\begin{multline}\label{6era1}
\int_{P_0(F)\bsl G(\A)^1} \Big|\,e^{-\lambda(H_{P_0}(g))}\\
\Big\{ \sum_{n_{12},n_{24}\in F^\times} \sum_{n_{13},n_{14}\in F}f(g^{-1}\nu(n_{12},n_{13},n_{14},n_{24})g) \\
- \sum_{n_{12}\in F^\times} \int_{N_{P_1}(\A)} f(g^{-1}\nu(n_{12},0,0,0)ng)\, \d n \, \tau_{P_0}^{P_2}(H_{P_0}(g)-T)   \\
- \sum_{n_{24}\in F^\times} \int_{N_{P_2}(\A)} f(g^{-1}\nu(0,0,0,n_{24})ng)\, \d n \, \tau_{P_0}^{P_1}(H_{P_0}(g)-T)  \\
+ \int_{N_{P_0}(\A)} f(g^{-1}ng)\, \d n \, \tau_{P_0}^{G}(H_{P_0}(g)-T)  \Big\}  \Big| \, \d^1 g , 
\end{multline}
\begin{multline}\label{6era2}
\int_{P_0(F)\bsl G(\A)^1} \Big| e^{-\lambda(H_{P_0}(g))} \Big\{ \sum_{n_{12}\in F^\times} \int_{N_{P_1}(\A)} f(g^{-1}\nu(n_{12},0,0,0)ng)\, \d n \\
- \int_{N_{P_0}(\A)} f(g^{-1}ng)\, \d n\, \tau_{P_0}^{P_1}(H_{P_0}(g)-T) \Big\} \\
\{\tau_{P_0}^{P_2}(H_{P_0}(g)-T)-\widehat\tau_{P_1}(H_{P_0}(g)-T)\} \, \Big| \, \d^1 g , 
\end{multline}
and
\begin{multline}\label{6era3}
\int_{P_0(F)\bsl G(\A)^1} \Big| e^{-\lambda(H_{P_0}(g))} \, \Big\{ \sum_{n_{24}\in F^\times} \int_{N_{P_2}(\A)} f(g^{-1}\nu(0,0,0,n_{24})ng)\, \d n   \\
- \int_{N_{P_0}(\A)} f(g^{-1}ng)\, \d n \, \tau_{P_0}^{P_2}(H_{P_0}(g)-T)   \Big\} \\
\{ \tau_{P_0}^{P_1}(H_{P_0}(g)-T) - \widehat\tau_{P_2}(H_{P_0}(g)-T) \} \,\Big|\, \d^1 g .
\end{multline}
We set
\begin{multline*}
\tilde f_{z,g}(n_{12},n_{13},n_{14},n_{24})=\\
\int_{\A^{\oplus 2}}  f(g^{-1}z\, \nu(n_{12},u_{13},u_{14},n_{24})g) \, \psi_F(n_{13}u_{13}+n_{14}u_{14}) \, \d u_{13} \, \d u_{14}
\end{multline*}
and
\[  g=\nu(u_{12},u_{13},u_{14},u_{24})\, \diag(a_1^{-1},a_2^{-1},ca_1,ca_2) \, k \in P_0(F)\bsl G(\A)^1 \]
where $k\in\bK$.
Then, using the Poisson summation formula, we have
\begin{multline}\label{6t4reg1}
\sum_{\nu\in N_{P_0}(F) , \, \nu\notin N_{P_1}(F)\cup N_{P_2}(F)} f(g^{-1}z\nu g) = \\
\sum_{n_{12},n_{24}\in F^\times} \sum_{n_{13},n_{14}\in F} \tilde f_{z,k}(a_1a_2^{-1}n_{12},a_1^{-2}c^{-1}n_{13},a_1^{-1}a_2^{-1}c^{-1}n_{14},a_2^2cn_{24}) \\
\times  |a_1|^{-3}|a_2|^{-1} \, \psi_F(-a_1^{-2}c^{-1}n_{13}\diamondsuit)\, \psi_F(-a_1^{-1}a_2^{-1}c^{-1}n_{14}\heartsuit) 
\end{multline}
where $\diamondsuit$ and $\heartsuit$ are polynomials of $a_1$, $a_2$, $c$, $u_*$, and $n_*$.
Let
\[
\lambda=\lambda_1\varpi_1+\lambda_2\varpi_2\in\a_{0,\C}^*.
\]
Then, we have
\[
e^{-\lambda( H_{P_0}(g))}=|a_1a_2|^{\lambda_1}|a_1|^{\lambda_2}=|a_1a_2^{-1}|^{\lambda_1+\lambda_2}|a_2^2|^{\lambda_1+\lambda_2/2}.
\]
We will show that \eqref{6era1} converges if $\Re(\lambda_1+\lambda_2+1)>0$ and $\Re(\lambda_1+\lambda_2/2+1)>0$.
We consider the domain
\[
D_{i,j}=\{ (a_1,a_2)\in\A^\times\times\A^\times \, | \, (-1)^i\log|a_1a_2^{-1}|>0 \, , \; \, (-1)^j\log|a_2^2|>0     \} 
\]
where $(i,j)=(0,0)$, $(1,0)$, $(0,1)$, or $(1,1)$.
By \eqref{6t4reg1} and the argument of Section \ref{3s}, it is sufficient to prove that the integral
\begin{multline}\label{6t4reg2}
 \int_{D_{i,j}}\d^\times a_1 \, \d^\times a_2 \int_{\A^1}\d^\times c \\
|a_1a_2^{-1}|^{\Re(\lambda_1+\lambda_2+1)}|a_2^2|^{\Re(\lambda_1+\lambda_2/2+1)}\sum_{n_{12},n_{24}\in F^\times} \\
\sum_{(n_{13},n_{14})\in F\oplus F-\{(0,0)\}} |\phi(a_1a_2^{-1}n_{12},a_1^{-2}c^{-1}n_{13},a_1^{-1}a_2^{-1}c^{-1}n_{14},a_2^2cn_{24})|   
\end{multline}
converges for any $\phi\in C^\inf_c(\A^{\oplus 4})$ and any $(i,j)$.
If $(i,j)=(1,0)$, then we have
\[
|a_1^{-2}||a_2^2|=|a_1a_2^{-1}|^{-2}>1, \quad |a_1^{-1}a_2^{-1}||a_2^2|=|a_1a_2^{-1}|^{-1}>1.
\]
If $(i,j)=(0,1)$, then
\[
|a_1^{-2}||a_1a_2^{-1}|^2=|a_2^2|^{-1}>1,\quad |a_1^{-1}a_2^{-1}||a_1a_2^{-1}|=|a_2^2|^{-1}>1.
\]
If $(i,j)=(1,1)$, then
\[
|a_1a_2^{-1}||a_1^{-1}a_2^{-1}|=|a_2^{-2}|>1 ,\quad |a_1^{-2}||a_2^2|=|a_1a_2^{-1}|^{-2}>1.
\]
Hence, by using the argument on the convergence of the Tate integrals, we can see that \eqref{6t4reg2} is convergent for any $(i,j)$ .
Hence, the convergence of \eqref{6era1} follows.
By change of variable $a_1a_2^{-1}=a_1'$, we see that $|a_1|^{\lambda_1+\lambda_2}|a_2|^{\lambda_1}\mapsto |a_1'|^{\lambda_1+\lambda_2}|a_2|^{2\lambda_1+\lambda_2}$,
\begin{multline*}
\int_{\A^\times}|a_2|^{2\lambda_1+\lambda_2} \{ \tau_{P_0}^{P_2}(H_{P_0}(g)-T)-\widehat\tau_{P_1}(H_{P_0}(g)-T)\} \d^\times a_2 \\
=c_F \begin{cases} (e^{(-T_1+T_2)(2\lambda_1+\lambda_2)}-|a_1'|^{-(\lambda_1+\lambda_2/2)}e^{-T_1(\lambda_1+\lambda_2/2)} ) & \\ \quad \times(2\lambda_1+\lambda_2)^{-1}  & \text{if $2\lambda_1+\lambda_2\neq 0$} , \\  -\frac{T_1}{2}+T_2+\frac{1}{2}\log|a_1'| & \text{if $2\lambda_1+\lambda_2=0$} , \end{cases}
\end{multline*}
and the integral \eqref{6era2} converges if $\Re(\lambda_2/2+1)>0$ and $\Re(\lambda_1+\lambda_2+1)>0$.
We also find that
\begin{multline*}
\int_{\A^\times}|a_1|^{\lambda_1+\lambda_2} \{ \tau_{P_0}^{P_1}(H_{P_0}(g)-T)-\widehat\tau_{P_2}(H_{P_0}(g)-T)\} \d^\times a_1 \\
=c_F \begin{cases} (|a_2|^{\lambda_1+\lambda_2}e^{(T_1-2T_2)(\lambda_1+\lambda_2)}-e^{-T_2(\lambda_1+\lambda_2)} ) & \\ \quad \times (\lambda_1+\lambda_2)^{-1} & \text{if $\lambda_1+\lambda_2\neq 0$} , \\  T_1-T_2+\log|a_2| & \text{if $\lambda_1+\lambda_2=0$} , \end{cases}
\end{multline*}
and the integral \eqref{6era3} converges if $\Re(\lambda_1/2+1)>0$ and $\Re(\lambda_1+\lambda_2/2+1)>0$.
Thus, we have proved that $J_{O_{\reg,z}}^T(f,\lambda)$ converges absolutely for the range
\[
\left\{ \lambda=\lambda_1\varpi_1+\lambda_2\varpi_2\in\a_{0,\C}^* \, \Big| \,   \begin{array}{c}  \Re(\lambda_1+\lambda_2+1)>0, \\ \Re(\lambda_1+\lambda_2/2+1)>0, \\ \Re(\lambda_1/2+1)>0  , \;\, \Re(\lambda_2/2+1)>0 \end{array}  \right\}. 
\]
Hence, the convergence assertion follows.
Furthermore, it follows from the above-mentioned argument that $J_{O_{\reg,z}}^T(f,\lambda)$ uniformly converges on a neighborhood at $\lambda=0$.
Therefore, we have
\[
J_{O_{\reg,z}}^T(f)=\lim_{\lambda\to 0} J_{O_{\reg,z}}^T(f,\lambda) .
\]

Assume that $\lambda$ belongs to the above range and $\lambda\neq 0$.
Let $U=N_{P_1}\cap N_{P_2}$, the commutator subgroup of~$N_{P_0}$.
For each~$\nu$ in the first sum, we have a bijection $N_{P_0,\nu}\bsl N_{P_0}\to\nu U$ given by $\eta\mapsto\eta^{-1}\nu\eta$.
When we apply the Iwasawa decomposition to the integral over~$g$, the integral over $N_{P_0}(F)\bsl N_{P_0}(\A)$ can be combined with the sum over~$\eta$.
Then we can apply the adelic version of that bijection to the integral over $N_{P_0,\nu}(\A)\bsl N_{P_0}(\A)$, while the integral over $N_{P_0,\nu}(F)\bsl N_{P_0,\nu}(\A)$ disappears.
We consider $V=N_{P_0}/U$ as a two-dimensional vector space with subspaces $V_1=N_{P_2}U/U$ and $V_2=N_{P_1}U/U$.
If we denote
\[
\tilde f_{\bK,z}(v)=\int_{U(\A)} f_{\bK}(zvu)\,\d u
\]
for $v\in V(\A)$, we obtain
\begin{multline*}
J_{O_{\reg,z}}^T(f,\lambda)=\int_{M_0(F)\bsl M_0(\A)\cap G(\A)^1}
\Bigg(\sum_{\substack{\nu\in V(F)\\
\nu\notin V_1(F)\cup V_2(F)}} \tilde f_{\bK,z}(m^{-1}\nu m)\\
-\sum_{\substack{\nu\in V_1(F)\\
\nu\ne0}}
\int_{V_2(\A)} \tilde f_{\bK,z}(m^{-1}\nu nm)\,\d n\,
\hat\tau_{P_1}(H_{M_0}(m)-T)\\
-\sum_{\substack{\nu\in V_2(F)\\
\nu\ne0}}
\int_{V_1(\A)} \tilde f_{\bK,z}(m^{-1}\nu nm)\,\d n\,
\hat\tau_{P_2}(H_{M_0}(m)-T)\\
+\int_{V(\A)} \tilde f_{\bK,z}(m^{-1}nm)\,\d n\,
\hat\tau_{P_0}(H_{M_0}(m)-T)\Bigg) \\
e^{-\lambda(H_{M_0}(m))}\,
\delta_V(m)^{-1}\d m,
\end{multline*}
where $\delta_V$ is the modular character of the action of $M_0(\A)$ on~$V(\A)$.
To simplify notation, we write
\[
(\lambda_1+\lambda_2+1,\lambda_1+\lambda_2/2+1)=(s_1,s_2), \quad  (-T_1+2T_2,2T_1-2T_2)=(T_1',T_2').
\]
If we replace $\hat\tau_{P_1}$, $\hat\tau_{P_2}$ and $\hat\tau_{P_0}$ in this expression by $\tau_{P_0}^{P_2}$, $\tau_{P_0}^{P_1}$ and $\tau_{P_0}$, respectively, (as \eqref{6era1}), we obtain the zeta integral 
\begin{equation}\label{6e4a1}
\frac{\vol_{M_0}}{2\, c_F^2}\,\zeta_{F\oplus F}^{(1,1)}(\tilde f_{\bK,z},(s_1,s_2),(T_1',T_2'))
\end{equation}
in the notation of Section~\ref{3s}, where $H=\{1\}$ and therefore $\iota$ is omitted.
We have identified $V_1(F)$ and $V_2(F)$ with $F$ via the coordinates $n_{12}$ resp.~$n_{24}$.
If we denote
\[
\tilde f_{\bK,i,z}(v_i)=\int_{N_{P_i}(\A)} f_{\bK}(zv_in)\,\d n
\]
for $v_i\in V_i(\A)$, we get
\[
J_{O_{\reg,z}}^T(f,\lambda)=\eqref{6e4a1}+\eqref{6e4a2}+\eqref{6e4a3}
\]
where
\begin{multline}\label{6e4a2}
\frac{\vol_{M_0}}{2\, c_F}\\
\times\frac{e^{-(s_2-1)T_2'}\zeta_F^1(\tilde f_{\bK,1,z},s_1,T_1')-e^{-(s_2-1)T_1} \zeta_F^1(\tilde f_{\bK,1,z},s_1-(s_2-1),T_1')}{s_2-1}
\end{multline}
and
\begin{multline}\label{6e4a3}
\frac{\vol_{M_0}}{2\, c_F}\\
\times \frac{e^{-(s_1-1)T_1'}\zeta_F^1(\tilde f_{\bK,2,z},s_2+(s_1-1)/2,T_2')-e^{-(s_1-1)T_2}\zeta_F^1(\tilde f_{\bK,2,z},s_2,T_2')}{s_1-1}
\end{multline}
where \eqref{6e4a2} (resp. \eqref{6e4a3}) corresponds to \eqref{6era2} (resp. \eqref{6era3}).
We set
\begin{multline*}
\mathfrak F(l_1,l_2) = \int_{F_S^{\oplus 4}} f_{\bK_S}(z\, \nu_0(n_{12},n_{13},n_{14},n_{24})) \\ (\log|n_{12}|_S)^{l_1}\, (\log|n_{24}|_S)^{l_2}   \, \d n_{12}\, \d n_{13}\, \d n_{14}\, \d n_{24} 
\end{multline*}
where $l_1$, $l_2\in\Z_{\geq 0}$.
It follows from (\ref{3e2}) and Theorem \ref{3t4} that
\begin{align*}
& \lim_{(s_1,s_2)\to(1,1)}\eqref{6e4a1} \\
&= \frac{\vol_{M_0}}{2\, c_F^2}\,\zeta_{F\oplus F}^{(1,1)}(\tilde f_{\bK,z},(1,1),(T_1',T_2'))
\\
&= \, \frac{\vol_{M_0}}{2}\, \mathfrak F(1,1)+\frac{\vol_{M_0}}{2\, c_F} \, \mathfrak c_F(S) \, c_S \, \{\mathfrak F(1,0)+\mathfrak F(0,1)\} \\
& \quad +\frac{\vol_{M_0}}{2\, c_F^2} \, (\mathfrak c_F(S))^2 \, (c_S)^2 \, \mathfrak F(0,0)  \\
& \quad +(-T_1+2T_2)\,\frac{\vol_{M_0}}{2\, c_F} \, \big\{  c_F \, \mathfrak F(0,1) +  \fc_F(S) \, c_S \, \mathfrak F(0,0) \big\} \\
& \quad +(2T_1-2T_2)\,\frac{\vol_{M_0}}{2\, c_F} \, \big\{  c_F \, \mathfrak F(1,0) +  \fc_F(S) \, c_S \, \mathfrak F(0,0) \big\}\\
& \quad +(-T_1+2T_2) (T_1-T_2)\,  \vol_{M_0} \, \mathfrak F(0,0) .
\end{align*}
Furthermore, using the argument of Section \ref{3s}, we obtain
\begin{align*}
& \lim_{(s_1,s_2)\to(1,1)}\eqref{6e4a2} \\
&=\frac{\vol_{M_0}}{2\, c_F}\,\frac{\d}{\d s}\zeta_F^1(\tilde f_{\bK,1,z},s,T_1')\Big|_{s=1}+\frac{\vol_{M_0}}{2\, c_F}\,(-T_1+2T_2)\zeta_F^1(\tilde f_{\bK,1,z},1,T_1') \\
&=\frac{\vol_{M_0}}{4\, c_F}\lim_{s\to 1}\frac{\d^2}{\d s^2}(s-1)^2\zeta_F^1(\tilde f_{\bK,1,z},s,T_1') - (T_1')^2 \frac{\vol_{M_0}}{4} \mathfrak F(0,0)\\
&\quad +\frac{\vol_{M_0}}{2\, c_F}\,(-T_1+2T_2)\zeta_F^1(\tilde f_{\bK,1,z},1,T_1') \\
&=\frac{\vol_{M_0}}{4}\, \mathfrak F(2,0)+\frac{\vol_{M_0}}{2\, c_F} \, \mathfrak c_F(S) \, c_S \, \mathfrak F(1,0)+\frac{\vol_{M_0}}{2\, c_F} \, \mathfrak c_F'(S) \, c_S \, \mathfrak F(0,0) \\
& \quad + (-T_1+2T_2)\frac{\vol_{M_0}}{2\, c_F}\big\{ c_F\, \mathfrak F(1,0) +  \fc_F(S) \, c_S \, \mathfrak F(0,0) \big\} \\
& \quad + (-T_1+2T_2)^2\frac{\vol_{M_0}}{4} \mathfrak F(0,0).
\end{align*}
It similarly follows that
\begin{align*}
&\lim_{(s_1,s_2)\to(1,1)}\eqref{6e4a3} \\
&=\frac{\vol_{M_0}}{4\, c_F}\,\frac{\d}{\d s}\zeta_F^1(\tilde f_{\bK,2,z},s,T_2')\Big|_{s=1}+\frac{\vol_{M_0}}{2\, c_F}\,(T_1-T_2)\zeta_F^1(\tilde f_{\bK,2,z},1,T_2') \\
&= \frac{\vol_{M_0}}{8}\, \mathfrak F(0,2)+\frac{\vol_{M_0}}{4\, c_F} \, \mathfrak c_F(S) \, c_S \, \mathfrak F(0,1)+\frac{\vol_{M_0}}{4\, c_F} \, \mathfrak c_F'(S) \, c_S \, \mathfrak F(0,0)\\
&\quad + (T_1-T_2)\frac{\vol_{M_0}}{2\, c_F} \big\{ c_F \,\mathfrak F(0,1) +  \fc_F(S) \, c_S \, \mathfrak F(0,0) \big\} \\
&\quad  + (T_1-T_2)^2\frac{\vol_{M_0}}{2} \mathfrak F(0,0).
\end{align*}
Therefore, the proof is completed.
\end{proof}

Theorems \ref{6t1}, \ref{6t2}, \ref{6t3}, and \ref{6t4} are summarized as follows.
\begin{thm}\label{6t5}
Assume that $S$ contains $\Sigma_\inf\cup\Sigma_2$ and $z\in Z(\fO_v)$ $(\forall v\not\in S)$.
Let $f\in C_c^\inf(G(F_S)^1)$.
The integrals $J_{O_{\min,z}}^T(f)$, $J_{O_{\sub,z}}^T(f)$, $J_{O_{\sub,z}'}^T(f)$, and $J_{O_{\reg,z}}^T(f)$ converge absolutely for any $T\in\a_0^+$.
Now, we see that
\[
(\cU_G(F))_{G,S}=\{ 1 , \; n_\min(1) , \; n_\sub(x), \; n_\reg(1) \, | \, x\in V^\ss(F)/{\sim_S} \,  \}.
\]
For the coefficients in the fine expansion of $J_{\fo_z}(f)$ (cf. Section \ref{basobj}), we have
\[
a^G(S,z\, n_\min(1))= \frac{ \vol_{M_2}  }{2\, c_F} \, \zeta_F^S(2) ,
\]
\begin{align*}
a^G(S,z\, n_\sub(x))=&\frac{\vol_{M_1}}{2 \, c_F}  \, \fC_F(S,\alpha)  \\
&+ \frac{\vol_{M_1}}{2\, c_F}  \begin{cases} \zeta_F^S(3)^{-1} \frac{\d}{\d s}\zeta_F^S(s)|_{s=3}   & \text{if $x\sim_S\xd_1$} , \\ 0 & \text{otherwise} , \end{cases}
\end{align*}
and
\[ a^G(S,z\, n_\reg(1))= \frac{\vol_{M_0}}{2\, c_F^2}\, (\mathfrak c_F(S))^2  +\frac{3\, \vol_{M_0}}{4\, c_F^2 }\, \mathfrak c_F'(S)  \, c_F^S. \]
\end{thm}

The contribution of subregular unipotent elements can be expressed by special values of zeta integrals at $s=0$.
\begin{thm}\label{6t6}
Fix a test function $f\in C_c^\inf(G(\A)^1)$.
We set $\Phi_z(x)=f_\bK(z \, n_\sub(x))$ and we recall the identification $M_1\cong G'=\GL(1)\times \GL(2)$ of \eqref{M1}.
If we assume $\hat\Phi_z(0)=0$, then we have
\[ J_{O_{\sub,z}}^T(f)+J_{O_{\sub,z}'}^T(f)= Z^{\GSp(2)}_\mathrm{ad}(\hat\Phi_z,0) + \frac{\vol_{M_0}}{2\, c_F} \mathfrak{T}_3(\widetilde{R_0\Phi_{z,\bK}},0,T_2) \]
where the notations were defined in Section \ref{4s9}.
\end{thm}
\begin{proof}
The formula follows from Proposition \ref{4p26}, Theorems \ref{6t2} and \ref{6t3}, and (\ref{6e6}).
\end{proof}
This formula is essentially the same as Shintani's formula \cite[Proposition 8]{Shintani}.
It is suitable for explicit calculations (cf. \cite[Section 3]{Shintani} and \cite[Section 5]{Wakatsuki}).
Actually, by Saito's calculation \cite[Theorem 2.2 and Corollary 2.3]{Saito1}, we can see that an explicit form of $Z_v(\hat\Phi_{0,v},s,\chi_v;d_v)$ is much simpler than that of $Z_v(\Phi_{0,v},s,\chi_v;d_v)$ for $v\in\Sigma_2$.
This means that the functional equation makes zeta functions simpler.
It would be interesting to find the reason for this phenomenon.

\subsection{Contribution of mixed elements}

We use the same notations as in Section \ref{5s3}.
Set
\[
u_1(\alpha,\beta)=\begin{pmatrix}I_2& \diag(\alpha,\beta) \\ O_2&I_2 \end{pmatrix}\quad \text{where} \quad \diag(\alpha,\beta)=\begin{pmatrix}\alpha&0 \\ 0 &\beta \end{pmatrix}   .
\]
Since
\[
\fo_{1,z}=\bigcup_{\alpha,\beta\in F}\{ \sigma_{1,z}\, u_1(\alpha,\beta) \}_G,
\]
the $\cO$-equivalence class $\fo_{1,z}$ is decomposed into four classes
\[
\fo_{1,z} = O^{1,z}_\tri \cup O^{1,z}_{\min,1} \cup O^{1,z}_{\min,2} \cup O^{1,z}_\reg
\]
where
\[
O^{1,z}_\tri=\{\sigma_{1,z}\}_G, \quad O^{1,z}_{\min,1}= \{\sigma_{1,z} u_1(1,0) \}_G, \quad O^{1,z}_{\min,2}= \{\sigma_{1,z} u_1(0,1) \}_G,
\]
\[
O^{1,z}_\reg= \bigcup_{\alpha\in F^\times/(F^\times)^2 } \{\sigma_{1,z} u_1(1,\alpha) \}_G.
\]
According to this decomposition we formally have
\[
J^T_{\fo_{1,z}}(f)= J_{O^{1,z}_\tri}^T(f) + J_{O^{1,z}_{\min,1}}^T(f)+ J_{O^{1,z}_{\min,2}}^T(f)+ J_{O^{1,z}_\reg}^T(f)
\]
where
\[
J_{O^{1,z}_\tri}^T(f)=\vol_{G_{\sigma_{1,z}}} \, J_G(\sigma_{1,z},f) ,
\]
\begin{multline*}
J_{O^{1,z}_{\min,1}}^T(f)=\int_{P_2(F)\bsl G(\A)^1}\\
\Big\{ \sum_{\delta \in (P_2(F)\cap G_{\sigma_{1,z}}(F))  \bsl P_2(F)}\sum_{\alpha\in F^\times} f(g^{-1}\delta^{-1}\sigma_{1,z}u_1(\alpha,0)\, \delta g) \\
 - \int_{N_{P_2}(\A)}f(g^{-1}\sigma_{1,z}ng)\, \d n \, \widehat{\tau}_{P_2}(H_{P_2}(g)-T)   \Big\} \d^1g , 
\end{multline*}
\begin{multline*}
J_{O^{1,z}_{\min,2}}^T(f)=\int_{s_0P_2(F)\bsl G(\A)^1}\\
\Big\{ \sum_{\delta \in (s_0P_2(F)\cap G_{\sigma_{1,z}}(F)) \bsl s_0P_2(F)} \sum_{\beta\in F^\times} f(g^{-1}\delta^{-1}\sigma_{1,z}u_1(0,\beta)\, \delta g) \\
- \int_{N_{s_0P_2}(\A)}f(g^{-1}\sigma_{1,z}ng)\, \d n \, \widehat{\tau}_{s_0P_2}(H_{s_0P_2}(g)-s_0T)  \Big\} \d^1g , 
\end{multline*}
and
\begin{align*}
& J_{O^{1,z}_\reg}^T(f) = \\
& \int_{G(F)\bsl G(\A)^1} \Big\{ \sum_{\delta\in (P_0(F)\cap G_{\sigma_{1,z}}(F))\bsl G(F) } \sum_{ \alpha, \, \beta \in F^\times} f(g^{-1}\delta^{-1} \sigma_{1,z}u_1(\alpha,\beta) \delta g)  \\
& - \sum_{\delta\in P_0(F)\bsl G(F)} \sum_{\gamma=\pm \sigma_{1,z}} \sum_{\beta\in F^\times} \\
&\qquad \int_{N_{P_2}(\A)} f(g^{-1}\delta^{-1}\gamma u_1(0,\beta)n \delta g)\, \d n \, \widehat{\tau}_{P_2}(H_{P_2}(\delta g)-T) \\
& - \sum_{\delta\in M_0(F)N_{P_1}(F)\bsl G(F)} \int_{N_{P_1}(\A)}f(g^{-1}\delta^{-1}\sigma_{1,z}n\delta g)\, \d n \, \widehat{\tau}_{P_1}(H_{P_1}(\delta g)-T) \\
& + \sum_{\delta\in P_0(F)\bsl G(F)} \sum_{\gamma=\pm \sigma_{1,z}} \int_{N_{P_0}(\A)} f(g^{-1}\delta^{-1} \gamma n \delta g) \, \d n\, \widehat{\tau}_{P_0}(H_{P_0}(\delta g)-T) \Big\} \, \d^1g .
\end{align*}
\begin{prop}\label{6p7}
Assume that $S$ contains $\Sigma_\inf$ and $z\in Z(\fO_v)$ $(\forall v\not\in S)$.
Let $f\in C_c^\inf(G(F_S)^1)$.
The integrals $J_{O^{1,z}_{\min,1}}^T(f)$, $J_{O^{1,z}_{\min,2}}^T(f)$, and $J_{O^{1,z}_{\reg}}^T(f)$ converge absolutely.
Furthermore, we have
\begin{align*}
J_{O^{1,z}_{\min,1}}^T(f)=&\vol_{M_2\cap G_{\sigma_{1,z}}} J_{M_2}^T(\sigma_{1,z},f)\\
&  + a^{G_{\sigma_{1,z}}}(S,u_1(1,0))\, J_G( \sigma_{1,z}u_1(1,0),f),
\end{align*}
\begin{align*}
J_{O^{1,z}_{\min,2}}^{T}(f)=&\vol_{s_0M_2\cap G_{\sigma_{1,z}}} J_{s_0M_2}^T(\sigma_{1,z},f) \\
&+ a^{G_{\sigma_{1,z}}}(S,u_1(0,1))\, J_G( \sigma_{1,z}u_1(0,1),f), 
\end{align*}
and
\begin{align*}
J_{O^{1,z}_{\reg}}^T(f)=& \, \vol_{M_0\cap G_{\sigma_{1,z}}} J_{M_0}^T(\sigma_{1,z},f)\\
&+a^{M_2\cap G_{\sigma_{1,z}}}(S,u_1(0,1))\, J_{M_2}^T( \sigma_{1,z}u_1(0,1),f) \\
&+a^{s_0M_2\cap G_{\sigma_{1,z}}}(S,u_1(1,0))\, J_{s_0M_2}^T( \sigma_{1,z}u_1(1,0),f) \\
&+\sum_{\alpha\in F^\times/(F^\times\cap(F_S^\times)^2)} a^{ G_{\sigma_{1,z}}}(S,u_1(\alpha,1))\, J_G( \sigma_{1,z}u_1(\alpha,1),f).
\end{align*}
For formulas of the coefficients, we refer to Examples \ref{3ex1} and \ref{3ex10}.
\end{prop}
\begin{proof}
It is clear that the integrals $J_{O^{1,z}_{\min,1}}^T(f)$ and $J_{O^{1,z}_{\min,2}}^T(f)$ converge absolutely, since
\begin{align*}
& \widehat{\tau}_{P_2}(H_{P_2}(a)-T)=\begin{cases}  1 & \text{if $\log t_{13}<-2T_2$}, \\ 0 & \text{if $\log t_{13}\geq -2T_2$}, \end{cases}  \\
& \widehat{\tau}_{s_0P_2}(H_{s_0P_2}(a)-s_0T)=\begin{cases}  1 & \text{if $\log t_{24}<-2T_2$}, \\ 0 & \text{if $\log t_{24}\geq -2T_2$} \end{cases}     
\end{align*}
where $a=(t_{13}^{-1/2},t_{24}^{-1/2},t_{13}^{1/2},t_{24}^{1/2})\in (A_{M_0}^G)^+$.
Hence, the integral $J^T_{O^{1,z}_\reg}(f)$ also converges absolutely, because the integral $J_{\fo_{1,z}}^T(f)$ is absolutely convergent.
The second assertion follows from the $T$-dependent version of~\cite[Theorem~8.1]{Arthur4} given in Section~\ref{desc}.
\end{proof}
For the other mixed elements of $\GSp(2,F)$ and $\Sp(2,F)$ (cf. Section \ref{5s3}), we can easily get results similar to Proposition \ref{6p7} by the same argument as in the proof of Proposition \ref{6p7} and the results in Section \ref{3s}.
We omit the details.

\section{The geometric side of the trace formula for $\Sp(2)$}\label{7s}

Throughout this section, we set $G=\Sp(2)$.
We use the same notation and assumptions as those in Sections \ref{5s} and \ref{6s} by restricting from $\GSp(2)$ to~$G$.
Especially, we assume Condition \ref{5c1} which gives normalizations of Haar measures on $\a_{M_0}$, $\a_{M_1}$, and $\a_{M_2}$.
We note that $\a_G=\{0\}$ and $G(\A)^1=G(\A)$.

For $z\in Z(F)=\{\pm I_4\}$, the $\cO$-equivalence class $\fo_z$ containing $z$ is divided into the five classes
\[
\fo_z=O_{\tri,z}\cup O_{\min,z} \cup O_{\sub,z}\cup O_{\sub,z}' \cup O_{\reg,z},  
\]
where $ O_{\tri,z}=\{ z \}$,
\[
O_{\min,z}=\bigcup_{\alpha\in F^\times/(F^\times)^2}\{z\, n_\min(\alpha)\}_G, \quad O_{\sub,z}=\bigcup_{x\in V^\st(F)} \{z \, n_\sub(x) \}_G,
\]
\[
O_{\sub,z}'= \{z \, n_\sub(\xd_1) \}_G, \quad O_{\reg,z}= \bigcup_{\alpha\in F^\times/(F^\times)^2} \{z\, n_\reg(\alpha)\}_G ,
\]
where $n_\min(\alpha)$, $n_\sub(x)$, $n_\reg(\alpha)$ were defined in (\ref{1e2}), $V^\st(F)$ was defined in Section \ref{4s1}, and $\xd_d=\diag(1,-d)$ (cf. (\ref{xd})).
Let $T\in \a_0^+$ and $f\in C_c^\inf(G(\A))$.
The definitions of $J_{O_{\reg,z}}^T(f)$, $J_{O_{\sub,z}}^T(f)$, $J_{O_{\sub,z}'}^T(f)$, $J_{O_{\min,z}}^T(f)$, and $J_{O_{\tri,z}}^T(f)$ are the same as in Section~\ref{6s}.
Obviously, $J_{O_{\tri,z}}^T(f)=\vol_G\, f(z)$ and
\[
J^T_{\mathrm{unip}}(f) = J_{O_{\tri,z}}^T(f)+ J_{O_{\min,z}}^T(f) + J_{O_{\sub,z}}^T(f) + J_{O_{\sub,z}'}^T(f) +J_{O_{\reg,z}}^T(f) .
\]

Choosing Haar measures on minimal unipotent conjugacy classes over $F_S$, we have
\[
J_G(z\, n_\min(\alpha),f)= c_S \, \int_{\alpha(F_S^\times)^2}f_{\bK_S}(z\, n_\min(x)) \, |x|_S \, \d x 
\]
where $\alpha\in F_S^\times$, $\d x$ is the Haar measure on $F_S$ defined in Section \ref{2s1}, and the constant $c_S$ was defined in Section \ref{2s2}.
\begin{thm}\label{7t1}
Assume that $S$ contains $\Sigma_\inf$.
Let $f\in C_c^\inf(G(F_S))$.
The integral $J_{O_{\min,z}}^T(f)$ converges absolutely and satisfies
\begin{align*}
&J_{O_{\min,z}}^T(f) \\
&= \frac{ \vol_{M_2}}{c_F} \int_{\A^\times/F^\times} \sum_{y\in F^\times} f_\bK (z\, n_\min(x^2 y))\,|x^2|^2 \d^\times x \\
&= \frac{ \vol_{M_2}}{2\, c_F} \sum_\chi \int_{\A^\times} f_\bK(z n_\min(x)) \, |x|^2 \, \chi(x) \, \d^\times x \\
&= \frac{ \vol_{M_2}}{2\, c_F} \sum_{\alpha\in F^\times/(F^\times\cap(F_S^\times)^2)}   \Big\{ \sum_\chi \chi_S(\alpha) \, L^S(2,\chi)  \Big\} J_G(z\, n_\min(\alpha),f) ,
\end{align*}
where $\d^\times x$ is the Haar measure on $\A^\times$ which was defined in Section \ref{2s1} and $\chi$ runs over all quadratic characters on $\A^1/F$ unramified outside~$S$.
\end{thm}
\begin{proof}
This theorem follows from the arguments in Section \ref{3s}.
\end{proof}

Let $T=T_1\alpha_1^\vee +T_2\alpha_2^\vee\in \a_0^+$.
If $S$ contains $\Sigma_\inf$, then we have the form of $J_{M_1}^T(z,f)$ as in Section \ref{6s1}.
Let $d\in F^\times$ and $d_S\in F_S^\times$.
Recall that the notations $V$, $V^\ss$, $V^\ss(F_S,d_S)$, $\varepsilon_S$, $\chi_{d,S}(x)$, and $V^\ss(F_S,d_S,\varepsilon_S)$ were defined in Sections \ref{4s1} and \ref{4s4}.
For any $y\in V^\ss(F_S,d_S,\varepsilon_S)$, by choosing a measure on the unipotent conjugacy class, we have
\[
J_G(z\, n_\sub(y),f)=  2^{|S|} c_S \int_{V^\ss(F_S,d_S,\varepsilon_S)}f_{\bK_S}(z\, n_\sub(x))  \, \d x
\]
where $\d x$ is the Haar measure on $V(F_S)$ defined in Section \ref{4s4}.
We denote the group $\GSp(2)$ by~$G^*$, because the next result can also be stated in terms of the distribution $J_{G^*}(z\, n_\sub(y),f^*)$ from Theorem~\ref{6t2} and its twisted version
\[
J_{G^*,\chi_d}(z\, n_\sub(\xd_d),f')=  2^{|S|} c_S \int_{V^\ss(F_S,d)}f^*_{\bK_S}(z\, n_\sub(x)) \, \chi_{d,S}(x) \, \d x
\]
for $d\in F^\times-(F^\times)^2$ and $f^*\in C_c^\inf(G^*(F_S)^1)$, where $\xd_d=\diag(1,-d)$ (cf. \eqref{4e1}).
We denote by $\varepsilon_v(x_v)$ the Hasse invariant of $x_v\in V^\ss(F_v)$ over $F_v$.
For $x=(x_v)_{v\in S}\in V^\ss(F_S)$, we put $\varepsilon_v(x)=\varepsilon_v(x_v)$.
Let $\sim'_S$ denote an equivalence relation in $V^\ss(F_S)$ such that $x\sim'_S y$ if and only if $\det(x^{-1}y)\in (F_S^\times)^2$ and $\varepsilon_v(x)=\varepsilon_v(y)$ $(\forall v\in S)$.
Note that every equivalence class has a representative in~$V^\ss(F)$, so that $V^\ss(F_S)/{\sim_S'} \, =V^\ss(F)/{\sim_S'}$.
\begin{thm}\label{7t2}
Assume that $S$ contains $\Sigma_\inf\cup\Sigma_2$.
Let $f\in C_c^\inf(G(F_S))$.
The integral $J_{O_{\sub,z}}^T(f)$ converges absolutely, and
\begin{multline*}
J_{O_{\sub,z}}^T(f)= \frac{\vol_{M_1}}{2}J_{M_1}^T(z,f)  + \frac{ \vol_{M_1}}{2\, c_F} \sum_{x\in V^\ss(F)/{\sim_S'}}  J_G(z \, n_\sub(x),f) \\
 \times \Big\{ \fC_F(S,-\det(x)) +\Big(\prod_{v\in S}\varepsilon_v(x) \Big)  \sum_{\ds\in\para^{\mathrm{ur}}(F,S,-\det(x))} L^S(1,\chi_d), \Big\}
\end{multline*}
where $\mathfrak C_F(S,\alpha)$ and $\para^{\mathrm{ur}}(F,S,\alpha)$ were defined in Section \ref{4s8}. We identify $M_1$ with $H'=\GL(2)$ by \eqref{M1}.
If $f$ is the restriction of a function $f^*\in C_c^\inf(G^*(F_S)^1)$ and $\Phi_z(x)=f_\bK(z \, n_\sub(x))$, then we have
\begin{align*}
 J_{O_{\sub,z}}^T(f)=& \,  Z^{\Sp(2)}(\Phi_z,3/2,T_1) \\
=& \, \frac{\vol_{M_1}}{2}J_{M_1}^T(z,f) \\
& + \frac{ \vol_{M_1}}{2\, c_F} \sum_{ x\in V^\ss(F)/{\sim_S} }\fC_F(S,-\det(x)) \, J_{G^*}(z \, n_\sub(x),f^*) \\
& + \frac{ \vol_{M_1}}{2\, c_F} \sum_{\ds\in\para^{\mathrm{ur}}(F,S)} L^S(1,\chi_d)\,   J_{G^*,\chi_d}(z\, n_\sub(\xd_d),f^*) \\
\end{align*}
where $Z^{\Sp(2)}(\Phi,s,T_1)$, and $\para^{\mathrm{ur}}(F,S)$ were defined in Section \ref{4s8} and $\sim_S$ was defined in Section \ref{6s1}.
\end{thm}
\begin{proof}
This theorem is deduced from Theorems \ref{4t22} and \ref{4t23} and the proof of Theorem \ref{6t2}.
\end{proof}

For $S\supset \Sigma_\inf$, the definition of $J_{M_2}^T(z,f)$ is the same as in Section \ref{6s1}.
\begin{thm}\label{7t3}
Assume that $S$ contains $\Sigma_\inf\cup\Sigma_2$.
Let $f\in C_c^\inf(G(F_S))$.
The integral $J_{O_{\sub,z}'}^T(f)$ converges absolutely and we have
\[ J_{O_{\sub,z}'}^T(f)= \frac{\vol_{M_2}}{2} J_{M_2}^T(z,f) + \frac{\vol_{M_1}}{2\, c_F} \frac{\frac{\d}{\d s}\zeta_F^S(s)|_{s=3} }{\zeta_F^S(3)} J_G(z \, n_\sub(\xd_1),f) . \]
\end{thm}
\begin{proof}
The proof is completely the same as Theorem \ref{6t3}.
\end{proof}

We use the notation $\nu(n_{12},n_{13},n_{14},n_{24})$ defined in \eqref{u}.
We have the forms of $J_{M_0}^T(z,f)$ and $J_{M_1}^T(z\nu(1,0,0,0),f)$ as in Section \ref{6s1}.
Let $\d n_*$ denote the Haar measure on $F_S$ which was defined in Section \ref{2s1}.
For $\alpha\in F_S^\times$, by choosing measures on unipotent conjugacy classes, we have
\begin{multline*}
J_{M_2}^T(z\nu(0,0,0,\alpha),f)= c_S \int_{F_S^{\oplus 3}} \d n_{12} \, \d n_{13} \, \d n_{14} \, \int_{\alpha(F_S^\times)^2}\d n_{24} \\
f_{\bK_S}(z\, \nu(n_{12},n_{13},n_{14},n_{24}))\, w_{M_2}(1,\nu(n_{12},n_{13},n_{14},n_{24}),T)
\end{multline*}
and
\begin{multline*}
J_G(z\nu(1,0,\alpha,\alpha),f)= \\
c_S^2 \int_{F_S^{\oplus 3}} \d n_{12} \, \d n_{13} \, \d n_{14} \, \int_{\alpha(F_S^\times)^2}\d n_{24} f_{\bK_S}(z\, \nu(n_{12},n_{13},n_{14},n_{24}))  . 
\end{multline*}
As in Theorem~\ref{7t2} we denote the group $\GSp(2)$ by $G^*$ and its Levi subgroup corresponding to $M_2$ by~$M_2^*$. For a quadratic character $\chi=\prod_{v\in\Sigma} \chi_v$ on $\A^1/F^\times$ and $f^*\in C_c^\inf(G^*(F_S)^1)$, we define the twisted versions
\begin{multline*}
J_{M_2^*,\chi}^T(z\nu(0,0,0,1),f^*)= c_S \int_{F_S^{\oplus 3}} \d n_{12} \, \d n_{13} \, \d n_{14} \, \int_{F_S^\times}\d n_{24} \\
\chi_S(n_{24})\,  f^*_{\bK_S}(z\, \nu(n_{12},n_{13},n_{14},n_{24}))\,  w_{M_2}(1,\nu(n_{12},n_{13},n_{14},n_{24}),T)
\end{multline*}
and
\begin{multline*}
 J_{G^*,\chi}(z\nu(1,0,1,1),f^*)= c_S^2 \int_{F_S^{\oplus 3}} \d n_{12} \, \d n_{13} \, \d n_{14} \, \int_{F_S^\times}\d n_{24} \\
  \chi_S(n_{24}) \,  f^*_{\bK_S}(z\, \nu(n_{12},n_{13},n_{14},n_{24}))  
\end{multline*}
of the distributions occurring in Theorem~\ref{6t4}.
\begin{thm}\label{7t4}
Assume that $S$ contains $\Sigma_\inf\cup\Sigma_2$.
Let $f\in C_c^\inf(G(F_S))$.
The integral $J_{O_{\reg,z}}^T(f)$ converges absolutely, and we have
\begin{align*}
& J_{O_{\reg,z}}^{T}(f)= \\
& \frac{\vol_{M_0}}{8} \, J_{M_0}^T(z,f)+\frac{ \vol_{M_0}}{4 \, c_F} \, \mathfrak c_F(S) \, J_{M_1}^T(z\nu(1,0,0,0),f) \\
& +\frac{ \vol_{M_0}}{4 \, c_F} \sum_{\alpha\in F^\times/(F^\times\cap(F_S^\times)^2)} \Big\{ \sum_\chi \mathfrak c_F(S,\chi) \, \chi_S(\alpha) \Big\} \, J_{M_2}^T(z\nu(0,0,0,\alpha),f) \\
&+\frac{ \vol_{M_0}}{4\, c_F^2 } \sum_{\alpha\in F^\times/(F^\times\cap(F_S^\times)^2)} J_G(z\nu(1,0,\alpha,\alpha),f) \\
& \qquad \qquad \qquad \times \sum_\chi \big\{  2\, \fc_F(S,\chi) \, \fc_F(S) +  w(\chi) \, \fc_F'(S,\chi) \, c_F^S \big\} \, \chi_S(\alpha) 
\end{align*}
where $\chi$ runs over all quadratic characters on $\A^1/F^\times$ unramified outside~$S$, we set $w(\chi)=\begin{cases} 3 & \text{if $\chi=\trep_F$,} \\ 1 &\text{if $\chi\neq \trep_F$,}   \end{cases}$ and the notations $\fc_F(S)$, $\fc_F(S,\chi)$, and $\fc_F'(S,\chi)$ were defined in Section \ref{2s2}. If $f$ is the restriction of a function $f^*\in C_c^\inf(G^*(F_S)^1)$, we also have
\begin{align*}
J_{O_{\reg,z}}^T(f)= & \frac{\vol_{M_0}}{8} \, J_{M_0}^T(z,f)+\frac{ \vol_{M_0}}{4 \, c_F} \, \mathfrak c_F(S) \, J_{M_1}^T(z\nu(1,0,0,0),f) \\
& +\frac{ \vol_{M_0}}{4 \, c_F} \sum_\chi \mathfrak c_F(S,\chi) \, J_{M_2^*,\chi}^T(z\nu(0,0,0,1),f^*) \\
& + \frac{\vol_{M_0}}{4 \, c_F^2} \sum_\chi \big\{  2\, \fc_F(S,\chi) \, \fc_F(S) +  w(\chi) \, \fc_F'(S,\chi) \, c_F^S \big\} \\
& \qquad \qquad \qquad \qquad \times J_{G^*,\chi}(z\nu(1,0,1,1),f^*) 
\end{align*}
where $\chi$ runs over all quadratic characters on $\A^1/F^\times$ unramified outside~$S$.
\end{thm}
\begin{proof}
This theorem follows from Theorems \ref{3t3} and \ref{3t4} and the proof of Theorem \ref{6t4}.
\end{proof}

Theorems \ref{7t1}, \ref{7t2}, \ref{7t3}, and \ref{7t4} are summarized as follows.
\begin{thm}\label{7t5}
Assume that $S$ contains $\Sigma_\inf\cup\Sigma_2$.
Let $f\in C_c^\inf(G(F_S))$.
The integrals $J_{O_{\min,z}}^T(f)$, $J_{O_{\sub,z}}^T(f)$, $J_{O_{\sub,z}'}^T(f)$, and $J_{O_{\reg,z}}^T(f)$ converge absolutely for any $T\in\a_0^+$.
We have
\[ (\cU_G(F))_{G,S}=\Big\{ 1 , \; n_\min(\alpha) , \; n_\sub(x), \; n_\reg(\alpha) \, \Big| \, \begin{array}{c} \alpha\in F^\times/F^\times\cap(F_S^\times)^2 , \\  x\in V^\ss(F)/{\sim_S'} \end{array} \Big\} . \]
For the coefficients in the fine expansion of $J_{\fo_z}(f)$ (cf. Section \ref{basobj}), we have
\[
a^G(S,z\,n_\min(\alpha))= \frac{ \vol_{M_2}  }{2\, c_F} \sum_\chi \chi_S(\alpha) \, L^S(2,\chi),
\]
where $\chi$ runs over all quadratic characters on $\A^1/F^\times$ unramified outside~$S$,
\begin{align*}
 a^G(S,z\,n_\sub(x))=& \, \frac{\vol_{M_1}}{2 \, c_F}  \, \fC_F(S,-\det(x))\\
& + \frac{\vol_{M_1}}{2 \, c_F} \, \Big( \prod_{v\in S}\varepsilon_v(x) \Big) \sum_{\ds\in\para^{\mathrm{ur}}(F,S,-\det(x))}  L^S(1,\chi_d) \\
& + \frac{\vol_{M_1}}{2\, c_F}  \begin{cases} \zeta_F^S(3)^{-1} \frac{\d}{\d s}\zeta_F^S(s)|_{s=3}   & \text{if $x\sim_S'\xd_1$} , \\ 0 & \text{otherwise} , \end{cases} 
\end{align*}
and
\begin{multline*}
a^G(S,z\,n_\reg(\alpha))= \\
\frac{\vol_{M_0}}{4\, c_F^2} \sum_\chi \big\{  2\, \fc_F(S,\chi) \, \fc_F(S) +  w(\chi) \, \fc_F'(S,\chi) \, c_F^S \big\} \, \chi_S(\alpha)
\end{multline*}
where $w(\chi)$ was defined in Theorem \ref{7t4} and $\chi$ runs over all quadratic characters on $\A^1/F^\times$ unramified outside~$S$.
\end{thm}

As an analogue of Theorem \ref{6t6}, we also see that the contribution of subregular unipotent elements is expressed by special values of zeta integrals at $s=0$.
\begin{thm}\label{7t6}
Fix a test function $f\in C_c^\inf(G(\A))$.
We set $\Phi_z(x)=f_\bK(z \, n_\sub(x))$ and we identify $M_1$ with $H'=\GL(2)$ by \eqref{M1}.
If we assume $\hat\Phi_z(0)=0$, then we have
\[ J_{O_{\sub,z}}^T(f)+J_{O_{\sub,z}'}^T(f)= Z^{\Sp(2)}_\mathrm{ad}(\hat\Phi_z,0) + \frac{\vol_{M_0}}{2\, c_F} \mathfrak{T}_3(\widetilde{R_0\Phi_{z,\bK}},0,T_2) \]
in the notation defined in Section~\ref{4s9}.
\end{thm}
\begin{proof}
The formula follows from Proposition \ref{4p26}, Theorem \ref{7t2}, and the proof of Theorem \ref{6t3}.
\end{proof}
In a way analogous to Theorem \ref{6t6}, this theorem is also suitable for explicit calculations of traces of Hecke operators (cf. \cite{Arakawa,IS,Shintani,Wakatsuki}).

\appendix 

\section{The group $\GL(3)$}\label{appen1}

\subsection{Setup}\label{appen11}

Throughout this appendix, we set $G=\GL(3)$.
We fix the maximal compact subgroup
\[
\bK=\prod_{v\in\Sigma}\bK_v \quad \text{where} \quad  \bK_v=\begin{cases} \mathrm{U}(3) & \text{if $v\in\Sigma_\C$}, \\ \mathrm{O}(3) & \text{if $v\in\Sigma_\R$}, \\ \GL(3,\fO_v) & \text{if $v\in\Sigma_\fin$}, \end{cases}
\]
and choose the minimal Levi subgroup $M_0$ and the minimal parabolic subgroup $P_0$ as
\[
M_0=\left\{  \begin{pmatrix}*&0&0 \\ 0&*&0 \\ 0&0&* \end{pmatrix}\in G   \right\} \quad \text{and} \quad P_0=\left\{  \begin{pmatrix}*&*&* \\ 0&*&* \\ 0&0&* \end{pmatrix}\in G   \right\} . 
\]

The rational homomorphism $\chi_j : M_0 \to \GL(1)$ is defined by
\[
\chi_j( \diag(a_1,a_2,a_3)) = a_j , \quad 1\leq j\leq 3  . 
\]
Then $\{\chi_1,\chi_2,\chi_3\}$ is a basis of the free abelian group $X(M_0)_F$.
An element $e_j\in \a_0$ is defined by $e_j(  \chi_1^{k_1}\chi_2^{k_2}\chi_3^{k_3}  ) = k_j$.
Then, we see that
\[
\a_0^*=\R\chi_1\oplus\R\chi_2\oplus\R\chi_3 \quad \text{and} \quad \a_0=\R e_1\oplus \R e_2 \oplus \R e_3  
\]
and we have
\[ \Delta_0=\{\alpha_1 \, , \; \alpha_2 \} \quad (\alpha_1=\chi_1-\chi_2 \, , \; \alpha_2=\chi_2-\chi_3).  \]

We set
\[
P_1=\left\{  \begin{pmatrix}*&*&* \\ *&*&* \\ 0&0&* \end{pmatrix}\in G   \right\} \quad \text{and} \quad P_2=\left\{  \begin{pmatrix}*&*&* \\ 0&*&* \\ 0&*&* \end{pmatrix}\in G   \right\}  .
\]
Then, we have $\{ P\in\P\, |\, P\supset P_0  \}=\{ P_0, \, P_1, \, P_2 , \, G \}$.
Let $W_0$ denote the symmetric group of degree three.
The group $W_0$ acts on $\a_0$ as $se_j=e_{s(j)}$ for $s\in W_0$.
Let $w_s$ denote the matrix representing the action of $s\in W_0$.
For $s\in W_0$ and a subgroup $H$ of $G$, we set $s H=w_s Hw_s^{-1}$. 
Using this action, we have
\[
\L=\{ M_0,\, M',\, (13)M', \, (23)M',\, G   \} , \quad  \L(M')=\{ M', \, G \}
\]
where we set $M'=M_{P_1}$.
Note that $M_{P_2}=(13)M'$.
Furthermore, we have
\[
\P=\{ s P_0 \; | \; s \in W_0   \}, \quad \P(M')=\{ P_1 , \; (13)P_2  \},
\]
\[
\P((13)M')=\{ P_2 , \; (13)P_1  \}, \quad \P((23)M')=\{ (23)P_1 , \; (12)P_2\}.
\]

\subsection{Weight factors}\label{appen13}
From now on, we fix Haar measures on $\a_0^G$ and $\a_{M'}^G$.
\begin{cond}\label{a1c1}
We choose a Haar measure on $\a_0^G\cong A_{M_0}^+/A_G^+$ as $\d r_1  \, \d r_2$ for $r_1(e_1-e_2)+ r_2(e_1-e_2)\in \a_0^G$, where $\d r_1$ and $\d r_2$ are the Lebesgue measure on $\R$.
A Haar measure on $\a_{M'}^G\cong A_{M'}^+/A_G^+$ is fixed by the Lebesgue measure $\d r$ on $\R$ for $r( \frac{e_1+e_2}{2}- e_3)\in\a_{M'}^G$.
\end{cond}
Under Condition \ref{a1c1}, we have
\[
\vol(\a^G_{M_0}/\Z(\Delta^\vee_{P_0}))=\vol(\a^G_{M'}/\Z(\Delta^\vee_{P_1}))=1,
\]
where $\Delta^\vee_P$ is the basis of $\a_P^G$ dual to $\widehat\Delta_P$.

We set
\[
\alpha_1^\vee=e_1-e_2, \quad  \alpha_2^\vee=e_2-e_3\in\a_0^G,
\]
\[
\varpi_1=\frac{2}{3}\alpha_1+\frac{1}{3}\alpha_2,\quad  \varpi_1=\frac{1}{3}\alpha_1+\frac{2}{3}\alpha_2\in(\a_0^G)^* .
\]
Then, we have $\Delta_0^\vee=\{\alpha_1^\vee\, , \; \alpha_2^\vee\}$ and $\widehat\Delta_0=\{\varpi_1 \, , \; \varpi_2\}$.
We set
\[
T=T_1\alpha_1^\vee +T_2\alpha_2^\vee\in \a_0,
\]
where $T_1$, $T_2\in\R$, and
\[
u(x_{12},x_{13},x_{23})=\begin{pmatrix}1&x_{12}&x_{13} \\ 0&1&x_{23} \\ 0&0&1 \end{pmatrix}.
\]
\begin{lem}\label{al1}
For the element $n=u(n_{12},n_{13},n_{23})\in N_{P_0}(\A)$ we have
\begin{multline*}
v_{sP_0}(\lambda,n,T)= \\
\begin{cases}
e^{\lambda_1T_1+\lambda_2T_2}  & \text{if $s=1$,} \\
\| (1,n_{12}) \|^{-\lambda_1} \,e^{\lambda_1(T_2-T_1)+\lambda_2T_2}  & \text{if $s=(12)$,} \\
\|(1,n_{23})\|^{-\lambda_2} \,e^{\lambda_1T_1+\lambda_2(T_1-T_2)}  & \text{if $s=(23)$,} \\
\|(1,n_{12})\|^{\lambda_2} \,\|(1,n_{12},n_{13})\|^{-\lambda_1-\lambda_2} \, e^{-\lambda_1T_2+\lambda_2(T_1-T_2)}  & \text{if $s=(123)$,} \\
\|(1,n_{23})\|^{\lambda_1} \,\|(1,n_{13}-n_{12}n_{23},n_{23})\|^{-\lambda_1-\lambda_2} &  \\
\quad \times  e^{\lambda_1(T_2-T_1)-\lambda_2T_1}  & \text{if $s=(132)$,} \\
\|(1,n_{12},n_{13})\|^{-\lambda_1} \,\|(1,n_{13}-n_{12}n_{23},n_{23})\|^{-\lambda_2} &  \\
\quad \times e^{-\lambda_1T_2-\lambda_2T_1}  & \text{if $s=(13)$} 
\end{cases}
\end{multline*}
where $\lambda=\lambda_1\varpi_1+\lambda_2\varpi_2$.
\end{lem}
\begin{proof}
This follows from an argument similar to the proof of Lemma \ref{5l3}.
\end{proof}

\begin{prop}\label{ap2}
Assume Condition \ref{a1c1}.
For $\nu=u(\nu_{12},\nu_{13},\nu_{23}) \in N_{P_0}(F_S)$ we have
\begin{align*}
w_{M_0}(1,\nu,T)  = & \frac{1}{2}\{  (\log|\nu_{12}|_S)^2+(\log|\nu_{23}|_S)^2+4(\log|\nu_{12}|_S)(\log|\nu_{23}|_S)  \} \\
& + 3T_2\log|\nu_{12}|_S + 3T_1\log|\nu_{23}|_S - \frac{3}{2}T_1^2 - \frac{3}{2}T_2^2 + 6 T_1T_2.
\end{align*}
\end{prop}
\begin{proof}
It follows from Lemma \ref{al1} that
\[  w_{sP_0}(\lambda,1,\nu,T)=\begin{cases}
e^{\lambda_1T_1+\lambda_2T_2}  & \text{if $s=1$,} \\
|\nu_{12}|_S^{-\lambda_1} \,e^{\lambda_1(T_2-T_1)+\lambda_2T_2}  & \text{if $s=(12)$,} \\
|\nu_{23}|_S^{-\lambda_2} \,e^{\lambda_1T_1+\lambda_2(T_1-T_2)}  & \text{if $s=(23)$,} \\
|\nu_{12}\nu_{23}|_S^{-\lambda_1}|\nu_{23}|_S^{-\lambda_2} \, e^{-\lambda_1T_2+\lambda_2(T_1-T_2)}  & \text{if $s=(123)$,} \\
|\nu_{12}|_S^{-\lambda_1}|\nu_{12}\nu_{23}|_S^{-\lambda_2} \, e^{\lambda_1(T_2-T_1)-\lambda_2T_1}  & \text{if $s=(132)$,} \\
|\nu_{12}\nu_{23}|_S^{-\lambda_1}|\nu_{12}\nu_{23}|_S^{-\lambda_2} \, e^{-\lambda_1T_2-\lambda_2T_1}  & \text{if $s=(13)$} 
\end{cases}    \]
where $\lambda=\lambda_1\varpi_1+\lambda_2\varpi_2$.
Hence, this proposition follows.
\end{proof}
Using Lemma \ref{al1} we can easily compute the following.
\begin{prop}\label{ap3}
Assume Condition \ref{a1c1}.
For $u=u(u_{12},0,0)$ and $\nu=u(0,\nu_{13},\nu_{23}) \in N_{P_1}(F_S)$ we have
\begin{align*}
& w_{M'}(1,\nu,T)=\log \|(\nu_{13},\nu_{23})\|_S+T_1+T_2 \, , \\ 
& w_{M'}(1,u\nu,T)= \log|\nu_{23}u_{12}|_S+T_1+T_2 . 
\end{align*}
\end{prop}

\subsection{Unipotent contribution}\label{appen14}
Throughout this subsection, we assume Condition \ref{a1c1} and $S\supset\Sigma_\inf$.
Let $Z$ denote the center of $G$.
For each $z\in Z(F)$, we denote by $\fo_z$ the $\cO$-equivalence class containing $z$ in $G(F)$.
The notation $\{x\}_G$ means the $G(F)$-conjugacy class of $x\in G(F)$.
The set $\fo_z$ is divided into the three classes
\[ \fo_z=O_{\tri,z}\cup O_{\min,z}  \cup O_{\reg,z}  \]
where $O_{\tri,z}=\{ z \}$, $O_{\min,z}=\{z\, u(0,1,0)\}_G$, and $O_{\reg,z}=\{z\, u(1,0,1)\}_G$.
Let $P\in\F$, $P\supset P_0$, and $f\in C_c^\inf(G(\A)^1)$.
The function $K_{P,O}(g,h)$ is defined by (\ref{6ekk}) for each subset $O$ in $M_P(F)$.
For $T\in \a_0^+$, we have the terms on the geometric side
\[  J_{O_{\tri,z}}^T(f)=\int_{G(F)\bsl G(\A)^1}   K_{G,O_{\tri,z}}(g,g)  \, \d^1 g , \]
\begin{align*}
 J_{O_{\min,z}}^T(f)=&\int_{G(F)\bsl G(\A)^1} \Big\{  K_{G,O_{\min,z}}(g,g) \\
& \qquad -\sum_{\delta\in P_1(F)\bsl G(F)}K_{P_1,\{z \}}(\delta g,\delta g) \, \widehat{\tau}_{P_1}(H_{P_1}(\delta g)-T) \\
& \qquad -\sum_{\delta\in P_2(F)\bsl G(F)}K_{P_2,\{z \}}(\delta g,\delta g) \, \widehat{\tau}_{P_2}(H_{P_2}(\delta g)-T) \Big\} \, \d^1 g ,
\end{align*}
and
\begin{align*}
J_{O_{\reg,z}}^T(f)=&\int_{G(F)\bsl G(\A)^1} \Big\{  K_{G,O_{\reg,z}}(g,g) \\
&\quad  -\sum_{\delta\in P_1(F)\bsl G(F)}K_{P_1,\{z u(1,0,0) \}_{M_{P_1}}}(\delta g,\delta g) \, \widehat{\tau}_{P_1}(H_{P_1}(\delta g)-T)       \\
&\quad  -\sum_{\delta\in P_2(F)\bsl G(F)}K_{P_2,\{z u(0,0,1)\}_{M_{P_2}}}(\delta g,\delta g) \, \widehat{\tau}_{P_2}(H_{P_2}(\delta g)-T)       \\
&\quad  +  \sum_{\delta\in P_0(F)\bsl G(F)}K_{P_0,\{z \}}(\delta g,\delta g) \, \widehat{\tau}_{P_0}(H_{P_0}(\delta g)-T)  \Big\} \d^1 g .
\end{align*}

It is clear that $J_{O_{\tri,z}}^T(f)=\vol_G \, f(z)$.
The function $f_{\bK_S}(g)$ is defined as
\[   f_{\bK_S}(g)=\int_{\bK_S} f(k_S^{-1}g k_S)\, \d k_S .  \]
Now, the weighted orbital integrals $J_{M'}^T(z,f)$ is
\[ J_{M'}^T(z,f)=  \int_{N_{P_1}(F_S)} f_{\bK_S}(z\, n) \, w_{M'}(1,n,T) \, \d n    \]
(cf. Proposition \ref{ap3}).
Choosing a Haar measure on the $G(F_S)$-conjugacy class of $u(0,1,0)$ we have
\[ J_G(z\, u(0,1,0),f)=  \int_{N_{P_1}(F_S)} f_{\bK_S}(z\, n) \, \d n .  \]
\begin{thm}\label{at4}
Assume $z\in Z(\fO_v)$ $(\forall v\not\in S)$.
Let $f\in C_c^\inf(G(F_S)^1)$.
The integral $J_{O_{\min,z}}^T(f)$ converges absolutely and satisfies
\[ J_{O_{\min,z}}^T(f) = \vol_{M'}\, J_{M'}^T(z,f) + \vol_{M'}\, \frac{ \frac{\d}{\d s}\zeta_F^S(s)|_{s=2} }{\zeta_F^S(2)} \, J_G(z\, u(0,1,0),f) . \]
\end{thm}
\begin{proof}
This theorem can be proved by an argument similar to the proof of Theorem \ref{6t3}.
\end{proof}

The weighted orbital integrals $J_{M_0}(z,f)$ is
\[ J_{M_0}^T(z,f)=  \int_{N_{P_0}(F_S)}  f_{\bK_S}(z\, n) \, w_{M_0}(1,n,T) \, \d n \]
(cf. Proposition \ref{ap2}).
Choosing Haar measures on unipotent conjugacy classes the weighted orbital integrals $J_{M'}(z\, u(1,0,0),f)$ and $J_G(z\, u(1,0,1),f)$ are 
\[ J_{M'}^T(z\, u(1,0,0),f)= c_S\, \int_{N_{P_0}(F_S)} f_{\bK_S}(z\, n) \, w_{M'}(1,n,T) \, \d n \]
and
\[ J_G(z\, u(0,1,0),f)= c_S^2 \int_{N_{P_0}(F_S)}  f_{\bK_S}(z\, n) \, \d n  \]
where the constant $c_S$ was given in Section \ref{2s2}.
\begin{thm}\label{at5}
Assume $z\in Z(\fO_v)$ $(\forall v\not\in S)$.
Let $f\in C_c^\inf(G(F_S)^1)$.
The integral $J_{O_{\reg,z}}^T(f)$ converges absolutely and satisfies
\begin{align*}
J_{O_{\reg,z}}^T(f)=& \frac{\vol_{M_0}}{6}\, J_{M_0}^T(z,f) + \frac{\vol_{M_0}}{2 c_F}\, \fc_F(S) \, J_{M'}^T(z\, u(1,0,0),f) \\
& + \frac{\vol_{M_0}}{3 \, c_F^2} \Big\{ (\mathfrak c_F(S))^2 +  \fc_F'(S)\, c_F^S \Big\}\, J_G(z\, u(1,0,1),f) .
\end{align*}
\end{thm}
\begin{proof}
This theorem follows from an argument similar to the proof of Theorem \ref{6t4}.
\end{proof}

It is clear that $(\cU_G(F))_{G,S}=\{ \, I_3 \, , \; u(0,1,0) \, , \; u(1,0,1) \}$.
Hence, it follows that
\[   a^{\GL(3)}(S,z\,u(0,1,0))= \vol_{M'} \frac{ \frac{\d}{\d s}\zeta_F^S(s)|_{s=2} }{\zeta_F^S(2)}  \]
and
\[ a^{\GL(3)}(S,z\,u(1,0,1))= \frac{\vol_{M_0}}{3 \, c_F^2} \Big\{ (\mathfrak c_F(S))^2 +  \fc_F'(S)\, c_F^S \Big\} \]
if $z\in Z(\fO_v)$ $(\forall v\not\in S)$.

\section{The group $\SL(3)$}\label{appen2}

Throughout this section, we set $G=\SL(3)$.
We will use the same notation as in Appendix \ref{appen1} by restricting from $\GL(3)$ to $G$.
We also assume Condition \ref{a1c1} and $S\supset\Sigma_\inf$.
Let $z\in Z(F)=\{\pm I_3 \}$.
The set $\fo_z$ is divided into the three classes
\[
\fo_z=O_{\tri,z}\cup O_{\min,z}  \cup O_{\reg,z},
\]
where $O_{\tri,z}=\{ z \}$,
\[
O_{\min,z}=\{z\, u(0,1,0)\}_G, \quad \text{and} \quad O_{\reg,z}= \bigcup_{\alpha\in F^\times/(F^\times)^3} \{z\, u(1,0,\alpha) \}_G.
\]
Now, the contribution $J_{\fo_z}^T(f)$ is the sum of $J_{O_{\tri,z}}^T(f)$, $J_{O_{\min,z}}^T(f)$, and $J_{O_{\reg,z}}^T(f)$.

It is clear that $J_{O_{\tri,z}}^T(f)=\vol_G \, f(z)$.
The weighted orbital integrals $J_{M'}^T(z,f)$ and $J_G(z\, u(0,1,0),f)$ are the same as those of Appendix \ref{appen14}.
\begin{thm}\label{bt1}
Let $f\in C_c^\inf(G(F_S))$.
The integral $J_{O_{\min,z}}^T(f)$ converges absolutely and satisfies
\[ J_{O_{\min,z}}^T(f) = \vol_{M'}\, J_{M'}^T(z,f) + \vol_{M'}\, \frac{ \frac{\d}{\d s}\zeta_F^S(s)|_{s=2} }{\zeta_F^S(2)} \, J_G(z\, u(0,1,0),f) . \]
\end{thm}
\begin{proof}
We can prove this by an argument similar to the proof of Theorem \ref{6t3}.
\end{proof}

The weighted orbital integrals $J_{M_0}^T(z,f)$ and $J_{M'}^T(z\, u(1,0,0),f)$ are the same as Appendix \ref{appen14}.
Let $\d x_*$ denote the Haar measure on $F_S$ defined in Section \ref{2s1}.
For an element $\alpha\in F_S$ we set
\begin{multline*}
J_G(z\,u(1,0,\alpha),f)= \\
c_S^2 \int_{(x_{12},x_{13},x_{23})\in D_S} f_{\bK_S}(z\, u(x_{12},x_{13},x_{23}))  \, \d x_{12} \, \d x_{13}\, \d x_{23},
\end{multline*}
where $D_S=\{ (x_{12},x_{13},x_{23})\in F_S\oplus F_S\oplus F_S \, | \, x_{23}\in x_{12}\alpha(F_S^\times)^3\}$. We denote the group $\GL(3)$ by $G^*$, and for a cubic character $\chi=\prod_v \chi_v$ on $\A^1/F^\times$ we set
\begin{multline*}
J_{G^*,\chi}(z\,u(1,0,1),f^*)= \\
c_S^2 \int_{F_S^{\oplus 3}} f^*_{\bK_S}(z\, u(x_{12},x_{13},x_{23})) \, \chi_S(x_{12}^2x_{23}) \, \d x_{12} \,  \d x_{13} \,  \d x_{23}, 
\end{multline*}
where $\chi_S=\prod_{v\in S}\chi_v$.
\begin{thm}\label{bt2}
Let $f\in C_c^\inf(G(F_S))$.
The integral $J_{O_{\reg,z}}^T(f)$ converges absolutely and satisfies
\begin{align*}
J_{O_{\reg,z}}^{T}(f)=& \frac{\vol_{M_0}}{6}\, J_{M_0}^T(z,f) + \frac{\vol_{M_0}}{2 c_F}\, \fc_F(S) \, J_{M'}^T(z\, u(1,0,0),f) \\
& + \frac{\vol_{M_0}}{3 \, c_F^2} \sum_{\alpha\in F^\times/(F^\times\cap(F_S^\times)^3)} J_G(z\, u(1,0,\alpha),f) \\
& \qquad \qquad \times  \Big\{ \fc_F'(S)\, c_F^S + \sum_\chi \chi_S(\alpha) \, \fc_F(S,\chi)\, \fc_F(S,\chi^{-1})  \Big\} 
\end{align*}
where $\chi$ runs over all non-trivial cubic characters on $\A^1/F^\times$ unramified outside~$S$.
If $f$ is the restriction of a function $f^*\in C_c^\inf(G^*(F_S)^1)$, then
\begin{align*}
J_{O_{\reg,z}}^T(f)=& \frac{\vol_{M_0}}{6}\, J_{M_0}^T(z,f) + \frac{\vol_{M_0}}{2 c_F}\, \fc_F(S) \, J_{M'}^T(z\, u(1,0,0),f) \\
& + \frac{\vol_{M_0}}{3 \, c_F^2}  \Big\{ (\fc_F(S))^2 +  \fc_F'(S)\, c_F^S \Big\}\,  J_{G^*}(z\, u(1,0,1),f^*) \\
& + \frac{\vol_{M_0}}{3 \, c_F^2} \sum_{\chi\neq\trep_F}  \fc_F(S,\chi)\, \fc_F(S,\chi^{-1}) \,  J_{G^*,\chi}(z\, u(1,0,1),f^*),
\end{align*}
where $\chi$ runs over all cubic characters on $\A^1/F^\times$ unramified outside~$S$.
\end{thm}
\begin{proof}
This theorem follows from Theorems \ref{3t3} and \ref{3t4} and an argument similar to the proof of Theorem \ref{6t4}.
\end{proof}
It is clear that
\[
(\cU_G(F))_{G,S}=\{ \, I_3 \, , \; u(0,1,0) \, , \; u(1,0,\alpha)\, | \, \alpha\in F^\times/(F^\times\cap(F_S^\times)^3) \}.
\]
It follows from Theorems \ref{bt1} and \ref{bt2} that
\[
a^{\SL(3)}(S,z\,u(0,1,0))= \vol_{M'} \frac{ \frac{\d}{\d s}\zeta_F^S(s)|_{s=2} }{\zeta_F^S(2)}
\]
and
\begin{multline*}
a^{\SL(3)}(S,z\,u(1,0,\alpha))= \\
\frac{\vol_{M_0}}{3 \, c_F^2} \Big\{   \fc_F'(S)\, c_F^S + \sum_\chi \chi_S(\alpha) \, \fc_F(S,\chi)\, \fc_F(S,\chi^{-1}) \Big\} \quad (\alpha\in F^\times)
\end{multline*}
where $\chi$ runs over all cubic characters on $\A^1/F^\times$ unramified outside~$S$.
\bigskip

\noindent
{\bf Acknowledgement}.
The second author thanks Prof. Kaoru Hiraga, Prof. Tamotsu Ikeda, Prof. Takuya Konno, and Prof. Steven Spallone for their advice and discussions on endoscopy.
The authors were partially supported by the German Research Council~(DFG) within the CRC~701.
The second author was partially supported by JSPS Grant-in-Aid for Scientific Research No.~20740007.

\end{document}